\documentclass[11pt]{article}

\usepackage[letterpaper,margin=1in]{geometry}
\usepackage{amsmath, amsfonts, amssymb, amsthm, amsxtra}
\usepackage{mathrsfs, mathtools}
\usepackage{bbm}
\usepackage{bm}

\usepackage[shortlabels]{enumitem}
\usepackage[dvipsnames]{xcolor}
\usepackage[ruled]{algorithm2e}
\usepackage[colorlinks=true,linkcolor=blue,citecolor=blue,urlcolor=blue]{hyperref}
\usepackage{comment}

\usepackage{booktabs}
\usepackage{multirow}
\usepackage{tikz}
\usepackage{tkz-graph}
\usetikzlibrary{arrows.meta,positioning,decorations.pathmorphing,calc}
\usepackage{subcaption}

\usepackage{wrapfig}

\usepackage{cypriot}

\newtheorem{theorem}{Theorem}[section]

\newtheorem{lemma}[theorem]{Lemma}
\newtheorem{corollary}[theorem]{Corollary}

\newtheorem{remark}[theorem]{Remark}
\newtheorem{proposition}[theorem]{Proposition}

\newcommand{\cB}{\mathcal{B}}

\newcommand{\cE}{\mathcal{E}}

\newcommand{\cG}{\mathcal{G}}
\newcommand{\cH}{\mathcal{H}}

\newcommand{\cN}{\mathcal{N}}

\newcommand{\cZ}{\mathcal{Z}}

\newcommand{\EE}{\mathbb{E}}

\newcommand{\PP}{\mathbb{P}}

\newcommand{\RR}{\mathbb{R}}

\newcommand{\ZZ}{\mathbb{Z}}

\newcommand{\bi}{\bm{i}}

\newcommand{\bS}{\bm{S}}

\newcommand{\ib}{\mathbf{i}}

\newcommand{\xb}{\mathbf{x}}

\newcommand{\zb}{\mathbf{z}}

\newcommand{\Gb}{\mathbf{G}}

\newcommand{\Wb}{\mathbf{W}}
\newcommand{\Xb}{\mathbf{X}}
\newcommand{\Yb}{\mathbf{Y}}
\newcommand{\Zb}{\mathbf{Z}}

\newcommand{\R}{\mathbbm{R}}

\newcommand{\E}{\operatorname{\mathbbm{E}}}

\newcommand{\1}{\boldsymbol{1}}

\makeatletter
\DeclareRobustCommand\Equiv{\mathrel{%
  \mathchoice
    {\Equiv@\textfont\displaystyle{.45}}
    {\Equiv@\textfont\textstyle{.45}}
    {\Equiv@\scriptfont\scriptstyle{.5}}
    {\Equiv@\scriptscriptfont\scriptscriptstyle{.55}}
}}
\newcommand{\Equiv@}[3]{%
  \rlap{\raisebox{#3\fontdimen5#12}{$\m@th#2 = $}}%
  \raisebox{-#3\fontdimen5#12}{$\m@th#2 = $}%
}
\makeatother

\newcommand{\tr}{\operatorname{\mathsf{Tr}}}

\newcommand{\ud}{\mathrm{d}}

\newcommand{\argmax}{\operatorname*{argmax}}

\newcommand{\Cov}{\operatorname{\mathsf{Cov}}}

\newcommand{\Var}{\operatorname{\mathsf{Var}}}

\newcommand{\MMSE}{\operatorname{\mathsf{MMSE}}}
\newcommand{\YMMSE}{\operatorname{\mathsf{Y\text{-}MMSE}}}

\newcommand{\bitm}{\begin{itemize}[leftmargin=*]}
\newcommand{\eitm}{\end{itemize}}

\newcommand{\benm}{\begin{enumerate}[leftmargin=*]}
\newcommand{\eenm}{\end{enumerate}}

\definecolor{forestgreen}{rgb}{0.13, 0.55, 0.13}

\def\abs#1{\left| #1 \right|}
\newcommand{\norm}[1]{\left\lVert#1\right\rVert}

\newcommand{\ceil}[1]{\left\lceil\, {#1}\,\right\rceil}
\newcommand{\inparen}[1]{\left(#1\right)}             
\newcommand{\inbraces}[1]{\left\{#1\right\}}           
\newcommand{\insquare}[1]{\left[#1\right]}             
\newcommand{\inangle}[1]{\left\langle#1\right\rangle} 

\makeatletter
\newcommand{\leqnomode}{\tagsleft@true\let\veqno\@@leqno}
\newcommand{\reqnomode}{\tagsleft@false\let\veqno\@@eqno}
\makeatother

\newcommand{\postGeomMPartition}{\ensuremath{\cZ_{n,d}}}

\title{Geometric planted matchings in high dimensions: The power of multiple views}

\author{Timothy L.\ H.\ Wee\thanks{School of Mathematics, Georgia Institute
of Technology, timothy.wee@gatech.edu} \and
Kaylee Y.\ Yang\thanks{Department of Statistics and Data Science,
Yale University, yingxi.yang@yale.edu} \and
Zhou Fan\thanks{Department of Statistics and Data Science, Yale
University, zhou.fan@yale.edu} \and
Cheng Mao \thanks{School of Mathematics, Georgia Institute
of Technology, cheng.mao@math.gatech.edu}
}

\date{}

\begin{document}

\maketitle

\begin{abstract}
We study the problem of recovering the correspondence between a collection of $n$ points in $\mathbb{R}^d$ and a noisy, permuted version of those points. In the high-dimensional regime $d=\omega(\log n)$, under a Gaussian model with noise variance $\sigma^2=d/(b\log n)$, prior work identifies $b=2$ as the threshold for almost exact recovery. We prove that this threshold is all-or-nothing: for every fixed $b<2$, no estimator recovers a positive fraction of the matching, and even estimating the matched point cloud in Euclidean distance is asymptotically no better than ignoring the correspondence. On the other hand, we consider a multi-view generalization of the problem where $K$ noisy, independently permuted copies of the same latent point cloud are observed. Here we show that a simple polynomial-time procedure recovers all relative matchings up to $o(n)$ errors whenever $b>K/(K-1)$. Thus multiple views can break the impossibility barrier $b=2$ for the original matching problem: in particular, for $3/2 < b < 2$, the two-view model has no nontrivial recovery, but a third view makes all latent correspondences efficiently recoverable.
\end{abstract}

\tableofcontents

\section{Introduction}

Let $X_1,\dots,X_n$ be a point cloud consisting of independent standard Gaussian vectors in $\RR^d$. In the \emph{geometric planted matching} model, one observes the original point cloud as well as a noisy permuted copy
\begin{equation}\label{eq:intro_two_view}
    Y_i=X_{\Pi_*(i)}+\sigma Z_i,\qquad i=1,\dots,n,
\end{equation}
where the noise $Z_1,\dots,Z_n$ are independent standard Gaussian vectors in $\RR^d$, $\sigma$ is the noise scaling, and $\Pi_*$ is an unknown uniformly random permutation, viewed as a bijection on $[n]$. The goal is to recover $\Pi_*$. 

This model captures the fundamental task of identifying unknown correspondences between unordered objects from noisy features. Such matching problems arise across science and engineering, including in single-cell data integration and multi-omics alignment \cite{haghverdi2018batch,demetci2022scot,chen2022onewayLowRankMatching}, particle tracking \cite{chertkov2010inference}, image matching \cite{ma2021image}, record linkage \cite{sayers2016probabilistic}, and network alignment \cite{bayati2009algorithms}. Across these settings, the same statistical and algorithmic question recurs: what level of noise can the observations tolerate so that the latent correspondence remains recoverable, and when can this be done efficiently?

Our results and techniques pertain to the high-dimensional regime where $d=\omega(\log n)$ and where the noise scales as
\[
    \sigma^2=\frac{d}{b\log n},
\]
where $b > 0$ is a fixed  parameter.
This is the critical scaling in which sharp recovery transitions occur at constant values of $b$.

Much of the existing theory for geometric planted matching has focused on exact or \emph{almost exact recovery}, where one must identify all but $o(n)$ of the matched pairs. Prior results \cite{kunisky2022strong,dai2023gaussian,wang2022random} place the almost exact recovery threshold at $b=2$. However, this still leaves open the nature of the recovery landscape below the threshold for this canonical model. In particular, for $b < 2$, it is plausible that an estimator might still recover a positive fraction of the matching, or might localize the matched point cloud well in Euclidean distance, producing a ``something'' phase that straddles between almost exact recovery and total failure.

In this paper we resolve this uncertainty. Our results rule out the presence of a ``something'' phase for $b < 2$. Instead we show that the model is in a strong ``nothing'' phase: the posterior assigns exponentially small mass to the set of permutations that agree with the truth on any fixed positive fraction of indices. In particular, no estimator can achieve nonvanishing overlap with the true permutation $\Pi_*$. Furthermore, this failure is not only combinatorial: even if one only asks to reconstruct the matched point cloud $(X_{\Pi_*(1)},\dots,X_{\Pi_*(n)})$ in Euclidean distance, the observations are asymptotically no more informative than in the case where the correspondence is completely absent.  Thus, for $b<2$, the data support neither weak recovery of the permutation nor geometrically accurate matching of the points.

On the other hand, multiple views reveal a different side of the same matching problem. Informally, instead of one noisy copy, we observe $K$ independently relabeled noisy views of the same point cloud $X_1,\dots,X_n$,
\begin{equation}\label{eq:intro_k_view}
    Y_i^{(a)}=X_{\bar{\Pi}_*^{(a)}(i)}+\text{independent Gaussian noise},
    \qquad
    i=1,\dots,n,\quad a=0,\dots,K-1,
\end{equation}
where the $\bar{\Pi}_*^{(a)}$'s are independent uniformly random permutations. The original point cloud itself is not observed, so the absolute labels of the latent points are unidentifiable; the task is instead to determine which rows in different views correspond to the same latent point. For $K=2$, this coincides with \eqref{eq:intro_two_view} after choosing one view as the reference and passing to the relative permutation. This can be viewed as a stylized multi-view registration problem, abstracting settings such as multi-sensor point-cloud fusion \cite{pomerleau2015review}, record linkage across multiple databases \cite{christen2012data}, single-particle cryo-EM \cite{bendory2020single}, or matching high-dimensional embeddings of the same entities across several modalities \cite{demetci2022scot,chen2022onewayLowRankMatching}.

The effect of multiple views is not obvious a priori. Additional observations of the same latent point cloud are obtained, but more latent matchings have to be estimated. At the same time, multiple views create consistency checks that are absent from the two-view problem: a proposed tuple of matches must agree across many pairwise comparisons. Our results show that the balance comes out favorably: for fixed $K$, and with an analogous high-dimensional scaling indexed by $b$, a polynomial-time algorithm achieves almost exact recovery for $b>K/(K-1)$. Already for $K=3$, this gives a clean illustration of the gain from multiple views: for $3/2<b<2$, two views give us nothing, but an additional third view breaks this barrier and permits efficient recovery of \emph{all} relative matchings.

As by-products of our arguments, we also obtain analogous negative results for a ``$d=\infty$'' Gaussian weighted matching model of standalone interest, where geometric correlations are replaced by independent Gaussian noise. In addition, by data processing, our negative results transfers over to geometric \emph{graph} matching models \cite{wang2022random} that are stochastically degraded variants of \eqref{eq:intro_two_view}.

\subsection{Related work}

\paragraph{Planted matchings.}

The overarching goal in these problems is to uncover a hidden correspondence between two collections of objects based on noisy observations. A mathematical formulation arose in the context of particle tracking \cite{chertkov2010inference}, where geometric observations determine the likelihood of pairing particles across images. A parallel, non-geometric line studied independent-edge planted matching models \cite{semerjian2020recovery,moharrami2021planted,ding2023planted,wee2025cluster,addarioberry2026statistical}. Closer to the geometric motivation, \cite{kunisky2022strong,dai2023gaussian} studied Gaussian point-cloud matching, with related work going beyond Gaussianity \cite{schwengber2024geometric}, models with low-rank signals 
\cite{chen2022onewayLowRankMatching}, and
uncertainty quantification in match probabilities \cite{fan2026bayesian}. A multiple view setting in the context of non-geometric Gaussian graph alignment was studied in \cite{even2025statistical}.

\paragraph{High-dimensional geometric matching.}

The geometric matching model \eqref{eq:intro_two_view} in the high-dimensional setting of $d = \omega(\log n)$ and $\sigma^2=d/(b\log n)$ has been studied by \cite{kunisky2022strong,dai2023gaussian, wang2022random}. A major focus in \cite{kunisky2022strong,dai2023gaussian} is the maximum likelihood estimator (MLE), which corresponds to the solution of a linear assignment problem with weights induced by the two observed point clouds.

For $b>4$, Kunisky and Niles-Weed~\cite{kunisky2022strong} proves that the MLE achieves exact recovery. For $b < 4$, they instead conjectured that the MLE makes $\Omega(n)$ errors. Additionally, a heuristic analysis of a row-wise inner-product procedure is provided, predicting almost exact recovery for $b>2$.

Dai et al.~\cite{dai2023gaussian} studied an asymptotically equivalent setting from the perspective of database alignment. Their results determine polynomial error rates for the MLE throughout the intermediate regime $2<b<4$. In particular, \cite{dai2023gaussian} asserts that the MLE attains almost exact recovery for every $b>2$. This contrasts with the high-dimensional conjectural picture in \cite{kunisky2022strong}, while agreeing with the proved exact recovery threshold at $b=4$.

Wang et al.~\cite{wang2022random} studied geometric graph observations generated from latent point clouds \eqref{eq:intro_two_view}. In order to prove negative results, \cite{wang2022random} gave a necessary condition ruling out almost exact recovery for $b < 2$ in the more stochastically informative model \eqref{eq:intro_two_view}, agreeing with and supplementing the results of \cite{kunisky2022strong,dai2023gaussian}. Our results strengthen this by ruling out the possibility of nontrivial recovery for $b < 2$. Furthermore, our results also rule out a ``something'' phase for the geometric graph models of \cite{wang2022random} below the critical threshold. Our proofs are based on a conditional second moment method which differs from the rate distortion approach in \cite{wang2022random}.

\section{Main results}

\subsection{Two-view geometric planted matching}
\label{sec:twoView_GM}

We start by formally defining the model. 
Let $\mathscr{P}_n$ denote the set of $n \times n$ permutation matrices.
In the geometric planted matching problem we observe
\begin{equation}\label{eq:GM_def}
    \Xb \qquad\text{and}\qquad \Yb = \Pi_* \Xb + \sigma \Zb,
\end{equation}
where $\Pi_*$ is an $n \times n$ uniformly random permutation matrix in $\mathscr{P}_n$, where $\Xb$ and $\Zb$ are $n \times d$ matrices with i.i.d.\ standard Gaussian entries, and where $\sigma > 0$ is a noise parameter. We think of the rows of $\Xb$ as an $n$-point cloud in $\RR^d$, so that the rows of $\Yb$ are a noisy, randomly permuted version of those points. Here and throughout, we identify a permutation matrix $\Pi$ with its associated bijection, and abuse notation by writing $\Pi(i)$ for the unique $j$ such that $\Pi_{ij} = 1$. Thus \eqref{eq:GM_def} is equivalent to the row-wise formulation \eqref{eq:intro_two_view} in the introduction.

We are interested in the ``high-dimensional'' regime where
\begin{equation}\label{eq:GM_highDimensions_def}
    \sigma^2 = \frac{d}{b \log n} \qquad\text{and}\qquad d = \omega(\log n),
\end{equation}
where $b > 0 $ is a fixed parameter. In particular $\sigma^2 \to \infty$.

Our main message for this model is that $b=2$ is a sharp ``all-or-nothing'' phase transition for recovering $\Pi_*$. The ``all'' side refers to almost exact recovery (AER): producing an estimator $\widehat{\Pi}$ whose overlap with the truth satisfies $\tr \widehat{\Pi}\Pi_*^\top/n \to 1$, or equivalently whose number of mismatches is $o(n)$. For $b>2$, this positive side can be deduced from prior work \cite{kunisky2022strong,dai2023gaussian,wang2022random}. 

The novel content is the ``nothing'' side for $b<2$. Here we show that not only does AER fail, but no estimator can recover even a positive fraction of the matching. In fact, we show a stronger statement: the posterior assigns exponentially small mass to permutations whose overlap with the truth is bounded away from zero. Consequently, we derive two information-theoretic impossibility statements, one for estimating the discrete permutation and one for estimating the matched point cloud in Euclidean distance.

Let $\inangle{\cdot}$ denote the expectation with respect to the posterior distribution of $\Pi_*$, given $\Xb$ and $\Yb$. This posterior is given as follows: for any function $f$ on permutation matrices in $\mathscr{P}_n$,
\begin{align}\label{eq:GM_posterior_def}
    \inangle{f} = \frac{1}{\postGeomMPartition(\sigma)} \sum_{\Pi \in \mathscr{P}_n} f(\Pi) \exp \inbraces{\frac{1}{\sigma^2} \tr \Pi \Xb \Yb^\top}, \quad \postGeomMPartition(\sigma) = \sum_{\Pi \in \mathscr{P}_n} \exp \inbraces{\frac{1}{\sigma^2} \tr \Pi \Xb \Yb^\top}.
\end{align}

\begin{theorem}\label{thm:GM_AoN}
Let $\Pi$ denote a sample from the posterior distribution. Then the following hold.
\begin{enumerate}[(i)]
    \item Suppose $b < 2$ is fixed. For any fixed $\delta \in (0,1]$, there exists a constant $C_{b,\delta} > 0$ such that, for all large $n$, 
    \begin{align}\label{eq:GM_AoN}
        \EE \inangle{\1\!\inbraces{\tr \Pi_* \Pi^\top \geq \delta n }}
        \leq \exp\inparen{-C_{b,\delta}n\log n}.
    \end{align}
    \item Suppose $b>2$ is fixed. Then for any fixed $\delta\in(0,1)$,
    \[
        \EE \inangle{\1\!\inbraces{\tr \Pi_* \Pi^\top \leq (1-\delta)n}}
        \to 0.
    \]
\end{enumerate}
\end{theorem}

The negative result (i) is proved in Appendix \ref{sec:GM_AoN_proofs}. The positive result (ii) follows from the $K=2$ specialization of Theorem \ref{thm:Kway_positive_AER} below, whose proof is given in Appendix \ref{sec:Kway_positive_AER_proofs}, and can also be deduced from prior almost exact recovery results \cite{kunisky2022strong,dai2023gaussian,wang2022random}. Nevertheless, for completeness, we give formal proofs for both parts of Theorem \ref{thm:GM_AoN} in Appendices \ref{sec:GM_AoN_proofs} and \ref{sec:Kway_positive_AER_proofs}.

We now translate this posterior overlap statement into information-theoretic impossibility statements. Let $\inangle{\Pi} = \EE[\Pi \mid \Xb, \Yb]$ denote the posterior mean. 
Define
\begin{align*}
    \MMSE(b) \coloneqq \inf_{\widehat{\Pi}(\Xb, \Yb)} \EE \norm{\widehat{\Pi}(\Xb, \Yb) - \Pi_*}_F^2 = \EE \norm{\inangle{\Pi} - \Pi_*}_F^2,
\end{align*}
where the infimum ranges over all possible estimators. Thus $\MMSE(b)$ is the minimum achievable mean-squared error for recovering the planted matching based on $(\Xb,\Yb)$.

We also consider a geometric version of this error, in which the goal is not to recover the permutation itself, but rather the matched point cloud $\Pi_*\Xb$. Define
\begin{equation}\label{eq:GM_YMMSE_def}
    \YMMSE(b)
    \coloneqq
    \inf_{\widehat{\bS}(\Xb,\Yb)}
    \mathbb{E}
    \left\|
        \widehat{\bS}(\Xb,\Yb)-\Pi_*\Xb
    \right\|_F^2,
\end{equation}
where the infimum ranges over all estimators of the matched point cloud.

A computation shows that the mean-squared error for recovering $\Pi_*$ achieved by the dummy estimator $\EE\Pi_*=\frac{1}{n}\boldsymbol{1}\boldsymbol{1}^{\top}$ is $n-1$. Similarly the mean-squared error for recovering the matched point cloud $\Pi_* \Xb$ achieved by the dummy estimator $\frac1n\boldsymbol{1}\boldsymbol{1}^{\top}\Xb$ is $(n-1)d$. Our next result shows that for $b < 2$, asymptotically, we can do no better than these dummy estimators.
\begin{corollary}\label{cor:GM_infoTheoretic_nothing}
Suppose $0<b<2$ is fixed. Then the following occur.
\begin{enumerate}[(i)]
    \item We have $\frac{\MMSE(b)}{n} \to 1$.
    \item We have
    $
        \frac{\YMMSE(b)}{nd}\to 1.
    $
\end{enumerate}
\end{corollary}

Thus, below $b=2$, we fail to recover the latent discrete matching, as well as match the points geometrically, in the sense that we can do no better than random guessing. Corollary \ref{cor:GM_infoTheoretic_nothing} is proved in Appendix \ref{sec:GM_infoTheoretic_nothing_proofs}.

Our results are based on precise asymptotics of the posterior normalizing constant $\cZ_{n,d}(\sigma)$ in \eqref{eq:GM_posterior_def}. In our case, this quantity can be interpreted as a type of random assignment \emph{partition function}. Its normalized logarithm is referred to as the \emph{free energy}, and is spelled out as
\begin{equation}\label{eq:GM_fullFE_def}
    \frac{1}{n \log n}\log \postGeomMPartition(\sigma)
    =
    \frac{1}{n \log n} \log \sum_{\Pi \in \mathscr{P}_n} \exp \inbraces{  \frac{1}{\sigma^2} \tr \Pi \Xb \Xb^\top\Pi_*^\top + \frac{1}{\sigma} \tr \Pi \Xb \Zb^\top  }.
\end{equation}

We highlight a key free energy asymptotic result.

\begin{proposition}\label{prop:GM_full_free_energy_subcritical}
    Suppose $b<2$ is fixed. Then
    \[
        \EE\insquare{\frac{1}{n \log n}\log \postGeomMPartition(\sigma)}
        =
        1+\frac b2+o(1).
    \]
\end{proposition}

We remark that the proof of Theorem \ref{thm:GM_AoN} also uses \emph{restricted} versions of this free energy, known as the \emph{Franz--Parisi potential}, where the sum in \eqref{prop:GM_full_free_energy_subcritical} is constrained to permutations $\Pi$ with a fixed overlap with the truth.

We finish this section by explaining the implication of our negative result Theorem \ref{thm:GM_AoN} (i) on related matching problems.

\begin{remark}\label{remark:WWXY_implication}
Wang, Wu, Xu, and Yolou \cite{wang2022random} study geometric graph matching models in which the latent point clouds satisfy \eqref{eq:GM_def}, but one observes only pairwise data generated through a transition kernel $W$:
\[
    A_{ij}\sim W(\cdot\mid X_i,X_j),
    \qquad
    B_{ij}\sim W(\cdot\mid Y_i,Y_j).
\]
Specific examples include the dot-product and distance kernel models. These observations are stochastically degraded relative to just observing $(\Xb,\Yb)$ as in \eqref{eq:GM_def}. Thus, to prove negative results, \cite[Appendix E, Proposition 1]{wang2022random} gives a necessary condition ruling out almost exact recovery for $b < 2$ in model \eqref{eq:GM_def} (referred to as ``linear assignment'' there).

However, this still leaves open the possibility of recovering a positive fraction of the matching. Our result closes this gap: Theorem \ref{thm:GM_AoN} and Corollary \ref{cor:GM_infoTheoretic_nothing} imply, by data processing, that when $d=\omega(\log n)$ and $\sigma^2=d/(b\log n)$ with fixed $b<2$, no estimator from such graph observations can have $c n$ correct matches for any fixed $c>0$. Thus, our results also reveal a new consequence for high-dimensional geometric graph matching models: there is no ``something'' phase below $b=2$.
\end{remark}

\subsection{Multi-view geometric planted matching}

In this section we consider a natural extension of the setting in \eqref{eq:GM_def}. We observe $K$ noisy, independently permuted views of the same latent point cloud:
\begin{align}
    \Yb^{(0)} &= \bar{\Pi}_*^{(0)}\Xb + \tau \Wb^{(0)} , \nonumber\\
    \Yb^{(1)} &= \bar{\Pi}_*^{(1)}\Xb + \tau \Wb^{(1)} , \nonumber\\
    &\,\,\,\vdots \nonumber\\
    \Yb^{(K-1)} &= \bar{\Pi}_*^{(K-1)}\Xb + \tau \Wb^{(K-1)} , \label{eq:KwayMatching_def}
\end{align}
where $\Xb$, $\Wb^{(0)},\dots,\Wb^{(K-1)}$ are independent random matrices in $\RR^{n \times d}$, each with i.i.d.~standard Gaussian entries, and where $\bar{\Pi}_*^{(0)},\dots,\bar{\Pi}_*^{(K-1)}$ are independent uniformly random permutation matrices, independent of these Gaussian matrices. The absolute latent matchings are not identifiable. The identifiable signals of interest are the relative matchings from view $0$ to the other views:
\begin{equation}
    \Pi_*^{(a)}:=\bar{\Pi}_*^{(0)}(\bar{\Pi}_*^{(a)})^\top,
    \qquad a=1,\dots,K-1.
\label{eq:def-relative-matching}
\end{equation}
Indeed, row $i$ of view $0$ and row $j$ of view $a$ correspond to the same latent row precisely when $\bar{\Pi}_*^{(0)}(i)=\bar{\Pi}_*^{(a)}(j)$, or equivalently $(\Pi_*^{(a)})_{ij}=1$.

We claim that with the mapping
\begin{equation}\label{eq:multi_twoView_Mapping}
\sigma^2 = 2\tau^2 +  \tau^4,
\end{equation}
i.e.~$\sigma^2 \sim \tau^4$ asymptotically,
the model \eqref{eq:KwayMatching_def} is equivalent to a $K$-view extension of the setting \eqref{eq:GM_def}, which would be a $2$-view model. We provide details in Remark \ref{remark:Kview_TwoView_Mapping}.

To recover the relative matchings $\Pi_*^{(a)}$ in \eqref{eq:def-relative-matching}, we
define the statistic
\begin{equation}\label{eq:KwayMatching_Ti_def}
    T_{\ib}
    =
    \sum_{0\leq a<b\leq K-1}
    Y_{i_a}^{(a)\top}Y_{i_b}^{(b)}
\end{equation}
for any candidate tuple
$
    \ib=(i_0,\dots,i_{K-1})\in[n]^K.
$
One should think of $T_{\ib}$ as the score when attempting to match indices $i_1,\dots,i_{K-1}$ from views $1,\dots,K-1$ respectively to $i_0$.

We next give a relatively simple algorithm based on the statistics $T_{\ib}$ for identifying the relative matchings $\Pi_*^{(a)}$, $a=1,\dots,K-1$.
Briefly, for each row $i$ in view $0$, the algorithm selects $j_1^\star(i),\dots,j_{K-1}^\star(i)$ for views $1,\dots,K-1$ as candidate matches, which is the tuple $(j_1,\dots,j_{K-1})$ that attains the largest score $T_{i,j_1,\dots,j_{K-1}}$ over all such tuples. A slight complication is that these greedy row-wise choices may not define a valid permutation altogether: for instance row $j$ of view 1 may end up as a candidate match for two rows $i$ and $i'$ of view 0. For each such $j$, the algorithm keeps one corresponding row $i$ arbitrarily. Finally, we complete the remaining unmatched rows and columns so that the output matrix for each view is indeed a valid permutation matrix in $\mathscr{P}_n$.

\begin{algorithm}[hbt!]
\DontPrintSemicolon
\caption{Row-wise thresholding and completion for fixed $K$}
\label{alg:Kway_general_rowwise_thresholding}
\For{$i \gets 1$ \KwTo $n$}{
    Choose
    $\displaystyle
    \inparen{j_1^\star(i),\dots,j_{K-1}^\star(i)}
    \in
    \argmax_{(j_1,\dots,j_{K-1})\in[n]^{K-1}}
    T_{i,j_1,\dots,j_{K-1}}$.
    Break any ties arbitrarily.\;
}
\For{$a \gets 1$ \KwTo $K-1$}{
    Initialize $\widehat{\Pi}^{(a)}$ as the $n\times n$ zero matrix.\;
    \ForEach{$j\in \inbraces{j_a^\star(i):i\in[n]}$}{
        Choose one $i\in[n]$ such that $j_a^\star(i)=j$, and set $\widehat{\Pi}^{(a)}_{ij}=1$.\;
    }
    Complete $\widehat{\Pi}^{(a)}$ arbitrarily to a permutation matrix in $\mathscr{P}_n$ by padding ones appropriately in the remaining zero rows and columns.\;
}
\Return $\widehat{\Pi}^{(1)},\dots,\widehat{\Pi}^{(K-1)}$.
\end{algorithm}
For fixed $K$, Algorithm \ref{alg:Kway_general_rowwise_thresholding} is a polynomial-time procedure: the exhaustive maximization considers $n^{K-1}$ tuples for each of the $n$ anchor rows, and each score is computed from $O_K(1)$ inner products in $\RR^d$. Thus the running time is $O_K(n^K d)$.

\begin{theorem}\label{thm:Kway_positive_AER}
Fix $K\geq 2$ and consider the $K$-view model \eqref{eq:KwayMatching_def} in the regime 

\[
    d = \omega(\log n),\qquad
    \tau^4=\frac{d}{b\log n},
    \qquad\text{and}\qquad
    b>\frac{K}{K-1}.
\]

Then Algorithm \ref{alg:Kway_general_rowwise_thresholding} achieves almost exact recovery for the relative matchings: for every $a=1,\dots,K-1$,
\begin{equation}\label{eq:KwayMatching_AER_statement}
    \frac{1}{n}\tr \widehat{\Pi}^{(a)}\inparen{\Pi_*^{(a)}}^\top \to 1
    \qquad\text{in probability.}
\end{equation}
\end{theorem}
By \eqref{eq:multi_twoView_Mapping}, the scaling $\tau^4 = d/(b \log n)$ is asymptotically equivalent to the two-view scaling in \eqref{eq:GM_def}. Thus Theorem \ref{thm:Kway_positive_AER} should be read as the multi-view analogue of the positive side of the two-view threshold. It shows that additional views improve the sufficient condition for recovery from $b>2$ when $K=2$ to $b>K/(K-1)$ for general fixed $K$.

\begin{remark}
    Specializing to $K=3$, we see that for $b$ between $3/2$ and $2$, we get nothing for the $2$-view matching, i.e. the relative matching $\Pi_*^{(1)}$ cannot be recovered up to a linear number of errors (c.f.~Corollary \ref{cor:GM_infoTheoretic_nothing}). But with a third view, we can recover both relative matchings $\Pi_*^{(1)}$ and $\Pi_*^{(2)}$.

    An intuitive explanation is that, for \(K=3\), the statistic \(T_{\ib}\) does more than score two candidate matches against view \(0\): it also rewards consistency between the proposed rows in views \(1\) and \(2\). Theorem \ref{thm:Kway_positive_AER} quantifies the resulting gain in the recovery threshold.
\end{remark}

Taken together, Corollary \ref{cor:GM_infoTheoretic_nothing} and \ref{thm:Kway_positive_AER} imply an ``all-or-nothing'' phase transition at $b = 2$ for the two-view model \eqref{eq:intro_two_view}.
Determining whether $b=K/(K-1)$ is the sharp information-theoretic threshold for $K$-view recovery in general is an interesting direction for future work. A challenge is that a matching lower bound may require controlling the multi-view posterior over a higher-dimensional overlap structure.

\subsection{Gaussian weighted matching}

We now record a ``$d=\infty$'' version of the two-view matching problem, in which the geometry does not play any role. Let $K_{n,n}$ be the complete bipartite graph with both partitions identified with $[n]$. Perfect matchings in $K_{n,n}$ may be identified with $n \times n$ permutation matrices in $\mathscr{P}_n$.

In \emph{Gaussian weighted matching}, we first draw a perfect matching $\Pi_* \in \mathscr{P}_n$ on $K_{n,n}$ uniformly at random. We observe a weighted adjacency matrix $A=(A_{ij})_{i,j\in[n]}$ on $K_{n,n}$ with weights given as follows. For a parameter $\beta > 0$, conditionally on $\Pi_*$, the edge weights of $A$ are independent: edges in $\Pi_*$ have law $\cN(1, 1/\beta)$, while all other edges have law $\cN(0, 1/\beta)$. Equivalently, for $Z=(Z_{ij})_{i,j\in[n]}$ with i.i.d.\ entries $Z_{ij} \sim \cN(0,1)$ and independent of $\Pi_*$, we observe
\begin{equation}\label{eq:gaussianWeightedMatching_def}
    A = \Pi_* + \frac{1}{\sqrt{\beta}} Z .
\end{equation}
Such planted matching models were studied in \cite{ding2023planted}, with particular focus on exponential edge weights. Remark~1 in \cite{ding2023planted} also discusses the Gaussian weighted version. In this paper, we include \eqref{eq:gaussianWeightedMatching_def} as a non-geometric companion to \eqref{eq:GM_def}. Both models share the same matching structure, but conditionally on $\Pi_*$, the edge weights in \eqref{eq:gaussianWeightedMatching_def} are independent Gaussians, whereas in \eqref{eq:GM_def} the edge weights remain dependent based on an underlying geometric structure.

The posterior distribution is a distribution over matchings in $\mathscr{P}_n$ given by 
\begin{align}
    P(\Pi \mid A) = \frac{\exp \cH(\Pi)}{\cZ_{\text{GWM}}(\beta)}, \qquad \cZ_{\text{GWM}}(\beta) = \sum_{\Pi \in \mathscr{P}_n} \exp(\cH(\Pi)),
\end{align}
where 
$
    \cH(\Pi) = \beta \tr \Pi_* \Pi^\top + \sqrt{\beta}\tr Z \Pi^\top.
$
Throughout this subsection, $\inangle{\cdot}$ denotes the expectation with respect to this posterior. Our main result for this model is analogous to Theorem \ref{thm:GM_AoN}.

\begin{theorem}\label{prop:plantedGaussianAoN}
Consider the Gaussian weighted matching model \eqref{eq:gaussianWeightedMatching_def} and suppose $\beta = b \log n$ for $b > 0$. Let $\Pi$ denote a sample from the posterior distribution. Then the following hold.
\begin{enumerate}[(i)]
    \item Suppose $b < 2$ is fixed. For any fixed $\delta \in  (0,1]$, there exists a constant $C_{b,\delta} > 0$ such that
    \begin{align}\label{eq:plantedGaussianAoN} 
        \EE \inangle{\1\!\inbraces{\tr \Pi_* \Pi^\top \geq \delta n }}
        \leq \exp\inparen{- C_{b,\delta}n \log n}.
    \end{align}
    \item Suppose $b>2$ is fixed. Then for any fixed $\delta\in(0,1)$,
    \[
        \EE \inangle{\1\!\inbraces{\tr \Pi_* \Pi^\top \leq (1-\delta)n}}
        \to 0.
    \]
\end{enumerate}
\end{theorem}

\begin{remark}
    We omit the proof of Theorem~\ref{prop:plantedGaussianAoN} because it follows from the same ideas as the geometric result Theorem \ref{thm:GM_AoN}. In the geometric model \eqref{eq:GM_def}, the noise term in the posterior involves the matrix $\Xb\Zb^\top/\sqrt{d}$; in the ``$d=\infty$'' model \eqref{eq:gaussianWeightedMatching_def} this is replaced by an $n\times n$ matrix $Z$ of i.i.d.~standard Gaussians. Thus the proof of the negative result (i) is analogous to that for the geometric setting, and is in fact simpler.

    For part (ii), after some notation translation, \cite[Remark 1]{ding2023planted} gives $b > 2$ as the threshold for almost exact recovery. The subsequent passage to the posterior overlap statement is identical to that for Theorem \ref{thm:GM_AoN}(ii). Hence we omit the details of the proofs.
\end{remark}

Let $\inangle{\Pi} = \EE[\Pi_* \mid A]$ be the posterior mean. Define
\begin{align*}
    \MMSE_{\mathrm{GWM}}(b) \coloneqq \inf_{\widehat{\Pi}(A)} \EE \norm{\widehat{\Pi}(A) - \Pi_*}_F^2
    = \EE \norm{\inangle{\Pi} - \Pi_*}_F^2,
\end{align*}
where the infimum ranges over all possible matrix-valued estimators. The dummy estimator $\EE \Pi_*=\frac{1}{n}\boldsymbol{1}\boldsymbol{1}^{\top}$ has risk $n-1$, and the next result shows that this is asymptotically optimal below the critical threshold. The proof is analogous to that of Corollary~\ref{cor:GM_infoTheoretic_nothing} and is omitted.

\begin{corollary}\label{thm:plantedGaussianNoRecovery}
    Consider the Gaussian weighted matching model \eqref{eq:gaussianWeightedMatching_def} and suppose $\beta = b \log n$ for some fixed $b < 2$. Then
    $
        \frac{\MMSE_{\mathrm{GWM}}(b)}{n}\to 1.
    $
\end{corollary}

\begin{remark}
Recent independent work \cite{hou2026recovery} studied recovery of planted sparse structures in graphs in great generality. In particular, \cite[Theorem 6]{hou2026recovery} can be specialized to the setting of \eqref{eq:gaussianWeightedMatching_def} to recover the same $b=2$ all-or-nothing threshold for \eqref{eq:gaussianWeightedMatching_def}, and implies Corollary \ref{thm:plantedGaussianNoRecovery}.

Our Theorem \ref{prop:plantedGaussianAoN}(i) provides a stronger statement by showing that the posterior mass assigned to matchings with positive overlap with the truth decays to zero at a quantitative exponential rate. The area I-MMSE approach used in \cite{hou2026recovery} differs from our conditional second moment method, which tackles the posterior directly.

At a high level, the framework in \cite{hou2026recovery} assumes independent edge weights conditional on the planted structure, and so does not encompass the geometric setting \eqref{eq:GM_def} where the edge weights remain dependent even conditional on the planted permutation.
\end{remark}

\section{Proof overview of Theorems \ref{thm:GM_AoN} and \ref{thm:Kway_positive_AER}}

\paragraph{Proof overview of Theorem \ref{thm:GM_AoN}.}

We focus on the negative result (i), which is the main content. By an equivariance argument, we reduce to the case $\Pi_*=I$. 
We consider a fixed overlap level $\alpha n \in \ZZ$ with $\alpha\in(0,1]$ and show through a concentration argument that the expected posterior mass at this overlap is approximately
\begin{equation}\label{eq:proofOverview_overlapMass}
    \EE\inangle{\1\!\inbraces{\tr \Pi=\alpha n}}
    \approx
    \exp\!\left\{
        n\log n\,
        \big(\EE\Psi_\sigma(\alpha)-\EE\widetilde \Psi_\sigma+o(1)\big)
    \right\},
\end{equation}
where
$
    \Psi_{\sigma}(\alpha)
    =
    (n \log n)^{-1}
    \log
    \sum_{\Pi:\tr \Pi=\alpha n}
    \exp\!\left\{
        \frac{1}{\sigma^2}\tr \Pi\Xb\Xb^\top
        +
        \frac{1}{\sigma}\tr \Pi\Xb\Zb^\top
    \right\},
$
and $\widetilde{\Psi}_{\sigma}$ is defined in the same way, but without the constraint on $\Pi$. That is, $\widetilde{\Psi}_{\sigma}$ is the free energy in \eqref{eq:GM_fullFE_def} after setting $\Pi_* = I$. The focus is on showing that $\EE\Psi_\sigma(\alpha) < \EE\widetilde \Psi_\sigma$ strictly.

\medskip
\noindent
\underline{Upper bound on $\EE\Psi_\sigma(\alpha)$.}
For any permutation $\Pi$, denote its \emph{fixed points} by $S(\Pi) = \{i:\Pi_{ii}=1\}$, and its \emph{derangement} part by $S(\Pi)^c = [n] \backslash S(\Pi)$. Set $k=\alpha n$. For $\tr\Pi=k$ and $S=S(\Pi)$, define 
\[
    F(S)
    =
    \frac{1}{\sigma^2}\norm{\Xb_S}_{F}^{2}
    +
    \frac{1}{\sigma}\tr \Xb_S\Zb_S^\top,
    \qquad
    R(S,\Pi)
    =
    \frac{1}{\sigma^2}\tr \Xb_{\Pi(S^c)}\Xb_{S^c}^\top
    +
    \frac{1}{\sigma}\tr \Xb_{\Pi(S^c)}\Zb_{S^c}^\top ,
\]
where the notation $\Xb_S \in \RR^{k \times d}$ extracts the rows of $\Xb$ indexed by $S$, preserving their natural ordering, and where $\Xb_{\Pi(S^c)}=\Pi\big|_{S^c}\Xb_{S^c}$, with $\Pi\big|_{S^c}$ being the $(n-k) \times (n-k)$ permutation matrix corresponding to the restriction of $\Pi$ to $S^c$.

The proof decomposes $\EE\Psi_\sigma(\alpha)$ into the fixed point part and the derangement part:
\begin{align*}
    \EE\Psi_\sigma(\alpha)
    &=
    \frac{1}{n\log n}
    \EE\log
    \sum_{\Pi:\tr \Pi=k}
    \exp\!\left\{F(S(\Pi))+R(S(\Pi),\Pi)\right\} \\
    &\leq
    \frac{1}{n\log n}
    \EE\log
    \max_{S:\abs{S}=k}
    \exp\!\left\{F(S)\right\}
    +
    \frac{1}{n\log n}
    \EE\log
    \sum_{S:\abs{S}=k}\,
    \sum_{\Pi:S(\Pi)=S}
    \exp\!\left\{R(S,\Pi)\right\}.
\end{align*}
In the first term, for each $S$, the norm $\norm{\Xb_S}_F^2$ is chi-squared, while the $\tr \Xb_S\Zb_S^\top$ term is conditionally Gaussian given $\Xb$. Standard chi-squared and Gaussian maximum bounds then yield that the first term's contribution is at most $b\alpha+o(1)$. In the second term, we apply Jensen's inequality to push the expectation over the $\log$, which reduces the problem to controlling a Gaussian quadratic moment generating function: for an admissible $(S,\Pi)$ pair, we compute $\EE e^{R(S,\Pi)} = \EE_{\Xb} \EE_{\Zb}e^{R(S,\Pi)}=\EE_{\Xb}\exp\{\xb^\top H \xb\}$, where $\xb=\text{vec}(\Xb_{S^c})$ and $H = \frac{1}{2\sigma^2}(I_{n-k}+2\Pi\big|_{S^c})\otimes I_d$. This quantity depends only on the cycle structure of the derangement $\Pi\big|_{S^c}$, and we carry out a spectral analysis to show that it is maximized when $\Pi\big|_{S^c}$ consists of as many $2$-cycles as possible. This leads to a bound on the second term's contribution of $(1-\alpha)(1+b/2)+o(1)$. Altogether, 
$
    \EE\Psi_\sigma(\alpha)
    \leq
    1+b/2+\alpha(b-2)/2+o(1).
$

\medskip
\noindent
\underline{Lower bound on $\EE\widetilde\Psi_\sigma$.}
We aim to show $\EE\widetilde\Psi_\sigma \geq 1 + \frac{b}{2} + o(1)$. We start with a reduction argument to a \emph{noise} partition function $\cZ$, such that
\[
    \EE\widetilde\Psi_\sigma
    \approx
    \frac{1}{n\log n}\EE\log\cZ,
    \qquad\text{where}\qquad
    \cZ=
    \sum_{\Pi\in\mathscr P_n}
    \exp\!\left\{
        \sqrt{b\log n}\,
        \tr \Pi\frac{\Xb\Zb^\top}{\sqrt d}
    \right\}.
\]
We will adopt a conditional second moment method, and show that
\[ \frac{1}{n\log n}\log \EE[\cZ^2 \1_{\cE}]
\approx \frac{2}{n\log n} \log \EE [\cZ \1_{\cE}]\]
for an appropriately chosen high probability event $\cE$:
For any small $\epsilon,c_0>0$ and $k_0=c_0n$, set
\[
    \cE
    =
    \left\{
    \max_{S\in\mathscr{P}_k}
    \frac{1}{\sqrt{dk}}\tr S\Xb\Zb^\top
    \leq
    \sqrt{(2+\epsilon)(k\vee k_0)\log n}
    \quad\text{for all }1\leq k\leq n \text{ s.t.}\ \mathscr{P}_k \ne \varnothing
    \right\},
\]
where $\mathscr{P}_k$ denotes the set 
of $k$-subpermutation matrices, i.e.~$n\times n$ zero-one matrices 
with exactly $k$ rows and $k$ columns containing a single one.

Certainly, the event $\cE$ is not immediately intuitive, so let us provide motivation. Consider the untruncated first moment $\EE\cZ = \exp\inbraces{\inparen{1+\frac b2+o(1)}n\log n}$ which
matches the target exponent.
However, the second moment blows up when $b>1$. To see this, write
\begin{align}
    \EE\cZ^2
    &=
    \sum_{\Pi,\Pi'\in\mathscr{P}_n}
    \EE\exp\inbraces{
        \sqrt{b\log n}\,
        \tr(\Pi+\Pi')\frac{\Xb\Zb^\top}{\sqrt d}
    } \label{eq:proofOverview_noise_second_moment} \\
    &=
    n! \sum_{k=0}^n \sum_{\Pi:\tr \Pi=k}
    \EE\exp\inbraces{
        \sqrt{b\log n}\,
        \tr(I+\Pi)\frac{\Xb\Zb^\top}{\sqrt d}
    }. \label{eq:proofOverview_noise_second_moment_symmetrized}
\end{align}
We will argue that the main culprits are pairs $(\Pi,\Pi')$ in \eqref{eq:proofOverview_noise_second_moment} that share a submatching. In  \eqref{eq:proofOverview_noise_second_moment_symmetrized}, this corresponds to permutations $\Pi$ with say $k$ fixed points. These carry a doubled exponential
weight on the $k$ aligned coordinates 
$\inbraces{X_i^\top Z_i \colon \Pi_{ii}=1}$, yielding an excess factor $e^{(1+o(1))bk\log n}$
over the corresponding term in $(\EE\cZ)^2$. However, such permutations make
up a proportion of only $\approx 1/k! \approx \exp\{-k\log k\}$ of all
permutations. Consequently,
\begin{equation}\label{eq:proofOverview_secondMomentRatio}
    \frac{\EE\cZ^2}{(\EE\cZ)^2}
    \approx
    \exp\inbraces{\max_{0\leq k\leq n}\inparen{bk\log n - k\log k} + o(n\log n)},
\end{equation}
which is $e^{o(n\log n)}$ if $b\leq1$ (attained at $k=n^b$),
but $e^{(b-1+o(1))n\log n}$ if $b>1$ (attained at $k=n$).
Thus the vanilla second moment method fails for $b > 1$,
even though $\EE \log \cZ \approx \log \EE \cZ$ remains the correct scale for $b < 2$.

A pair $(\Pi,\Pi')$ in a summand of \eqref{eq:proofOverview_noise_second_moment} that share a $k$-subpermutation matrix $S$ contributes $2\sqrt{bk \log n} V_S$ to the exponent, where $V_S = \frac{1}{\sqrt{dk}} \tr S \Xb \Zb^\top$. On the other hand, one can show that the maximum of $V_S$ over $\mathscr{P}_k$ is typically $\sqrt{2k \log n}$. Thus, such summands in \eqref{eq:proofOverview_noise_second_moment} are carried by rare events. After accounting for the entropy discount $1/k!$ as in \eqref{eq:proofOverview_secondMomentRatio}, one finds that for $b > 1$, the ratio blowup is driven by these rare, unusually large $V_S$ contributions. The event $\cE$ is designed to cap the quantities $V_S$ at their typical scale, uniformly over all $k$ and $S \in \mathscr{P}_k$.

We show that $\cE$ is a high probability event. Furthermore, writing $W=\Xb\Zb^\top/\sqrt d$, we show that for any fixed permutation $\Pi$ (c.f.~Lemma
\ref{lemma:GM_highProbEvent_conditional_corrected}):
\[
   \PP\insquare{\cE^c \,\bigg|\, \tr \Pi W}
    \1\left\{\abs{\frac{1}{\sqrt n}\tr \Pi W}\leq \sqrt{2n\log n}\right\}
    = o(1).
\]
This conditional statement arises from the need to show that truncation preserves the first moment: $\EE[\cZ\1_{\cE}] \approx \EE \cZ$. In $\EE[\cZ\1_{\cE}]$, a generic summand is $\EE[\exp\{\sqrt{b\log n}\tr \Pi W\}\1_{\cE}]$, which tilts towards large values of $\tr \Pi W$, so a priori, the factor $\1_{\cE}$ could remove the contributions that carry the first moment. The conditional estimate above rules this out.

Altogether, we show that the truncated moments satisfy 
\[
    \frac{1}{n\log n}\log \EE[\cZ\1_\cE]
    =
    1+\frac b2+o(1),
    \qquad
    \frac{1}{n\log n}\log \EE[\cZ^2\1_\cE]
    \leq
    b+2+o(1).
\]
A routine application of the Paley--Zygmund inequality then leads to $\EE\widetilde\Psi_\sigma\geq 1+b/2+o(1)$.

Combining these bounds in \eqref{eq:proofOverview_overlapMass} gives the exponent $\alpha(b-2)/2+o(1)$, which is strictly negative for $b<2$. Summing over the at most $n$ possible overlap levels $\alpha=k/n\geq\delta$ yields the result.

\paragraph{Proof overview of Theorem \ref{thm:Kway_positive_AER}.}

By an equivariance argument, we work in relabeled coordinates where all planted relative matchings are the identity. Recall that, for each anchor row $i_0=i$, Algorithm~\ref{alg:Kway_general_rowwise_thresholding} selects the candidate tuple $\ib=(i,j_1,\dots,j_{K-1})$ maximizing the statistic $T_{\ib}$ from \eqref{eq:KwayMatching_Ti_def}, and then completes these greedy row-wise choices to permutation matrices.

For a true tuple $\ib = (i,i,\ldots,i)$, $T_{\ib}$ has mean $\binom K2 d$. A wrong tuple $\ib$ is classified by the multiplicities $\boldsymbol\ell=(\ell_0,\ldots,\ell_r)$ of its distinct latent labels, where $r$ is the number of labels other than the anchor label. Then $T_{\ib}$ has mean $d\sum_s\binom{\ell_s}{2}$, so the signal gap is
\[
    D_{\boldsymbol\ell}
    =
    \binom K2-\sum_{s=0}^r\binom{\ell_s}{2}.
\]
We show that the probability of a wrong tuple exceeding the typical true tuple score satisfies
\[
    \PP\!\left\{
    T_{\ib}\geq \left(\binom K2-\epsilon\right)d
    \right\}
    \leq
    \exp\!\left\{
    -\left(\frac{(D_{\boldsymbol\ell}-\epsilon)^2}{K(K-1)}+o(1)\right)
    \frac{d}{\tau^4}
    \right\}.
\]
At the same time there are only $O_K(n^r)$ such tuples. The key combinatorial fact driving the threshold is
\[
    \frac{D_{\boldsymbol\ell}^2}{K(K-1) r}\geq \frac{K-1}{K},
\]
with equality precisely for the one-outlier multiplicity pattern $\boldsymbol\ell=(K-1,1)$, corresponding to tuples such as $(i,i,\ldots,i,j)$ with $j\neq i$. Since $d/\tau^4= b \log n$, the union bound over all competitor patterns is $o(1)$ whenever $b>K/(K-1)$. Thus a fixed row is selected incorrectly with probability $o(1)$, and hence only $o(n)$ rows are wrong in expectation.

\appendix

\section{Proofs for the two-view negative result}
\label{sec:GM_AoN_proofs}
To prove Theorem \ref{thm:GM_AoN}, we first make some reductions.

We claim that without loss, it suffices to consider $\Pi_* = I$ in the left-hand side of \eqref{eq:GM_AoN}. Indeed, fix any event $A\subseteq \{0,1,\dots,n\}$. Then for any fixed $\Pi_*$, by the change of variables $\Pi \mapsto \Pi_*^\top \Pi =: Q$,
\begin{align*}
\EE \inangle{\1\!\inbraces{\tr \Pi_* \Pi^\top \in A}}
&=
\EE_{\Pi_*}\EE_{\Xb,\Zb}\insquare{
\frac{
\sum_{\Pi}
\1\!\inbraces{\tr \Pi_* \Pi^\top \in A}
\exp\inbraces{\frac{1}{\sigma^2}\tr \Pi \Xb(\Pi_*\Xb+\sigma\Zb)^\top}
}{
\sum_{\Pi}
\exp\inbraces{\frac{1}{\sigma^2}\tr \Pi \Xb(\Pi_*\Xb+\sigma\Zb)^\top}
}
} \\
&=
\EE_{\Pi_*}\EE_{\Xb,\Zb}\insquare{
\frac{
\sum_{Q}
\1\!\inbraces{\tr Q \in A}
\exp\inbraces{\frac{1}{\sigma^2}\tr Q \Xb\Xb^\top+\frac{1}{\sigma}\tr Q \Xb(\Pi_*^\top\Zb)^\top}
}{
\sum_{Q}
\exp\inbraces{\frac{1}{\sigma^2}\tr Q \Xb\Xb^\top+\frac{1}{\sigma}\tr Q \Xb(\Pi_*^\top\Zb)^\top}
}
} \\
&=
\EE_{\Xb,\Zb}\insquare{
\frac{
\sum_{\Pi}
\1\!\inbraces{\tr \Pi \in A}
\exp\inbraces{\frac{1}{\sigma^2}\tr \Pi \Xb\Xb^\top+\frac{1}{\sigma}\tr \Pi \Xb\Zb^\top}
}{
\sum_{\Pi}
\exp\inbraces{\frac{1}{\sigma^2}\tr \Pi \Xb\Xb^\top+\frac{1}{\sigma}\tr \Pi \Xb\Zb^\top}
}
} \\
&= \EE \inangle{\1\!\inbraces{\tr \Pi \in A}},
\end{align*}
where the third equality uses $\Pi_*^\top\Zb\overset{d}{=}\Zb$ and then renames $Q$ back to $\Pi$. Note that in the last line, and henceforth, $\inangle{\cdot}$ denotes the posterior in the $\Pi_* = I$ planted model.

We next derive an upper bound for one fixed overlap level. Fix $\alpha \in (0,1]$ such that $\alpha n$ is an integer. Then
    \begin{align}\label{eq:GM_overlapError_in_FPpotential}
        \EE \inangle{\1\!\inbraces{\tr \Pi_* \Pi^\top = \alpha n}}
        &= \EE\inangle{\1\!\inbraces{\tr \Pi = \alpha n}}
        = \EE \exp \inbraces{ n \log n \insquare{ \Psi_\sigma(\alpha) - \widetilde{\Psi}_\sigma }  },
    \end{align}
    where
    \begin{align}
        \Psi_{\sigma}(\alpha) &= \frac{1}{n \log n} \log \sum_{\Pi} \1\!\inbraces{ \tr \Pi = \alpha n  } \exp \inbraces{  \frac{1}{\sigma^2} \tr \Pi \Xb \Xb^\top + \frac{1}{\sigma} \tr \Pi \Xb \Zb^\top  }, \notag\\
        \widetilde{\Psi}_{\sigma} &= \frac{1}{n \log n} \log \sum_{\Pi} \exp \inbraces{  \frac{1}{\sigma^2} \tr \Pi \Xb \Xb^\top + \frac{1}{\sigma} \tr \Pi \Xb \Zb^\top  }.
    \end{align}
    We now push the expectation in \eqref{eq:GM_overlapError_in_FPpotential} into the exponent. Define
    \[
        F(\Xb,\Zb) = \log \sum_{\Pi} \exp \inbraces{  \frac{1}{\sigma^2} \tr \Pi \Xb \Xb^\top + \frac{1}{\sigma} \tr \Pi \Xb \Zb^\top  },
    \]
    so that $(n \log n)^{-1} F(\Xb,\Zb) = \widetilde{\Psi}_{\sigma}$. We claim that, for any $t > 0$, 

    \begin{equation}\label{eq:GM_UpperBound_GaussianConcentration_residual}
        \PP\left[ \abs{\frac{F}{n \log n} - \EE\left[\frac{F}{n \log n}\right] } \geq t  \right] \leq C' \exp \inbraces{ -c' \frac{n \log n}{b} t^2 } + C'\exp(-c'nd),
    \end{equation}
    for universal constants $C', c' > 0$. In particular, for every fixed $t>0$, since $d=\omega(\log n)$, this implies
    \begin{equation}\label{eq:GM_UpperBound_GaussianConcentration}
        \PP\left[ \abs{\frac{F}{n \log n} - \EE\left[\frac{F}{n \log n}\right] } \geq t  \right] \leq C' \exp \inbraces{ -c' \frac{n \log n}{b} t^2 },
    \end{equation}
    for universal constants $C', c' > 0$. Define $G(\Xb) = \EE_{\Zb}[F(\Xb,\Zb)]$. To show \eqref{eq:GM_UpperBound_GaussianConcentration_residual}, we first condition on $\Xb$. Then the map $\Zb \mapsto F(\Xb,\Zb)$ is $\norm{\Xb}_F/\sigma$-Lipschitz. Take $R = 2\sqrt{nd}$ and set $B_R=\inbraces{\norm{\Xb}_F\leq R}$. A standard chi-squared tail bound gives
    \begin{equation}\label{eq:GM_Xnorm_tail}
        \PP[\norm{\Xb}_F \geq R]
        \leq
        \exp(-cnd).
    \end{equation}
    Then by Gaussian concentration of Lipschitz functions (see e.g.~\cite[Theorem 1.3.4]{talagrand2010mean}) and \eqref{eq:GM_Xnorm_tail}, for every $s>0$,
    \begin{align}
        \PP[\abs{F-G(\Xb)}\geq s]
        &=
        \EE_{\Xb}\PP_{\Zb}[\abs{F-G(\Xb)}\geq s \mid \Xb]\nonumber\\
        &=
        \EE_{\Xb}\insquare{
        \PP_{\Zb}[\abs{F-G(\Xb)}\geq s \mid \Xb]\1_{B_R}
        }
        +
        \EE_{\Xb}\insquare{
        \PP_{\Zb}[\abs{F-G(\Xb)}\geq s \mid \Xb]\1_{B_R^c}
        }\nonumber\\
        &\leq
        2\exp\inparen{-c\frac{s^2\sigma^2}{R^2}}+\exp(-cnd).
        \label{eq:GM_FminusG_concentration}
    \end{align}
    We next show that $G(\Xb)$ concentrates as a function of $\Xb$. On the ball $B_R$, differentiating under the expectation and using $\E\|\Zb\|_F \leq \sqrt{nd}$  shows that $G$ is $L_G(R)$-Lipschitz with
    \[
        L_G(R) = \frac{2R}{\sigma^2} + \frac{\sqrt{nd}}{\sigma}.
    \]
    Let $\overline{G}$ be the extension of $G|_{B_R}$ to all of $\RR^{n\times d}$ given by
    \[
        \overline{G}(A)=\inf_{U\in B_R}\inbraces{G(U)+L_G(R)\norm{A-U}_F}.
    \]
    The purpose of $\overline{G}$ is to create a (globally) $L_G(R)$-Lipschitz proxy for $G$. Note that $\overline{G}=G$ on $B_R$. We now show that the difference $\abs{\EE[G(\Xb)]-\EE[\overline{G}(\Xb)]}$ is of lower order.  We have the bound, for each fixed $\Xb$,
    \begin{align*}
        \abs{G(\Xb)}
        &\leq
        \EE_{\Zb}[\abs{F(\Xb,\Zb)}] 
        \leq
        \log(n!)
        +
        \EE_{\Zb}\left[\max_{\Pi}
        \abs{
        \frac{1}{\sigma^2}\tr \Pi \Xb\Xb^\top
        +
        \frac{1}{\sigma}\tr \Pi \Xb\Zb^\top
        }\right] \\
        &\leq
        Cn\log n
        +
        \frac{C}{\sigma^2}\norm{\Xb}_F^2
        +
        \frac{C}{\sigma}\norm{\Xb}_F\sqrt{nd},
    \end{align*}
    where $C > 0$ is a constant. On $B_R^c$, a similar bound holds for $\overline{G}(\Xb)$. Since $\overline{G}$ is $L_G(R)$-Lipschitz, $\overline{G}(0)=G(0)=\log(n!)$, and $R\leq \norm{\Xb}_F$ on $B_R^c$, we have
    \begin{align*}
    \abs{\overline G(\Xb)}
    &\le \abs{\overline G(0)} + L_G(R)\norm{\Xb}_F 
    = \log(n!) + \left(\frac{2R}{\sigma^2}+\frac{\sqrt{nd}}{\sigma}\right)\norm{\Xb}_F \\
    &\le Cn\log n
        + \frac{C}{\sigma^2}\norm{\Xb}_F^2
        + \frac{C}{\sigma}\norm{\Xb}_F\sqrt{nd}.
    \end{align*}

    Therefore, using these bounds,
    \begin{align*}
        \abs{\EE[G(\Xb)]-\EE[\overline{G}(\Xb)]}
        &\leq
        \EE[
        \inparen{\abs{G(\Xb)}+\abs{\overline{G}(\Xb)}}\1_{B_R^c}
        ] \\
        &\leq
        \EE\insquare{
        \inparen{
        Cn\log n
        +
        \frac{C}{\sigma^2}\norm{\Xb}_F^2
        +
        \frac{C}{\sigma}\norm{\Xb}_F\sqrt{nd}
        }
        \1_{B_R^c}
         } \\
        &\leq
        \inparen{
        Cn\log n
        +
        \frac{C}{\sigma^2}\EE[\norm{\Xb}_F^4]^{1/2}
        +
        \frac{C}{\sigma}\sqrt{nd}\,\EE[\norm{\Xb}_F^2]^{1/2}
        }
        \PP[B_R^c]^{1/2} \\
        &\leq
        \inparen{Cn\log n+Cb n\log n+Cn\sqrt{bd\log n}}\exp(-cnd)
        \leq C\exp(-\bar{c}nd),
    \end{align*}
    where in the third line we used Cauchy--Schwarz, and in the last line we used \eqref{eq:GM_Xnorm_tail}, $\norm{\Xb}_F^2\sim\chi^2_{nd}$, and $\sigma^2=d/(b\log n)$, and where $\bar{c}  > 0$ is a constant.

    Thus $G$ concentrates on its expectation: for any $s > 0$,
    \begin{align}
        \PP[\abs{G(\Xb)-\EE[G(\Xb)]}\geq s]
        &\leq
        \PP[B_R^c]
        +
        \PP[
        \abs{\overline{G}(\Xb)-\EE[G(\Xb)]}
        \geq s, \, B_R
         ] \nonumber\\
        &\leq
        \PP[B_R^c]
        +
        \PP[
        \abs{\overline{G}(\Xb)-\EE[\overline{G}(\Xb)]}
        \geq
        s-\abs{\EE[G(\Xb)]-\EE[\overline{G}(\Xb)]}
        ] \nonumber\\
        &\leq
        \PP[B_R^c]
        +
        \PP[
        \abs{\overline{G}(\Xb)-\EE[\overline{G}(\Xb)]}
        \geq
        s-C\exp(-\bar c nd)
        ] \nonumber\\
        &\leq
        C\exp(-cnd)
        +
        C\exp\inparen{-c\frac{s^2}{L_G(R)^2}}.
        \label{eq:GM_G_concentration}
    \end{align}
    where in the last line, the first term comes from \eqref{eq:GM_Xnorm_tail}, while the second term comes from Gaussian concentration for the $L_G(R)$-Lipschitz function $\overline{G}$ with the lower order shift $C\exp(-\bar c nd)$ being absorbed by constants.

    Recall that we fixed $R=2\sqrt{nd}$, and that $\sigma^2=d/(b\log n)$ and $d=\omega(\log n)$. Then for large enough $n$ we have $R/\sigma = 2\sqrt{nb \log n}$ and $L_G(R) \leq C \sqrt{nb \log n}$. Therefore, using $\EE[G(\Xb)]=\EE[F]$, \eqref{eq:GM_FminusG_concentration}, and \eqref{eq:GM_G_concentration},
    \begin{align*}
        \PP[\abs{F-\EE[F]}\geq s]
        &\leq
        \PP[\abs{F-G(\Xb)}\geq s/2]
        +
        \PP[\abs{G(\Xb)-\EE[G(\Xb)]}\geq s/2] \\
        &\leq
        2\exp\inparen{-c\frac{s^2\sigma^2}{R^2}}
        +
        \exp(-cnd)
        +
        C\exp(-cnd)
        +
        C\exp\inparen{-c\frac{s^2}{L_G(R)^2}} \\
        &\leq
        C\exp\inparen{-c\frac{s^2}{bn\log n}}
        +
        C\exp(-cnd).
    \end{align*}
    Setting $s=t n\log n$ gives \eqref{eq:GM_UpperBound_GaussianConcentration_residual}, which implies \eqref{eq:GM_UpperBound_GaussianConcentration} for every fixed $t>0$. 
    An analogous argument for $\Psi_\sigma(\alpha)$ gives the same concentration scale.

    For $\epsilon > 0$, define
    \begin{align*}
        \cE_1 = \inbraces{\widetilde{\Psi}_\sigma \geq \EE \widetilde{\Psi}_\sigma - \frac{\epsilon}{3}}, \qquad\text{and}\qquad
        \cE_2 = \inbraces{\Psi_\sigma(\alpha) \leq \EE \Psi_\sigma(\alpha) + \frac{\epsilon}{6}}.
    \end{align*}
    Then
    \begin{align}\label{eq:GM_EqualDeltaN_goodEvents}
        \PP\insquare{ \cE_1 } \geq 1 - C \exp\inparen{- \frac{c \epsilon^2}{b} n \log n }, \qquad\text{and}\qquad
        \PP\insquare{ \cE_2 } \geq 1 - C \exp\inparen{-\frac{c\epsilon^2}{b}n\log n}.
    \end{align}
    Hence, using \eqref{eq:GM_overlapError_in_FPpotential}, \eqref{eq:GM_EqualDeltaN_goodEvents}, and $\Psi_\sigma(\alpha) \leq \widetilde{\Psi}_\sigma$,
    \begin{align}
        \EE \inangle{ \boldsymbol{1}\!\inbraces{\tr \Pi_* \Pi^\top  = \alpha n } }
        &\leq \exp\inbraces{n \log n \insquare{\EE \Psi_\sigma(\alpha) - \EE \widetilde{\Psi}_\sigma + \frac{\epsilon}{2}}  } + C\exp\inparen{-\frac{c\epsilon^2}{b}n\log n}. \label{eq:GM_concentration_of_FP_potentials}
    \end{align}

Our analysis therefore centers around establishing an upper bound for $\EE \Psi_\sigma(\alpha)$, and a lower bound for $\EE \widetilde{\Psi}_\sigma$. The former will be given next, while the latter was stated in Proposition \ref{prop:GM_full_free_energy_subcritical}. In Proposition  \ref{prop:GM_full_free_energy_subcritical}, the full asymptotics for $\EE \widetilde{\Psi}_\sigma$ was given, and this will also be used for the proof of Corollary \ref{cor:GM_infoTheoretic_nothing}.

\begin{proposition}\label{prop:GM_Psi_UB}
    Let $b > 0$ be fixed. Then for any fixed $\delta_0 > 0$, uniformly over all $\delta \in [\delta_0,1]$ such that $\delta n$ is an integer,
    \begin{equation}\label{eq:GM_Psi_UB}
        \EE \Psi_{\sigma}(\delta) \leq \frac{b}{2} + 1 + \frac{\delta}{2}(b-2) + o(1).
    \end{equation}
\end{proposition}

\begin{proof}[Proof of Theorem \ref{thm:GM_AoN}]
    We first prove (i).
    Take $m \in \inbraces{\ceil{\delta n},\dots,n}$ and set $\delta_m = m/n$. Combining \eqref{eq:GM_concentration_of_FP_potentials}, Proposition \ref{prop:GM_Psi_UB} applied with $\delta$ replaced by $\delta_m$, and Proposition \ref{prop:GM_full_free_energy_subcritical}, for any $\epsilon > 0$ we have that
    \[
        \EE\inangle{\1\!\inbraces{\tr \Pi = m}}
        \leq \exp\inbraces{n\log n\insquare{\frac{\delta_m}{2}(b-2) + \frac{\epsilon}{2} + o(1)}} + C\exp\inparen{-\frac{c\epsilon^2}{b}n\log n}.
    \]
    Since $\delta_m \geq \delta$ and $b < 2$,
    $
        \frac{\delta_m}{2}(b-2) \leq \frac{\delta}{2}(b-2).
    $
    Choosing $\epsilon = \frac{\delta}{4}(2-b)$ gives
    \[
        \EE\inangle{\1\!\inbraces{\tr \Pi = m}}
        \leq \exp\inbraces{n\log n\insquare{\frac{3\delta}{8}(b-2) + o(1)}} + C\exp\inparen{-C_{b,\delta}n\log n}
    \]
    for some constant $C_{b,\delta} > 0$. Summing over the at most $n$ possible values of $m$ proves (i).

    We next prove (ii). The $K=2$ case of Theorem \ref{thm:Kway_positive_AER}, together with the equivalence in Remark \ref{remark:Kview_TwoView_Mapping}, gives an estimator $\widehat\Pi=\widehat\Pi(\Xb,\Yb)$ such that
    \[
        \frac{1}{n}\tr\widehat\Pi\Pi_*^\top\to 1
        \qquad\text{in probability}
    \]
    whenever $b>2$. Since $0\leq n-\tr\widehat\Pi\Pi_*^\top\leq n$, this also implies
    \[
        \EE\norm{\widehat\Pi-\Pi_*}_F^2
        =
        2\,\EE\inparen{n-\tr\widehat\Pi\Pi_*^\top}
        =
        o(n).
    \]
    The posterior mean $\inangle{\Pi}$ minimizes squared error among all estimators, so
    $
        \EE\norm{\inangle{\Pi}-\Pi_*}_F^2=o(n).
    $
    Given $(\Xb,\Yb)$, let $\Pi$ be a fresh sample from the posterior law of $\Pi_*$. Then conditionally on $(\Xb,\Yb)$, $\Pi_*$ and $\Pi$ are two independent samples from this same law. Hence
    \begin{align*}
        \EE\inangle{\norm{\Pi-\Pi_*}_F^2}
        &=
        \EE\inangle{
        \norm{\Pi-\Pi_* \pm\inangle{\Pi} }_F^2
        } 
        =
        \EE\insquare{
        \inangle{
        \norm{\Pi-\inangle{\Pi}}_F^2
        }
        +
        \norm{\Pi_*-\inangle{\Pi}}_F^2
        } \\
        &=
        \EE\norm{\Pi_*-\inangle{\Pi}}_F^2
        +
        \EE\norm{\inangle{\Pi}-\Pi_*}_F^2 
        =
        o(n).
    \end{align*}
    Here the cross term vanishes because $\inangle{\Pi-\inangle{\Pi}}=0$.
    On the event $\tr\Pi_*\Pi^\top\leq (1-\delta)n$, we have $\norm{\Pi-\Pi_*}_F^2\geq 2\delta n$. Therefore
    \[
        \EE \inangle{\1\!\inbraces{\tr \Pi_* \Pi^\top \leq (1-\delta)n}}
        \leq
        \frac{1}{2\delta n}
        \EE\inangle{\norm{\Pi-\Pi_*}_F^2}
        \to 0. \qedhere
    \]
\end{proof}

\subsection{Upper bound on restricted free energy}

\begin{proof}[Proof of Proposition \ref{prop:GM_Psi_UB}]
For any permutation $\Pi$ of $[n]$, let $S(\Pi) \subset [n]$ denote the set of fixed points of $\Pi$. We have 
\begin{align*}
    \tr \Pi \Xb \Yb^\top &= \tr \Pi \Xb \Xb^\top + \sigma\tr \Pi \Xb \Zb^\top \\
    &= \sum_{i \in S(\Pi)} \norm{X_i}_2^2 +  \sum_{i \in S(\Pi)} \sigma  X_i^\top Z_i + \sum_{i \notin S(\Pi)} X_{\Pi(i)}^\top X_i + \sum_{i \notin S(\Pi)} \sigma X_{\Pi(i)}^\top Z_i.
\end{align*}
Set $k=\delta n \in \ZZ$.  We have
\begin{align*}
    &\EE \Psi_{\sigma}(\delta) = \frac{1}{n \log n} \EE \log \sum_{\Pi : \tr \Pi = k} \exp \inbraces{  \sum_{i \in S(\Pi)} \frac{1}{\sigma^2}\norm{X_i}_2^2 +  \sum_{i \in S(\Pi)} \frac{1}{\sigma}X_i^\top Z_i + \sum_{i \notin S(\Pi)} \frac{1}{\sigma^2} X_{\Pi(i)}^\top X_i + \sum_{i \notin S(\Pi)} \frac{1}{\sigma} X_{\Pi(i)}^\top Z_i } \\
    &= \frac{1}{n \log n} \EE \log \insquare{ \sum_{\substack{S \subset [n] \\ \abs{S} = k}} \exp \inbraces{ \frac{1}{\sigma^2} \norm{\Xb_{S}}_{F}^2 + \frac{1}{\sigma} \tr \Xb_S \Zb_S^\top } \sum_{\Pi : S(\Pi) = S} \exp\inbraces{ \frac{1}{\sigma^2} \tr \Xb_{\Pi(S^c)} \Xb_{S^c}^\top + \frac{1}{\sigma} \tr \Xb_{\Pi(S^c)} \Zb_{S^c}^\top  }  } \\
    &\leq (\text{I}) + (\text{II})
\end{align*}
where $\Xb_S$ denotes the $\abs{S} \times d$ matrix consisting of the rows in $\Xb$ indexed by $S$, and similarly for $\Zb_S$, $\Xb_{\Pi(S^c)}, \Zb_{S^c}$, and where 
\begin{align*}
    (\text{I}) &= \frac{1}{n\log n} \EE \log \max_{\substack{S \subset [n] \\ \abs{S} = k}} \exp \inbraces{  \frac{1}{\sigma^2} \norm{\Xb_{S}}_{F}^2 + \frac{1}{\sigma} \tr \Xb_S \Zb_S^\top   } \\
    (\text{II}) &= \frac{1}{n\log n} \EE \log \sum_{\substack{S \subset [n] \\ \abs{S} = k}} \sum_{\Pi : S(\Pi) = S} \exp\inbraces{ \frac{1}{\sigma^2} \tr \Xb_{\Pi(S^c)} \Xb_{S^c}^\top + \frac{1}{\sigma} \tr \Xb_{\Pi(S^c)} \Zb_{S^c}^\top  },  
\end{align*}

Let us consider the diagonal part $(\text{I})$ first. We have the bounds
\begin{align*}
    \EE \max_{\substack{S \subset [n] \\ \abs{S} = k}} \norm{\Xb_{S}}_{F}^2 &\leq kd + \sqrt{C kd \log \binom{n}{k}} = (1+o(1)) kd \\
    \EE \max_{\substack{S \subset [n] \\ \abs{S} = k}} \tr \Xb_S \Zb_S^\top  &\leq (1+o(1)) \sqrt{2 kd \log \binom{n}{k} } = O_{\delta}(k\sqrt{d}),
\end{align*}

where the first line follows from standard chi-squared maximum bounds (see e.g.~\cite[Lemma 1]{laurent2000adaptive}), and 
where the second line follows since conditionally on $\Xb_S$, $\tr \Xb_S \Zb_S^\top \sim \cN(0, \norm{\Xb_S}_F^2)$, and a maximum of Gaussians bound followed by the first line leads to the bound. Therefore, using $\sigma^2 = \frac{d}{b \log n}$, 
\begin{align}\label{eq:geomMatch_overlapError_diagTerm}
    (\text{I}) &\leq \frac{1}{n\log n} \insquare{ \frac{1}{\sigma^2} (1+o(1))kd + \frac{1}{\sigma} O_\delta(k\sqrt{d })  } = (1+o(1))b\delta + O_\delta\!\inparen{\frac{1}{\sqrt{\log n}}}.
\end{align}
We now turn to term $(\text{II})$. By Jensen's inequality,
\[
(\text{II}) \leq \frac{1}{n\log n}  \log \sum_{\substack{S \subset [n] \\ \abs{S} = k}} \sum_{\Pi : S(\Pi) = S} \underbrace{\EE \exp\inbraces{ \frac{1}{\sigma^2} \tr \Xb_{\Pi(S^c)} \Xb_{S^c}^\top + \frac{1}{\sigma} \tr \Xb_{\Pi(S^c)} \Zb_{S^c}^\top  }}_{=: V(S,\Pi)}.
\]
Fix $S$ and some $\Pi$ with $S(\Pi) = S$. We first evaluate the Gaussian moment generating function in $V(S, \Pi)$, conditionally on $\Xb$. This gives us
\begin{align*}
    V(S, \Pi) =  \EE \exp \inbraces{ \frac{1}{2\sigma^2} \norm{\Xb_{S^c}}_{F}^2  + \frac{1}{\sigma^2} \tr \Xb_{\Pi(S^c)} \Xb_{S^c}^\top  } = \EE\exp \bigg\{\xb^\top \underbrace{\frac{1}{2\sigma^2} \inparen{  I_{n-k} \otimes I_{d} + 2 \Pi \big|_{S^c} \otimes I_{d}  }}_{= H} \xb  \bigg\},
\end{align*}
where $\xb = \text{vec}(\Xb_{S^c}) \in \RR^{(n-k)d}$ is its vectorization and $H \in \RR^{(n-k)d \times (n-k)d}$. We let $\Pi\big|_{S^c}$ denote the $(n-k)\times(n-k)$ submatrix of $\Pi$ obtained by keeping only the rows and columns indexed by $S^c$. We replace the exponent by $\xb^\top H \xb = \tfrac12 \xb^\top (H + H^\top) \xb$. Let the real $(n-k)d$ eigenvalues of $\tfrac{1}{2}(H + H^\top)$ be $\lambda_1,\dots,\lambda_{(n-k)d}$.  By spectral decomposition of the symmetric matrix $\tfrac{1}{2}(H + H^\top)$ and orthogonal invariance of the Gaussian, we have $V(S,\Pi) = \prod_{j=1}^{(n-k)d}\EE \exp \lambda_j x_j^2$, where $x_j$ is the $j$-th coordinate of $\xb$, and $x_j \sim \cN(0,1)$. Suppose that $\Pi\big|_{S^c}$ has cycle decomposition of lengths $\ell_1,\dots,\ell_r$, where $\ell_j \geq 2$ (no singletons) and $\sum_j \ell_j = n-k$. Since eigenvalues are invariant to simultaneous row and column permutation, and by orthogonal invariance of the Gaussian again, we can suppose without loss of generality that $\Pi \big|_{S^c}$ is in block diagonal form, where each of the $r$ blocks is a circulant matrix with the $j$-th block's contribution to $\Pi\big|_{S^c}+\Pi\big|_{S^c}^{\top}$ being 
\[
\begin{bmatrix}
0 & 1 &   &   &   \\
  & 0 & 1 &   &   \\
  &   & 0 & \ddots &   \\
  &   &   & \ddots & 1 \\
1 &   &   &   & 0
\end{bmatrix} + \begin{bmatrix}
0 & 1 &   &   &   \\
  & 0 & 1 &   &   \\
  &   & 0 & \ddots &   \\
  &   &   & \ddots & 1 \\
1 &   &   &   & 0
\end{bmatrix}^\top =
\begin{bmatrix}
0 & 1 &  &  & 1 \\
1 & 0 & 1 &  &  \\
 & 1 & 0 & \ddots &  \\
 &  & \ddots & \ddots & 1 \\
1 &  &  & 1 & 0
\end{bmatrix} \in \RR^{\ell_j \times \ell_j}.
\] 
Let $\omega_{\ell_j} = \exp\inparen{\frac{2\pi \bi}{\ell_j}}$ be the $\ell_j$-th root of unity. The eigenvalues of the matrix in the above display are $\big\{  \omega_{\ell_j}^{a} + \omega_{\ell_j}^{-a}  \big\}_{a=0}^{\ell_j - 1}$. For $x \sim \cN(0,1)$, we have $\EE \exp t x^2 = (1-2t)^{-1/2}$ for $t<\frac{1}{2}$. Since $\abs{1+\omega_{\ell_j}^{a}+\omega_{\ell_j}^{-a}}\leq 3$ and $\sigma^2\to\infty$, it follows that for large enough $n$,
\begin{align}\label{eq:geomMatch_overlapError_V_in_EigvalCirculant}
    V(S,\Pi) = \prod_{j=1}^{r} \prod_{a=0}^{\ell_j - 1} \inparen{ 1 - \frac{1}{\sigma^2} \inparen{1 + \omega_{\ell_j}^{a} + \omega_{\ell_j}^{-a}}  }^{-\frac{d}{2}}.
\end{align}
We now give an upper bound for $V(S,\Pi)$. For any fixed integer $\ell \geq 2$, define $\omega = \exp\inparen{\frac{2\pi \bi}{\ell}}$. Let $\xi = \frac{1}{\sigma^2}$ and we have 
\begin{align*}
G(\ell) &:= \prod_{a=0}^{\ell - 1} \inparen{ 1 - \xi(  1 + \omega^a + \omega^{-a} )  } = \prod_{a=0}^{\ell-1} \frac{-\xi \omega^{2a} + (1-\xi)\omega^{a} - \xi }{\omega^a} \\
&= (-1)^{\ell-1} \prod_{a=0}^{\ell -1} \underbrace{(-\xi z^2 + (1-\xi) z - \xi)}_{=: \, q(z)} \Big|_{z = \omega^a}.
\end{align*}
Observe that the two roots of $q(z)$ are real if $0 < \xi < 1/3$ (valid for large enough $n$). Furthermore, the two roots are reciprocals of each other: call them $\rho$ and $1/\rho$, where $\rho = (2\xi)^{-1}(1 - \xi + \sqrt{-3\xi^2 - 2\xi + 1})$ is the larger of the two. Note also $\rho > 1$ if $\xi < 1/3$. We may write $q(z) = -\xi (z-\rho)(z-\rho^{-1})$. By definition of the roots of unity, we have the identity $(-1)^\ell \prod_{a=0}^{\ell-1} (\omega^a - x) = x^\ell - 1$ for any $x \in \R$.  Thus 
\begin{align}
    G(\ell) &= (-1)^{\ell-1} \prod_{a=0}^{\ell-1} q(\omega^a) = (-1)^{\ell-1} \prod_{a=0}^{\ell-1} -\xi(\omega^a - \rho) (\omega^{a} - \rho^{-1}) \nonumber\\
    &= (-1)^{\ell-1} (-\xi)^{\ell} (\rho^{\ell} - 1) (\rho^{-\ell} - 1) = \inparen{\frac{\xi}{\rho}}^{\ell} (\rho^\ell - 1)^2. \label{eq:geomMatch_overlapError_G(ell)_finalExpression}
\end{align}
We claim the following: For $n$ large enough so that $\xi < 1/3$ and hence $\rho > 1$,
\begin{enumerate}[(i)]
    \item For $\ell \geq 4$, $G(\ell) \geq G(2) G(\ell-2)$.
    \item $G(3)^2 > G(2)^3$.
\end{enumerate}
Statement (i) reduces to showing that $\rho^\ell - 1 \geq (\rho^2 - 1)(\rho^{\ell-2} - 1)$. This further simplifies to $\rho^{\ell-2} + \rho^{2} \geq 2$ which is true since $\rho > 1$. For statement (ii), since $\rho > 1$, it suffices to show $(\rho^3 - 1)^2 > (\rho^2 - 1)^3$. This simplifies to $3\rho^4 - 2\rho^3 -3\rho^2 + 2 > 0$ where the LHS factorizes as $(\rho - 1)^2(3\rho^2 + 4\rho + 2)$ which proves the statement.

Note that 
\[
V(S,\Pi) = \exp\inbraces{ -\frac{d}{2} \sum_{j=1}^{r} \log \prod_{a=0}^{\ell_j - 1} (1-\xi(1+\omega_{\ell_j}^{a}+\omega_{\ell_j}^{-a}))  } = \exp\inbraces{ -\frac{d}{2} \sum_{j=1}^{r} \log G(\ell_j) }.
\]
Therefore, by Claims (i) and (ii), for any fixed $S$, $V(S,\Pi)$ is maximized for $\Pi$'s whose derangement part $\Pi \big|_{S^c}$ has a maximal number of two cycles. It follows that 
\begin{align}
V(S,\Pi) &\leq \exp \inparen{ -\frac{d}{2}\times \begin{cases}
    \frac{n-k}{2} \log G(2) & n-k \text{ even} \\
    \frac{n-k-3}{2} \log G(2) + \log G(3) & n-k \text{ odd} 
\end{cases} } \nonumber \\
&\leq \exp \inbraces{ - \frac{d(n-k)}{4} \log G(2)  } \nonumber \\
&= \insquare{ \inparen{1 - \frac{3}{\sigma^2}} \inparen{1 + \frac{1}{\sigma^2}}  }^{-\frac{d(n-k)}{4}}.\label{eq:geomMatch_overlapError_V_finalBound}
\end{align}
Thus, upper bounding the number of derangements of $S^c$ as $2(n-k)!/e \leq C e^{(n-k) \log n}$, and upper bounding $\binom{n}{k} \leq \inparen{\tfrac{e}{\delta}}^k$ we can upper bound term $(\text{II})$ as 
\begin{align}
    \text{(II)} &\leq \frac{1}{n \log n} \log \sum_{\substack{S \subset [n] \\ \abs{S} = k}} \sum_{\Pi : S(\Pi) = S} V(S,\Pi) \nonumber \\
    &\leq \frac{1}{n \log n} \log \inbraces{ \inparen{\frac{e}{\delta}}^k C\exp( (n-k)\log n )  \insquare{ \inparen{1 - \frac{3}{\sigma^2}} \inparen{1 + \frac{1}{\sigma^2}}  }^{-\frac{d(n-k)}{4}}  } \nonumber\\
    &\leq 1-\delta + \frac{d(n-k)}{4n\log n}\frac{2}{\sigma^2}(1+o(1)) + O\inparen{\frac{1}{\log n}} \nonumber \\
    &= \inparen{\frac{b}{2} + 1}(1-\delta) + o(1), \label{eq:geomMatch_overlapError_derangementPart}
\end{align}
where the third line follows from Taylor expanding $\log \big(1- \big(  \tfrac{2}{\sigma^2} + \tfrac{3}{\sigma^4} \big) \big) = - \tfrac{2}{\sigma^2} + O(\frac{1}{\sigma^4})$, and the last line by plugging in the definition of $\sigma^2 = \tfrac{d}{b \log n}$. The estimates above are uniform over $\delta \in [\delta_0,1]$, since $\log(e/\delta)$ is then uniformly bounded. Combining \eqref{eq:geomMatch_overlapError_diagTerm} and \eqref{eq:geomMatch_overlapError_derangementPart} gives \eqref{eq:GM_Psi_UB}.
\end{proof}

\subsection{Asymptotics of the full free energy}

This section gives the proof for Proposition \ref{prop:GM_full_free_energy_subcritical}.

By symmetry, it suffices to prove Proposition \ref{prop:GM_full_free_energy_subcritical} under $\Pi_*=I$, in which case the free energy in \eqref{eq:GM_fullFE_def} reduces exactly to $\widetilde{\Psi}_\sigma$. Thus it suffices to show that
\[
    \EE\widetilde{\Psi}_{\sigma}
    =
    1+\frac b2+o(1).
\]
The proposition will be shown by first reducing $\widetilde{\Psi}_\sigma$ as follows: define 
\begin{equation}\label{eq:GM_noiseFE_def}
    \cZ = \sum_{\Pi} \exp \inbraces{ \sqrt{b \log n} \tr \Pi \frac{\Xb \Zb^\top}{\sqrt{d}}  }
\end{equation}
This quantity $\cZ$ can be thought of as the noise part of $\widetilde{\Psi}_{\sigma}$ in \eqref{eq:GM_overlapError_in_FPpotential}. The following lemma shows that we can focus our attention on $\cZ$, because dropping the signal part in $\widetilde{\Psi}_{\sigma}$ costs only an arbitrarily small exponential factor at the $n\log n$ scale.

\begin{lemma}\label{lemma:GM_uniformSignalLowerBound}
Suppose $b > 0$ is fixed. Then for every $\eta > 0$, there exists a constant $c_{\eta,b} > 0$ such that, for all sufficiently large $n$,
\[
    \PP\insquare{ \widetilde{\Psi}_{\sigma} \geq \frac{1}{n\log n}\log \cZ - \eta } \geq 1 - \exp\inbraces{- c_{\eta,b} n d }.
\]
\end{lemma}

\begin{proof}
We first show that 
\begin{equation}\label{eq:GM_PsiTilde_UniformSignalBound}
    \PP\insquare{ \min_{\Pi} \frac{1}{\sigma^2}\tr \Pi \Xb\Xb^\top \leq - \eta n \log n } \leq \exp\inbraces{- c_{\eta,b} n d }.
\end{equation}
Fix a permutation matrix $\Pi$ and write $X^{(1)},\dots,X^{(d)} \in \RR^n$ for the columns of $\Xb$. Let $A_\Pi = \frac{1}{2}(\Pi + \Pi^\top)$.
Then
\[
    \tr \Pi\Xb\Xb^\top = \sum_{a=1}^{d} (X^{(a)})^\top A_\Pi X^{(a)}.
\]
If $g \sim \cN(0,I_{nd})$ and $B_\Pi = I_d \otimes A_\Pi$, then the right-hand side of the above display has the same distribution as $g^\top B_\Pi g$. Since $\tr B_\Pi = d\tr \Pi \geq 0$, $\norm{B_\Pi}_{\text{op}} \leq 1$, and $\norm{B_\Pi}_{F}^2 \leq nd$, the Hanson-Wright inequality for Gaussian quadratic forms (see e.g.~\cite[Theorem 6.2.1]{vershynin2018}) gives, for any $t > 0$,
\[
    \PP\insquare{ \tr \Pi\Xb\Xb^\top \leq -t }
    \leq 2\exp\inbraces{-c\min\inparen{\frac{t^2}{nd}, t}}.
\]
Taking $t = \eta nd/b$, and a union bound over at most $n! \leq \exp(n\log n)$ permutation matrices gives
\[
    \PP\insquare{ \min_{\Pi} \frac{1}{\sigma^2}\tr \Pi \Xb\Xb^\top \leq - \eta n \log n }
    \leq 2\exp\inbraces{n\log n - c'_{\eta,b}nd},
\]
for some constant $c'_{\eta,b} > 0$. Since $d = \omega(\log n)$ this proves \eqref{eq:GM_PsiTilde_UniformSignalBound} for sufficiently large $n$.

On the complementary event in \eqref{eq:GM_PsiTilde_UniformSignalBound}, for every $\Pi$,
$
    \frac{1}{\sigma^2}\tr \Pi\Xb\Xb^\top \geq -\eta n\log n.
$
Therefore,
\begin{align*}
    \sum_{\Pi} \exp\inbraces{ \frac{1}{\sigma^2}\tr \Pi\Xb\Xb^\top + \frac{1}{\sigma}\tr \Pi\Xb\Zb^\top}
    &\geq e^{-\eta n\log n}
    \sum_{\Pi} \exp\inbraces{ \frac{1}{\sigma}\tr \Pi\Xb\Zb^\top} = e^{-\eta n\log n}\cZ.
\end{align*}
Taking logarithms and dividing by $n\log n$ finishes the proof.
\end{proof}

The main work of the lower bound is therefore the asymptotics of the free energy $\cZ$. 

\begin{theorem}\label{thm:GM_noiseFE_asymptotics}
    Let $b > 0$ be fixed. Then
    \begin{equation}\label{eq:GM_noiseFE_asymptotics}
        \lim_{n \to \infty} \frac{1}{n\log n} \log \cZ =
        \begin{cases}
            \frac{b}{2} + 1, & b < 2,\\
            \sqrt{2b}, & b \geq 2,
        \end{cases}
        \qquad \text{a.s.}
    \end{equation}
\end{theorem}

The starting point of the proof is to establish that the normalized free energy concentrates on its expectation. The proof follows by straightforward modifications of \eqref{eq:GM_UpperBound_GaussianConcentration}.

\begin{lemma}\label{lemma:GM_noiseFE_concentration}
    There exist constants $C,c > 0$ such that for every fixed $t > 0$ and all sufficiently large $n$,
    \begin{equation}
        \PP\insquare{\abs{\frac{1}{n\log n}\log \cZ - \frac{1}{n\log n} \EE \log \cZ} \geq t} \leq C\exp\inbraces{-c \frac{n\log n}{b}t^2}.
    \end{equation}
\end{lemma}

By Lemma \ref{lemma:GM_noiseFE_concentration} and Borel-Cantelli, we see that to establish Theorem \ref{thm:GM_noiseFE_asymptotics}, it suffices to show that the deterministic sequence $\frac{1}{n\log n} \EE \log \cZ$ has the desired limit.
This will follow from a conditional second moment method. The high probability event is given in Appendix \ref{sec:GM_highProbEvent}. The sections afterwards make first and second moment calculations which will be used to prove Theorem \ref{thm:GM_noiseFE_asymptotics}.
The conditional second moment argument establishes the subcritical branch. The critical and supercritical branch is then deduced from it by convexity and a fractional-moment bound.

\begin{proof}[Proof of Proposition \ref{prop:GM_full_free_energy_subcritical}]
    \underline{\textbf{Lower bound.}}
    Fix $\epsilon_1 > 0$. By Theorem \ref{thm:GM_noiseFE_asymptotics},
    \[
        \PP\insquare{\frac{1}{n\log n}\log \cZ \geq \frac{b}{2} + 1 - \epsilon_1} = 1 - o(1).
    \]
    Combining this with Lemma \ref{lemma:GM_uniformSignalLowerBound}, applied with $\eta = \epsilon_1$, gives
    \[
        \PP\insquare{\widetilde{\Psi}_{\sigma} \geq \frac{b}{2} + 1 - 2\epsilon_1} = 1 - o(1).
    \]
    On the other hand, \eqref{eq:GM_UpperBound_GaussianConcentration} gives
    \[
        \PP\insquare{\widetilde{\Psi}_{\sigma} \leq \EE\widetilde{\Psi}_{\sigma} + \epsilon_1} = 1 - o(1).
    \]
    For all sufficiently large $n$, the last two events have nonempty intersection. On this intersection,
    \[
        \EE\widetilde{\Psi}_{\sigma} \geq \frac{b}{2} + 1 - 3\epsilon_1.
    \]
    Taking $\liminf_{n\to\infty}$ and then letting $\epsilon_1 \downarrow 0$ gives
    \[
        \liminf_{n\to\infty}\EE\widetilde{\Psi}_{\sigma}
        \geq
        1+\frac b2.
    \]

    \underline{\textbf{Upper bound.}}
    For $0\leq k\leq n$, let
    \[
        \widetilde Z_k
        =
        \sum_{\Pi:\,\tr\Pi=k}
        \exp\left\{
            \frac{1}{\sigma^2}\tr\Pi\Xb\Xb^\top
            +
            \frac{1}{\sigma}\tr\Pi\Xb\Zb^\top
        \right\},
        \qquad\text{and}\qquad
        \Psi_{\sigma,k}
        =
        \frac{1}{n\log n}\log\widetilde Z_k.
    \]
    The empty sum is taken to be zero, so in particular $\Psi_{\sigma,n-1} = -\infty$.
    Then
    \[
        \widetilde{\Psi}_{\sigma}
        \leq
        \max_{0\leq k\leq n}\Psi_{\sigma,k}
        +
        \frac{\log(n+1)}{n\log n}.
    \]
    Fix $\eta\in(0,1)$. For $k\geq \eta n$, Proposition \ref{prop:GM_Psi_UB} gives uniformly
    \[
        \EE\Psi_{\sigma,k}
        \leq
        1+\frac b2+\frac{k}{2n}(b-2)+o_n(1)
        \leq
        1+\frac b2-\frac{2-b}{2}\eta+o_n(1).
    \]
    For $0\leq k\leq\eta n$, applying the estimates behind \eqref{eq:geomMatch_overlapError_diagTerm} and \eqref{eq:geomMatch_overlapError_derangementPart} with $\delta=k/n$ gives, uniformly over this range,
    \[
        \EE\Psi_{\sigma,k}
        \leq
        \frac{bk}{n}
        +
        \left(1+\frac b2\right)\left(1-\frac{k}{n}\right)
        +o_n(1)
        =
        1+\frac b2-\frac{2-b}{2}\frac{k}{n}+o_n(1).
    \]
    Since $b<2$, it follows that $\sup_{0\leq k\leq\eta n}\EE\Psi_{\sigma,k}\leq 1+b/2+o_n(1)$.
    Finally, the concentration argument used for \eqref{eq:GM_UpperBound_GaussianConcentration} applies to each restricted free energy $\Psi_{\sigma,k}$ with the same constants, uniformly in $k$. A union bound over the $n+1$ possible values of $k$ gives
    \[
        \EE\max_{0\leq k\leq n}
        \left(\Psi_{\sigma,k}-\EE\Psi_{\sigma,k}\right)_+
        =
        o(1).
    \]
    Therefore
    $
        \limsup_{n\to\infty}\EE\widetilde{\Psi}_{\sigma}
        \leq
        1+\frac b2
    $.
    Combining the upper and lower bounds proves the claim.
\end{proof}
\subsubsection{High probability event}
\label{sec:GM_highProbEvent}

Let $\mathscr{P}_k$ denote the set of $k$-subpermutation matrices: any $S \in \mathscr{P}_k$ is a $n \times n$ matrix with exactly $k$ rows and $k$ columns containing one nonzero entry (which is one) and all the rest are zeros. Notice that $\mathscr{P}_n$ coincides with the usual set of $n \times n$ permutation matrices.

For any $\epsilon, c_0 > 0$, and $k_0 = c_0 n$, define the event
\begin{align}\label{eq:GM_highProbEvent_def}
    \cE = \cE(\epsilon,c_0) = \bigcap_{k=1}^{n} \cE_k, \quad\text{where}\quad \cE_k = \begin{cases}
    \displaystyle
    \max_{S \in \mathscr{P}_k} \frac{1}{\sqrt{dk}} \tr S \Xb \Zb^\top \leq \sqrt{(2+\epsilon)k_0 \log n} & \text{if } k \leq k_0 \\
    \displaystyle \max_{S \in \mathscr{P}_k} \frac{1}{\sqrt{dk}} \tr S \Xb \Zb^\top \leq \sqrt{(2+\epsilon)k \log n} & \text{if } k > k_0
\end{cases}
\end{align}

\begin{lemma}\label{lemma:GM_highProbEvent}
For any fixed $\epsilon,c_0>0$, the event $\cE=\cE(\epsilon,c_0)$ holds with high probability, i.e.~$\PP \cE = 1 - o(1)$.
\end{lemma}

\begin{proof}[Proof of Lemma \ref{lemma:GM_highProbEvent}]
    For any $k$, for any $c > 0$ such that $c/\sqrt{dk} < 1$, we have
    \begin{align*}
        \EE \max_{S \in \mathscr{P}_k} \frac{1}{\sqrt{dk}} \tr S \Xb \Zb^\top &\leq \frac{1}{c} \log \sum_{S \in \mathscr{P}_k} \EE \exp\inbraces{ \frac{c}{\sqrt{dk}} \tr S \Xb \Zb^\top } \\
        &= \frac{1}{c} \log \abs{\mathscr{P}_k} - \frac{dk}{2c} \log \inparen{1 - \frac{c^2}{dk}}.
    \end{align*}
    We have $\abs{\mathscr{P}_k} = \binom{n}{k}\binom{n}{k} k! \leq n^{2k} (e/k)^k$. Then choosing $c = \sqrt{2 \log \abs{\mathscr{P}_k}} \ll \sqrt{dk}$, we obtain after Taylor expanding the $\log$ in the second term that
    \begin{equation}\label{eq:UncondHighProb_Emax_bound}
    \EE \max_{S \in \mathscr{P}_k} \frac{1}{\sqrt{dk}} \tr S \Xb \Zb^\top \leq (1+o(1)) \sqrt{2 \log \abs{\mathscr{P}_k}}
    \end{equation}
    The concentration step below is the same high-probability Lipschitz argument used in the proof of \eqref{eq:GM_UpperBound_GaussianConcentration}; we spell it out in this simpler max setting. Define
    \[
        M_k(\Xb,\Zb)=\max_{S\in\mathscr{P}_k}\frac{1}{\sqrt{dk}}\tr S\Xb\Zb^\top .
    \]
    We claim that
    \begin{equation}\label{eq:UncondHighProb_Mk_concentration}
        \PP\insquare{\abs{M_k-\EE M_k}\geq t}
        \leq C\exp(-ct^2)+C'\exp(-c'd)
    \end{equation}
    for some universal constants $C,C',c,c'>0$ and every $t>0$. To see this, note that conditionally on $\Zb$, the function $\Xb\mapsto M_k(\Xb,\Zb)$ is $L_{\Zb}$-Lipschitz, where
    \[
        L_{\Zb}
        =
        \max_{S\in\mathscr{P}_k}\frac{\norm{S^\top\Zb}_F}{\sqrt{dk}}.
    \]
    Since $S^\top\Zb$ selects $k$ rows of $\Zb$, we have
    $
        L_{\Zb}^2\leq d^{-1}\max_{i\leq n}\norm{Z_i}_2^2.
    $
    Since $\norm{Z_i}_2^2\sim\chi_d^2$, a union bound gives $\PP\insquare{L_{\Zb}^2>2}\leq \exp(-c'd)$. Thus, by Gaussian concentration conditionally on $\Zb$,
    \[
        \PP\insquare{\abs{M_k-\EE[M_k\mid\Zb]}\geq t}
        \leq 2e^{-t^2/4}+\exp(-c'd).
    \]
    On the other hand, the function $\EE[M_k\mid\Zb]$ is $L_*$-Lipschitz, where
    \[
        L_*
        =
        \EE\insquare{\max_{S\in\mathscr{P}_k}\frac{\norm{S\Xb}_F}{\sqrt{dk}}}
        \leq
        \EE\sqrt{\frac{1}{d}\max_i\norm{X_i}_2^2}
        \leq C_0
    \]
    for a universal constant $C_0>0$. Hence
    \[
        \PP\insquare{\abs{\EE[M_k\mid\Zb]-\EE M_k}\geq t}
        \leq 2\exp(-t^2/(2C_0^2)).
    \]
    The claim then follows from the triangle inequality.

    For $k > k_0$, we have
    \begin{align*}
        \PP[\cE_k^c] &= \PP\insquare{ M_k - \EE M_k > \sqrt{(2+\epsilon)k \log n} - \EE M_k  } \\
        &\leq \PP\insquare{ M_k - \EE M_k > \inparen{\sqrt{2+\epsilon} - \sqrt{2 + \epsilon/2}}\sqrt{k \log n}  },
    \end{align*}
    where the inequality uses \eqref{eq:UncondHighProb_Emax_bound} and   that $\lim_{n \to \infty} \frac{\log \abs{\mathscr{P}_k}}{k \log n} = 1$ . Indeed, since $\abs{\mathscr{P}_k}=\binom{n}{k}^2k!$, Stirling's formula gives $\log\abs{\mathscr{P}_k}=k\log n+O(n)$ uniformly for $k>k_0=c_0n$. By \eqref{eq:UncondHighProb_Mk_concentration}, for $k>k_0$,
    \begin{align}\label{eq:UncondHighProb_Ek_complement_k_bigger_k_0}
        \PP\insquare{\cE_k^c} \leq C\exp\inparen{-c(\epsilon) k \log n}+C'\exp(-c'd)
    \end{align}

    For $k \leq k_0$, since $\log \abs{\mathscr{P}_k} \leq 2k \log n + k - k \log k =: f(k)$ where $f(k)$ is increasing for $k \leq k_0$, we have $f(k) \leq 2k_0 \log n + k_0 - k_0 \log k_0 = k_0 \log n + k_0 (1-\log c_0) \leq \inparen{1 + \frac{\epsilon}{8}}k_0 \log n$ for sufficiently large $n$. It follows that for sufficiently large $n$, from \eqref{eq:UncondHighProb_Emax_bound},
    \[
    \EE \max_{S \in \mathscr{P}_k} \frac{1}{\sqrt{dk}} \tr S \Xb \Zb^\top \leq \sqrt{ \inparen{2 + \frac{\epsilon}{2}} k_0 \log n }.
    \]
    By \eqref{eq:UncondHighProb_Mk_concentration}, we obtain 
    \begin{align}
        \PP[\cE_k^c] &\leq \PP\insquare{ M_k - \EE M_k > \inparen{\sqrt{2+\epsilon} - \sqrt{2 + \epsilon/2}}\sqrt{k_0 \log n}  } \nonumber \\
        &\leq C\exp\inparen{-c(\epsilon)k_0 \log n}+C'\exp(-c'd) \label{eq:UncondHighProb_Ek_complement_k_smaller_k_0}
    \end{align}
    A union bound combined with \eqref{eq:UncondHighProb_Ek_complement_k_bigger_k_0} and \eqref{eq:UncondHighProb_Ek_complement_k_smaller_k_0} gives, as $n\to \infty$,
    \[
    \PP \cE^c \leq C\sum_{k \leq k_0} e^{-c(\epsilon) k_0 \log n} + C\sum_{k > k_0} e^{-c(\epsilon) k \log n} + Cn e^{-c'd} \to 0. \qedhere
    \]
\end{proof}

\begin{lemma}\label{lemma:GM_highProbEvent_conditional_corrected}
Let $W = \frac{1}{\sqrt{d}}\Xb\Zb^\top$. For any fixed $\epsilon,c_0>0$ and any fixed permutation matrix $\Pi_0$, with $\cE=\cE(\epsilon,c_0)$,
\[
    \PP\insquare{\cE^c \,\bigg|\, \tr \Pi_0 W}
    \1\left\{\abs{\frac{1}{\sqrt n}\tr \Pi_0 W}\leq \sqrt{2n\log n}\right\}
    = o(1).
\]
\end{lemma}

\begin{proof}
Without loss of generality, assume $\Pi_0=I$. Let $T=\tr W$. Write $t$ for a deterministic value of $T$ and condition on $T=t$. To prove the claim, it is enough to show that uniformly over $\abs{t}\leq n\sqrt{2\log n}$,
\[
    \PP\insquare{\bigcup_{k=1}^{n}\cE_k^c \,\bigg|\, T=t}=o(1).
\]
We first show that the norm and inner products of the rows of $
\Zb$ still concentrate, even with the conditioning on $\tr W$. Fix a sufficiently large constant $K>0$, to be chosen below, and set
\[
    \rho_n = \sqrt{\frac{K\log n}{d}}.
\]
Note that $\rho_n\downarrow0$ because $d=\omega(\log n)$. Define
\[
    \cG_n =
    \inbraces{
    \max_i \abs{\frac{\norm{Z_i}_2^2}{d}-1}\leq \rho_n}
    \cap
    \inbraces{\max_{i\neq j}\frac{\abs{Z_i^\top Z_j}}{d}\leq \rho_n
    }.
\]
We claim that
\begin{equation}\label{eq:GM_conditional_Z_regular}
    \sup_{\abs{t}\leq n\sqrt{2\log n}}
    \PP\insquare{\cG_n^c \,\big|\,T=t}=o(n^{-3}).
\end{equation}
This is seen as follows. Write $N=nd$, $R=\norm{\Zb}_F^2$, and $U=\Zb/\norm{\Zb}_F$ (so $\Zb = \sqrt{R} U$). By rotational invariance, $U$ is uniform on the sphere in $\RR^N$ and is independent of $R$. Moreover, $\tr W$ has the conditional distribution
\[
    \inparen{T \,\big|\, \Zb} \sim \cN\inparen{0,\frac{R}{d}}.
\]
Notice that this distribution depends on $\Zb$ only through $R$. Therefore the conditional density of $\tr W$ at $t$, given $(R,U)$, is the same as the conditional density given $R$, namely
\[
    f_{T | R}(t \mid r) = \sqrt{\frac{d}{2\pi r}}\exp\inparen{-\frac{dt^2}{2r}}.
\]
We have $R \sim \chi^2_N$. Denote the marginal density of $R$ by $f_R$. By Bayes' formula, for any Borel set $A\subset \RR_+$ and any Borel set $B$ on the unit sphere,
\begin{equation}\label{eq:condHighProbEvent_Bayes}
    \PP\insquare{R\in A,\,U\in B\,\big|\,T=t}
    =
    \frac{\int_A f_{T | R}(t \mid r)f_R(r)\,\ud r}{\int_0^\infty f_{T | R}(t \mid r)f_R(r)\,\ud r}
    \PP\insquare{U\in B}.
\end{equation}
We deduce that conditioning on $T=t$ only influences the radius $R$ of $\Zb$, while the direction $U$ of $\Zb$ remains uniform.

We further claim that, uniformly over $\abs{t}\leq n\sqrt{2\log n}$,
\begin{equation}\label{eq:condHighProbEvent_Zb_radius_bound}
    \PP\insquare{\abs{\frac{R}{N}-1}>\frac{\rho_n}{4}\,\bigg|\, T=t}=o(n^{-3}).
\end{equation}
We prove \eqref{eq:condHighProbEvent_Zb_radius_bound} now. Define $h_t(r)=f_{T | R}(t \mid r)f_R(r)$. Then by taking $B$ in \eqref{eq:condHighProbEvent_Bayes} to be the whole unit sphere, and using the change of variables $r=N(1+s)$, we have
\begin{equation}\label{eq:condHighProbEvent_Zb_radius_bound_changeOfVar}
    \PP\insquare{\abs{\frac{R}{N}-1}>\frac{\rho_n}{4}\,\bigg|\,T=t}
    =
    \frac{
    N\int_{\abs{s}>\rho_n/4}h_t(N(1+s))\,\ud s
    }{
    N\int_{-1}^{\infty}h_t(N(1+s))\,\ud s
    }.
\end{equation}
We consider the numerator first. Here and below, the condition \(|s|>\frac{\rho_n}{4}\) is understood together with \(s>-1\), since \(r=N(1+s)>0\). We have
\begin{align*}
    \log\frac{h_t(N(1+s))}{h_t(N)}
    =
    \inparen{\frac{N}{2}-\frac32}\log(1+s)-\frac{Ns}{2}
    +\frac{t^2}{2n}\frac{s}{1+s} \leq -cNs^2+C\abs{s}n\log n,
\end{align*}
where the inequality uses $t^2/n\leq 2n\log n$, and $\log(1+s)-s\leq -c s^2$ which holds for $\abs{s} < 1/2$. Furthermore, for $\abs{s}\geq \frac{\rho_n}{4}$,
\[
    \frac{C\abs{s}n\log n}{cNs^2}
    \leq
    \frac{C\log n}{d\abs{s}}
    \leq
    C\sqrt{\frac{\log n}{Kd}}
    =o(1).
\]
Thus on the region, $\frac{\rho_n}{4}\leq \abs{s}\leq \frac{1}{2}$, for $n$ large enough,
\[
    h_t(N(1+s))\leq h_t(N)\exp\inparen{-cNs^2}.
\]
The contribution to the numerator in \eqref{eq:condHighProbEvent_Zb_radius_bound_changeOfVar} from this region is therefore at most
\begin{equation}\label{eq:condHighProbEvent_Zb_radius_bound_numContribMain}
    N h_t(N)\int_{\frac{\rho_n}{4}\leq\abs{s}\leq \frac{1}{2}}\exp\inparen{-cNs^2}\,\ud s
    \leq
    C h_t(N)\sqrt N\exp\inparen{-cN\rho_n^2}.
\end{equation}
We now consider the remaining region $\abs{s}>\frac{1}{2}$. If $-1<s<-\frac{1}{2}$, then
$
    \frac{t^2}{2n}\frac{s}{1+s}
$
is nonpositive, and so
\[
    \log\frac{h_t(N(1+s))}{h_t(N)}
    \leq
    \inparen{\frac{N}{2}-\frac{3}{2}}\log(1+s)-\frac{Ns}{2}
    \leq -cN.
\]
Therefore,
\begin{equation}\label{eq:condHighProbEvent_Zb_radius_bound_negTail}
    N\int_{-1}^{-\frac{1}{2}}h_t(N(1+s))\,\ud s
    \leq
    C N h_t(N)\exp\inparen{-cN}.
\end{equation}
If $s>\frac{1}{2}$, then $s-\log(1+s)\geq c$, and so
\[
    \log\frac{h_t(N(1+s))}{h_t(N)}
    \leq -cN-\frac{3}{2}\log(1+s)+Cn\log n
    \leq -cN-\frac{3}{2}\log(1+s),
\]
where the last inequality uses $N=nd$ and $d=\omega(\log n)$. Thus, for $s>\frac{1}{2}$,
\[
    h_t(N(1+s))
    \leq
    h_t(N)\exp\inparen{-cN}(1+s)^{-\frac{3}{2}},
\]
and hence
\begin{equation}\label{eq:condHighProbEvent_Zb_radius_bound_posTail}
    N\int_{\frac{1}{2}}^\infty h_t(N(1+s))\,\ud s
    \leq
    C N h_t(N)\exp\inparen{-cN}.
\end{equation}
Combining \eqref{eq:condHighProbEvent_Zb_radius_bound_negTail} and \eqref{eq:condHighProbEvent_Zb_radius_bound_posTail}, we see that the contribution from the region $\abs{s}>\frac{1}{2}$ to the numerator in \eqref{eq:condHighProbEvent_Zb_radius_bound_changeOfVar} is negligible compared with \eqref{eq:condHighProbEvent_Zb_radius_bound_numContribMain}, since $\rho_n^2=o(1)$. Together with \eqref{eq:condHighProbEvent_Zb_radius_bound_numContribMain}, this gives the full numerator bound
\begin{equation}\label{eq:condHighProbEvent_Zb_radius_bound_numContrib}
    N\int_{\abs{s}>\rho_n/4}h_t(N(1+s))\,\ud s
    \leq
    C h_t(N)\sqrt N\exp\inparen{-cN\rho_n^2}.
\end{equation}

For the denominator, we instead lower bound the integral by restricting to $0\leq s\leq N^{-1/2}$. On this interval, by similar calculations, we have
\begin{equation}\label{eq:condHighProbEvent_Zb_radius_bound_denom}
    N\int_{-1}^{\infty}h_t(N(1+s))\,\ud s
    \geq
    c h_t(N)\sqrt N.
\end{equation}
Combining the numerator \eqref{eq:condHighProbEvent_Zb_radius_bound_numContrib} and denominator \eqref{eq:condHighProbEvent_Zb_radius_bound_denom} bounds gives the claimed \eqref{eq:condHighProbEvent_Zb_radius_bound}:
\[
    \PP\insquare{\abs{\frac{R}{N}-1}>\frac{\rho_n}{4}\,\bigg|\,T=t}
    \leq C\exp\inparen{-c\rho_n^2N}
    =o(n^{-3}).
\]

This size of $\norm{\Zb}_F^2$ estimate \eqref{eq:condHighProbEvent_Zb_radius_bound} will now be used to show \eqref{eq:GM_conditional_Z_regular}.

Since $U$ remains uniform after conditioning on $T=t$, we may realize the conditional law of $\Zb$ given $T=t$ as follows: draw $R$ from its conditional law given $T=t$, draw $\Gb = (G_1,\ldots,G_n)$ independently, where each $G_i\in\RR^d$ has i.i.d.~$\cN(0,1)$ entries, and set
\[
    \Zb=\sqrt R\,\frac{\Gb}{\norm{\Gb}_F}.
\]
Under this coupling, the following identities hold:
\[
    \frac{\norm{Z_i}_2^2}{d}
    =
    \frac{R}{nd}\cdot
    \frac{\norm{G_i}_2^2/d}{\norm{\Gb}_F^2/(nd)},
    \qquad
    \frac{\abs{Z_i^\top Z_j}}{d}
    =
    \frac{R}{nd}\cdot
    \frac{\abs{G_i^\top G_j}/d}{\norm{\Gb}_F^2/(nd)}.
\]
Claim \eqref{eq:condHighProbEvent_Zb_radius_bound} controls the prefactor $R/(nd)$. It remains to control the two ratios involving $\Gb$.
A standard chi-square tail bound gives, for $0<u<1$,
$
    \PP\insquare{\abs{\frac{\chi_m^2}{m}-1}>u}\leq 2\exp\inparen{-cmu^2}.
$
Applying this with $m=d$ to each row and with $m=nd$ to $\norm{\Gb}_F^2$, and using the Gaussian inner-product bound
\[
    \PP\insquare{\abs{G_i^\top G_j}>ud}\leq C\exp\inparen{-cu^2d},
    \qquad i\neq j,
\]
we obtain, after a union bound over $i,j$,
\[
    \PP\insquare{
    \max_i\abs{\frac{\norm{G_i}_2^2}{d}-1}>\frac{\rho_n}{8}
    \text{ or }
    \abs{\frac{\norm{\Gb}_F^2}{nd}-1}>\frac{\rho_n}{8}
    \text{ or }
    \max_{i\neq j}\frac{\abs{G_i^\top G_j}}{d}>\frac{\rho_n}{8}
    }
    \leq Cn^2\exp\inparen{-c\rho_n^2d}=o(n^{-3}),
\]
where the last step follows by choosing $K$ large enough. Altogether, Claim \eqref{eq:condHighProbEvent_Zb_radius_bound}, the last display and the identities above imply $\abs{\norm{Z_i}_2^2/d-1}\leq \rho_n$ for all $i$ and $\abs{Z_i^\top Z_j}/d\leq \rho_n$ for all $i\neq j$, except on an event with conditional probability $o(n^{-3})$. This is exactly $\PP[\cG_n^c\mid T=t]=o(n^{-3})$, proving \eqref{eq:GM_conditional_Z_regular}.

For $\Pi\in \mathscr{P}_k$, call $i$ a fixed point of $\Pi$ if $\Pi_{ii}=1$. Let $S(\Pi)=\{i:\Pi_{ii}=1\}$ be its set of fixed points. We now decompose each $\cE_k$ according to $\abs{S(\Pi)}$ . For $0\leq j\leq k$, define
\[
    \cE_{k,j}=
    \inbraces{
    \max_{\Pi\in\mathscr{P}_k:\abs{S(\Pi)}=j}
    \frac{1}{\sqrt{k}}\tr \Pi W
    \leq \sqrt{(2+\epsilon)(k\vee k_0)\log n}
    }.
\]
Then $\cE_k=\cap_{j=0}^{k}\cE_{k,j}$. We have the decomposition
\begin{equation}\label{eq:GM_conditional_kj_split}
    \max_{\Pi\in\mathscr{P}_k:\abs{S(\Pi)}=j}\frac{1}{\sqrt{k}}\tr \Pi W
    =
    \max_{S\subset[n]:\abs{S}=j}
    \inbraces{
    \frac{1}{\sqrt{k}}\sum_{i\in S}W_{ii}
    +
    \max_{\Pi'\in D_{[n]\backslash S}}\frac{1}{\sqrt{k}}\tr \Pi'W
    },
\end{equation}
where $D_{[n]\backslash S}$ is the set of $(k-j)$-subpermutation matrices supported on $[n]\backslash S$ with no fixed points.
When $j=0$ the fixed-point term is absent, and when $j=k$ the derangement term is absent.

We first consider the fixed point part in \eqref{eq:GM_conditional_kj_split}. For $S\subset[n]$ with $\abs{S}=j$, decompose
\[
    \sum_{i\in S}W_{ii}
    =
    \underbrace{\frac{\norm{\Zb_S}_F^2}{\norm{\Zb}_F^2}T}_{=:\mu_S}
    +
    \underbrace{\sum_{i\in S}W_{ii}-\frac{\norm{\Zb_S}_F^2}{\norm{\Zb}_F^2}T}_{=:R_S}.
\]
Conditionally on $\Zb$, a calculation shows
\begin{equation}\label{eq:GM_conditional_fixed_point_covariances}
    \Var\insquare{\sum_{i\in S}W_{ii}\,\bigg|\,\Zb}
    =
    \frac{\norm{\Zb_S}_F^2}{d},
    \qquad
    \Cov\insquare{\sum_{i\in S}W_{ii},T\,\bigg|\,\Zb}
    =
    \frac{\norm{\Zb_S}_F^2}{d}.
\end{equation}
Also $\Var[T\mid \Zb]=\norm{\Zb}_F^2/d$. Hence $R_S$ and $\mu_S$ are uncorrelated conditionally on $\Zb$; since they are jointly Gaussian conditionally on $\Zb$, they are independent. Conditioning further on $T$ fixes $\mu_S$, and as $S$ varies the variables $R_S$ form a centered Gaussian process. From \eqref{eq:GM_conditional_fixed_point_covariances},
\[
    \Var\insquare{R_S\,\big|\,T,\Zb}
    =
    \frac{1}{d}
    \inparen{
    \norm{\Zb_S}_F^2
    -
    \frac{\norm{\Zb_S}_F^4}{\norm{\Zb}_F^2}
    }
    =
    \frac{\norm{\Zb_S}_F^2\norm{\Zb_{S^c}}_F^2}{d\norm{\Zb}_F^2}.
\]
Equivalently, 
\[
    \inparen{\sum_{i\in S}W_{ii}\,\bigg|\,T=t,\Zb}
    \sim
    \cN\inparen{
    \frac{\norm{\Zb_S}_F^2}{\norm{\Zb}_F^2}t,
    \frac{\norm{\Zb_S}_F^2\norm{\Zb_{S^c}}_F^2}{d\norm{\Zb}_F^2}
    }.
\]
	On $\cG_n$, we have $\norm{\Zb_S}_F^2\leq (1+o(1))jd$, $\norm{\Zb_{S^c}}_F^2\leq (1+o(1))(n-j)d$, and $\norm{\Zb}_F^2\geq (1-o(1))nd$. Thus
	\[
	    \1_{\cG_n}
	    \sup_{\abs{S}=j}\Var\insquare{\frac{R_S}{\sqrt{k}}\,\bigg|\,T,\Zb}
	    \leq
	    (1+o(1))\frac{j}{k}\frac{n-j}{n},
	\]
	and therefore
	\[
	    \1_{\cG_n}
	    \EE\insquare{\max_{\abs{S}=j}\frac{R_S}{\sqrt{k}}\,\bigg|\,T,\Zb}
	    \leq
	    (1+o(1))\sqrt{\frac{j}{k}\frac{n-j}{n}}\sqrt{2(\log 2)n}.
\]

Fix $\epsilon'>0$ sufficiently small, depending only on $\epsilon$. Let
\[
    A_{k,j}=
    \inbraces{
    \max_{\abs{S}=j}\frac{R_S}{\sqrt{k}}
    \leq
    \sqrt{\frac{j}{k}\frac{n-j}{n}}\sqrt{(2+\epsilon')(\log 2)n}
    },
    \qquad 1\leq j\leq k.
\]

	Gaussian concentration applied conditionally on $T$ and $\Zb$  gives, uniformly in $k,j$ and $\abs{t}\leq n\sqrt{2\log n}$,
	\[
	    \1_{\cG_n}
	    \PP\insquare{A_{k,j}^c\,\big|\,T=t,\Zb}
	    \leq
	    Ce^{-cn}.
	\]
	Integrating over $\Zb$ and using \eqref{eq:GM_conditional_Z_regular} gives
	\begin{equation}\label{eq:GM_conditional_Akj_bound}
	    \PP\insquare{A_{k,j}^c\,\big|\,T=t}
	    \leq o(n^{-3})+Ce^{-cn}.
	\end{equation}

We next consider the derangement part in \eqref{eq:GM_conditional_kj_split}. For $\Pi'\in D_{[n]\backslash S}$, writing $(i,\ell)\in\Pi'$ to mean that $(i,\ell)$ is a matched pair, i.e.~$\Pi'_{i\ell}=1$, decompose
\[
    \tr \Pi'W
    =
    \underbrace{
    \frac{\sum_{(i,\ell)\in \Pi'}Z_i^\top Z_\ell}{\norm{\Zb}_F^2}T
    }_{=:\mu'(\Pi')}
    +
    \underbrace{
    \tr \Pi'W-\frac{\sum_{(i,\ell)\in \Pi'}Z_i^\top Z_\ell}{\norm{\Zb}_F^2}T
    }_{=:R'(\Pi')}.
\]
Conditionally on $(T,\Zb)$, we claim that the variables $\inbraces{R'(\Pi')}_{S:\abs{S}=j,\ \Pi'\in D_{[n]\backslash S}}$ form a centered Gaussian process:
\[
    \inparen{\tr \Pi'W\,\big|\,T=t,\Zb}
    \sim
    \cN\!\inparen{
    \frac{\sum_{(i,\ell)\in \Pi'}Z_i^\top Z_\ell}{\norm{\Zb}_F^2}t,
    \frac{1}{d}
    \inparen{
    \sum_{(i,\ell)\in \Pi'}\norm{Z_i}_2^2
    -
    \frac{\inparen{\sum_{(i,\ell)\in \Pi'}Z_i^\top Z_\ell}^2}{\norm{\Zb}_F^2}
    }
    }.
\]
To see this, for fixed $S$ and $\Pi'$, conditionally on $\Zb$, compute
\[
    \Var\insquare{\tr \Pi'W\,\bigg|\,\Zb}
    =
    \frac{1}{d}\sum_{(i,\ell)\in \Pi'}\norm{Z_i}_2^2,
    \qquad
    \Cov\insquare{\tr \Pi'W,T\,\bigg|\,\Zb}
    =
    \frac{1}{d}\sum_{(i,\ell)\in \Pi'}Z_i^\top Z_\ell.
\]
Therefore
\[
    \Var\insquare{R'(\Pi')\,\big|\,T,\Zb}
    =
    \frac{1}{d}
    \inparen{
    \sum_{(i,\ell)\in \Pi'}\norm{Z_i}_2^2
    -
    \frac{\inparen{\sum_{(i,\ell)\in \Pi'}Z_i^\top Z_\ell}^2}{\norm{\Zb}_F^2}
    }.
\]
On $\cG_n$, since $\Pi'$ has $k-j$ nonzero entries, we thus obtain 
\begin{align*}
    \1_{\cG_n}
    \sup_{S,\Pi'}\Var\insquare{\frac{R'(\Pi')}{\sqrt{k}}\,\bigg|\,T,\Zb}
    &\leq
    \frac{\1_{\cG_n}}{dk}\sup_{S,\Pi'}\sum_{(i,\ell)\in \Pi'}\norm{Z_i}_2^2 \leq
    (1+o(1))\frac{k-j}{k}.
\end{align*}
The number of choices of $S$ and $\Pi'$ is at most
\[
    \binom{n}{j}\binom{n-j}{k-j}^2(k-j)!
    \leq
    4^n n^{k-j}
    =
    \exp\inbraces{2(\log 2)n+(k-j)\log n}.
\]
Let
\[
    B_{k,j}=
    \inbraces{
    \max_{\abs{S}=j}\max_{\Pi'\in D_{[n]\backslash S}}
    \frac{R'(\Pi')}{\sqrt{k}}
    \leq
    \sqrt{\frac{k-j}{k}}
    \inparen{\sqrt{(4\log 2+\epsilon')n}+\sqrt{2(k-j)\log n}}
    },
    \qquad 0\leq j\leq k-1.
\]
Again by Gaussian concentration conditionally on $T$ and $\Zb$,
\[
    \1_{\cG_n}
    \PP\insquare{B_{k,j}^c\,\big|\,T=t,\Zb}
    \leq
    Ce^{-cn}.
\]
Integrating over $\Zb$ and using \eqref{eq:GM_conditional_Z_regular},
\begin{equation}\label{eq:GM_conditional_Bkj_bound}
    \PP\insquare{B_{k,j}^c\,\big|\,T=t}
    \leq o(n^{-3})+Ce^{-cn},
\end{equation}
uniformly in $k,j$ and $\abs{t}\leq n\sqrt{2\log n}$.

We now combine the two parts. On $\cG_n$ and $\abs{t}\leq n\sqrt{2\log n}$,
\[
    \max_{\abs{S}=j}\frac{\mu_S}{\sqrt{k}}
    \leq
    (1+o(1))\frac{j}{\sqrt{k}}\sqrt{2\log n},
\]
and
\[
    \max_{\abs{S}=j}\max_{\Pi'\in D_{[n]\backslash S}}
    \frac{\mu'(\Pi')}{\sqrt{k}}
    \leq
    o(1)\frac{k-j}{\sqrt{k}}\sqrt{2\log n}.
\]
Therefore, on $\cG_n$, together with $A_{k,j}$ when $j\geq 1$ and $B_{k,j}$ when $j\leq k-1$, the split \eqref{eq:GM_conditional_kj_split} gives 
\begin{align*}
    \max_{\Pi\in\mathscr{P}_k:\abs{S(\Pi)}=j}\frac{1}{\sqrt{k}}\tr \Pi W
    &\leq
    (1+o(1))\frac{j}{\sqrt{k}}\sqrt{2\log n}
    +\sqrt{\frac{j}{k}\frac{n-j}{n}}\sqrt{(2+\epsilon')(\log 2)n} \\
    &\quad
    +o(1)\frac{k-j}{\sqrt{k}}\sqrt{2\log n}
    +\sqrt{\frac{k-j}{k}}
    \inparen{\sqrt{(4\log 2+\epsilon')n}+\sqrt{2(k-j)\log n}} \\
    &\leq
    (1+o(1))\sqrt{2k\log n}+O_{\epsilon'}(\sqrt n).
\end{align*}
If $k>k_0$, this is at most $\sqrt{(2+\epsilon)k\log n}$ for sufficiently large $n$. If $k\leq k_0$, it is at most $\sqrt{(2+\epsilon)k_0\log n}$ for sufficiently large $n$. Hence $\cE_{k,j}$ holds for every $k,j$ on
\[
    \cG_n
    \cap\bigcap_{k=1}^{n}\bigcap_{j=1}^{k}A_{k,j}
    \cap\bigcap_{k=1}^{n}\bigcap_{j=0}^{k-1}B_{k,j}.
\]
Finally, by \eqref{eq:GM_conditional_Z_regular}, \eqref{eq:GM_conditional_Akj_bound}, and \eqref{eq:GM_conditional_Bkj_bound},
\begin{align*}
    \PP\insquare{\cE^c\,\big|\,T=t}
    &\leq
    \PP\insquare{\cG_n^c\,\big|\,T=t}
    +\sum_{k=1}^{n}\sum_{j=1}^{k}\PP\insquare{A_{k,j}^c\,\big|\,T=t}
    +\sum_{k=1}^{n}\sum_{j=0}^{k-1}\PP\insquare{B_{k,j}^c\,\big|\,T=t} \\
    &\leq o(n^{-3})+n^2o(n^{-3})+Cn^2e^{-cn}
    =o(1),
\end{align*}

uniformly over $\abs{t}\leq n\sqrt{2\log n}$.
\end{proof}

\subsubsection{First moment computations}

We make conditional first moment computations in this section. Recall the free energy $\cZ$ defined in \eqref{eq:GM_noiseFE_def} and the high probability event $\cE$ defined in \eqref{eq:GM_highProbEvent_def}.

\begin{lemma}\label{lemma:GM_logEZ1cE_UB}
    For any $b < 2$ and any fixed $\epsilon,c_0>0$, with $\cE=\cE(\epsilon,c_0)$,
    \begin{equation}\label{eq:firstMoment_epsilon_upperBound}
        \limsup_{n \to \infty} \frac{1}{n \log n} \log \EE[\cZ \boldsymbol{1}_{\cE}] \leq \frac{b}{2} + 1.
    \end{equation}
\end{lemma}

\begin{proof}
    By dropping the indicator, we have since $(b \log n)/d < 1$ for large enough $n$,
    \begin{align*}
        \EE[\cZ \boldsymbol{1}_{\cE}] &= \sum_{\Pi}\EE\insquare{\exp \inparen{ \sqrt{b\log n} \tr \Pi \frac{\Xb \Zb^\top}{\sqrt{d}} } \boldsymbol{1}_{\cE} } 
        \leq n! \EE\insquare{ \exp \inparen{ \sqrt{b\log n} \tr \frac{\Xb \Zb^\top}{\sqrt{d}} }   } \\
        &= n! \inparen{1 - \frac{b \log n}{d}}^{-\frac{dn}{2}} = n! \exp\inbraces{ \frac{b}{2}n\log n + o(n \log n)   },
    \end{align*}
    where we used $\log (1-x) = -x + O(x^2)$ and $d = \omega(\log n)$. Thus \eqref{eq:firstMoment_epsilon_upperBound} holds. 
\end{proof}

Our next result is a matching lower bound for the conditional first moment.

\begin{lemma}\label{lemma:GM_logEZ1cE_LB}
    For any $b < 2$ and any fixed $\epsilon,c_0>0$, with $\cE=\cE(\epsilon,c_0)$,
    \begin{equation}\label{eq:GM_logEZ1cE_LB}
        \liminf_{n \to \infty} \frac{1}{n \log n} \log \EE[\cZ \boldsymbol{1}_{\cE}] \geq \frac{b}{2} + 1.
    \end{equation}
\end{lemma}

\begin{proof}
    For a fixed permutation matrix $\Pi$, define the event
    $
        E(\Pi) = \inbraces{\abs{\frac{1}{\sqrt{nd}} \tr \Pi \Xb \Zb^\top} \leq \sqrt{2n \log n} }.
    $
     Then
    \begin{align*}
        \EE[\cZ \boldsymbol{1}_{\cE}] &= \sum_{\Pi} \EE \insquare{ \exp\inparen{ \sqrt{ bn\log n} \frac{\tr \Pi \Xb \Zb^\top}{\sqrt{nd}}  } \boldsymbol{1}_{\cE} } \\
        &\geq \sum_{\Pi} \EE \insquare{ \exp\inparen{ \sqrt{ bn\log n} \frac{\tr \Pi \Xb \Zb^\top}{\sqrt{nd}}}  \boldsymbol{1}_{E(\Pi)}   \boldsymbol{1}_{\cE} } \\
        &= \sum_{\Pi} \EE \insquare{ \exp\inparen{ \sqrt{ bn\log n} \frac{\tr \Pi \Xb \Zb^\top}{\sqrt{nd}}}  \boldsymbol{1}_{E(\Pi)} \PP\big[ \cE \,\big|\, \tr \Pi \Xb \Zb^\top  \big] } \\
        &= (1-o(1)) \sum_{\Pi} \EE \insquare{ \exp\inparen{ \sqrt{ bn\log n} \frac{\tr \Pi \Xb \Zb^\top}{\sqrt{nd}}}  \boldsymbol{1}_{E(\Pi)} } \\
        &= (1-o(1)) \ n! \ \EE \insquare{  \exp\inparen{ \sqrt{ bn\log n} \frac{\tr  \Xb \Zb^\top}{\sqrt{nd}} } \boldsymbol{1}\!\inbraces{ \abs{\frac{\tr \Xb \Zb^\top}{\sqrt{nd}}}  \leq \sqrt{2n \log n}  } },
    \end{align*}
    where the last two lines  follows from Lemma \ref{lemma:GM_highProbEvent_conditional_corrected} and the fact that $\tr \Pi \Xb \Zb^\top$ has the same distribution for all $\Pi$. Abbreviate
    \[
        V = \frac{\tr \Xb \Zb^\top}{\sqrt{nd}}, \qquad \lambda = \sqrt{bn\log n}, \qquad t = \sqrt{2n\log n}.
    \]
We claim that for $0 \leq b < 2$
\begin{align}
    &\EE\insquare{ \exp (\lambda V) \boldsymbol{1}\!\inbraces{ \abs{V} \leq t  }  } = \exp\inbraces{\frac{b}{2}n\log n+o(n\log n)}.
    \label{eq:FM_epsilon_LB_truncatedMGF}
\end{align}
The desired \eqref{eq:GM_logEZ1cE_LB} follows from this claim and the approximation $n! = \exp(n \log n + o(n \log n))$.

    We now verify the claim. Let $R = \norm{\Xb}_F^2/(nd)$. Conditionally on $\Xb$, we have $V \sim \cN(0,R)$. Note $ndR \sim \chi^2_{nd}$. Let 
    $
        a \coloneqq \frac{\lambda^2}{nd} = \frac{b\log n}{d} = o(1).
    $
    The \emph{untruncated} moment generating function is
    \begin{align}\label{eq:GM_firstMoment_untruncatedMGF}
        M(\lambda) \coloneqq \EE[e^{\lambda V}]
        = \EE\exp\inparen{\frac{\lambda^2R}{2}}
        = (1-a)^{-\frac{nd}{2}}
        = \exp\inparen{\frac{b}{2}n\log n+o(n\log n)}.
    \end{align}

    Consider a tilted law $\widetilde{\EE}$ of $R$, tilted by the factor $\exp\inbraces{ \frac{\lambda^2 R}{2}  } / M(\lambda)$; i.e.~for any function $f$, $\widetilde{\EE}f(R) = \EE\insquare{f(R) \exp\inbraces{ \frac{\lambda^2 R}{2}  } / M(\lambda)}$. Thus the left-hand side of \eqref{eq:FM_epsilon_LB_truncatedMGF} can be written as 
    \begin{align}
        \EE\insquare{e^{\lambda V}\1\!\inbraces{\abs{V} \leq t}}
        &= \EE\insquare{\exp\inparen{\frac{\lambda^2R}{2}}
        \inbraces{
        \Phi\inparen{\frac{t-\lambda R}{\sqrt{R}}}
        -
        \Phi\inparen{\frac{-t-\lambda R}{\sqrt{R}}}
        }} \nonumber\\
        &= M(\lambda) \cdot \widetilde{\EE}\insquare{
        \Phi\inparen{ \frac{t - \lambda R}{\sqrt{R}} }
        -
        \Phi\inparen{ \frac{-t - \lambda R}{\sqrt{R}} }
        }, \label{eq:GM_firstMoment_truncatedMGF_tiltedR}
    \end{align}
    where $\Phi$ is the standard Gaussian CDF, and where the first equality follows from standard Gaussian truncated moment generating function identities.
    
   Under this tilted law, $R$ follows the scaled chi-squared distribution
    \[
        R \sim \frac{1}{1-a}\frac{\chi^2_{nd}}{nd}.
    \]
    By concentration, a random variable  $Y \sim \chi^2_{nd}/(nd)$ satisfies $Y = 1 + o_{\PP}(1)$. It follows that $R = 1 + o_{\widetilde \PP}(1)$, where $\widetilde{\PP}$ refers to the tilted law. Then for any fixed $0 \leq b < 2$,
    \[
        \frac{t-\lambda R}{\sqrt{R}}
        = \inparen{\sqrt{2}-\sqrt{b}+o_{\widetilde{\PP}}(1)}\sqrt{n\log n}
        \longrightarrow +\infty.
    \]
    Also $\frac{-t-\lambda R}{\sqrt R}\to -\infty$, so the second CDF term in \eqref{eq:GM_firstMoment_truncatedMGF_tiltedR} is $o(1)$. Therefore the tilted expectation in \eqref{eq:GM_firstMoment_truncatedMGF_tiltedR} is $1-o(1)$. Together with \eqref{eq:GM_firstMoment_untruncatedMGF}, this finishes the proof.
\end{proof}

\subsubsection{Second moment}

The main result of this section is a bound on the conditional second moment of $\cZ$.

\begin{lemma}\label{lemma:GM_truncatedSecondMoment}
For any $b < 2$ and any $\zeta>0$, there exist $\epsilon,c_0>0$ such that, with $\cE=\cE(\epsilon,c_0)$ defined in \eqref{eq:GM_highProbEvent_def}, 

\begin{equation}\label{eq:GM_truncatedSecondMoment}
    \limsup_{n \to \infty} \frac{1}{n \log n} \log \EE[\cZ^2 \1_{\cE}] \leq b + 2+\zeta.
\end{equation}
\end{lemma}

\begin{proof}[Proof of Lemma \ref{lemma:GM_truncatedSecondMoment}]

\underline{\textbf{Case:}} $0 < b \leq 1$. In this regime we can drop the conditioning on the high probability event $\cE$. We have 
\begin{align}
    \EE[\cZ^2 \1_{\cE}] &\leq \EE{\cZ^2 } = \sum_{\Pi} \sum_{\Pi'} \EE \exp \inparen{ \sqrt{b \log n} \tr (\Pi + \Pi') \frac{\Xb \Zb^\top}{\sqrt{d}}  } \nonumber\\
    &= n! \sum_{\Pi} \EE \exp \inparen{ \sqrt{b \log n} \tr (I_n + \Pi) \frac{\Xb \Zb^\top}{\sqrt{d}}  } \nonumber\\
    &= n! \sum_{k=0}^{n} \sum_{\Pi : \abs{S(\Pi)} = k} \EE \exp \inbraces{  \frac{1}{\sigma} \cdot \zb^\top \insquare{(I_n + \Pi) \otimes I_{d} } \xb  }, \label{eq:geomMatch_UB_EZsquare_casebLessThanEqual1}
\end{align}
where $S(\Pi)\subset [n]$ is the set of fixed points of $\Pi$, and
where $\xb = \text{vec}(\Xb) \in \RR^{nd}$ and $\zb = \text{vec}(\Zb) \in \RR^{nd}$ be the vectorizations of $\Xb$ and $\Zb$, and where we recall $\sigma^2 = \frac{d}{b \log n}$. Fix $k$ and a $\Pi$ with $\abs{S(\Pi)} = k$. Conditionally on $\xb$,
\[
\EE\insquare{ \exp \inbraces{  \frac{1}{\sigma} \cdot \zb^\top \insquare{(I_n + \Pi) \otimes I_{d} } \xb  } \,\bigg|\, \xb } = \exp \inbraces{ \frac{1}{2\sigma^2} \xb^\top \insquare{(2I_n + \Pi + \Pi^\top) \otimes I_d} \xb }.
\] 
Let $S^c = [n]\setminus S$ be the set of non-fixed points of $\Pi$. Let $\xb_{S} \in \RR^{kd}$ and $\xb_{S^c}$ denote the corresponding fixed point and derangement sub-vectors of $\xb$. Taking expectation over $\xb$ in the above display, by independence between $\xb_{S}$ and $\xb_{S^c}$, we have
\begin{align}
&\EE\insquare{ \exp \inbraces{  \frac{1}{\sigma} \cdot \zb^\top \insquare{(I_n + \Pi) \otimes I_{d} } \xb  }  } = (A) \times (B), \label{eq:truncatedSecondMoment_(A)_and_(B)}
\end{align}
where
\[
(A) = \EE\insquare{ \exp\inbraces{  \frac{1}{2\sigma^2} \xb_S^\top [4 I_k\otimes I_d] \xb_S }  }, \quad\text{and}\quad (B) = \EE\insquare{ \exp\inbraces{  \frac{1}{2\sigma^2} \xb_{S^c}^\top \insquare{\inparen{2 I_{n-k} + \Pi\big|_{S^c} + \Pi\big|_{S^c}^\top } \otimes I_d }\xb_{S^c}  }  }.
\]
By evaluating the $\chi^2$ moment generating function, we obtain 
\[
(A) = \EE \insquare{\exp \frac{2}{\sigma^2} \xb_S^\top \xb_S} = \inparen{ 1 - \frac{4}{\sigma^2}  }^{-\frac{kd}{2}}.
\]
Now for term $(B)$. Suppose that $\Pi\big|_{S^c}$ has cycle decomposition of lengths $\ell_1,\dots,\ell_r$, where $\sum_j \ell_j = n-k$ and $\ell_j \geq 2$. Let $\omega_{\ell_j} = \exp \inparen{ \frac{2\pi \ib}{\ell_j}  }$ be the $\ell_j$-th root of unity. Then we obtain 
\[
(B) = \prod_{j=1}^{r} \prod_{a=0}^{\ell_j - 1} \inparen{ 1 - \frac{1}{\sigma^2} (2 + \omega_{\ell_j}^a + \omega_{\ell_j}^{-a} )  }^{-\frac{d}{2}} \leq \inparen{1 - \frac{4}{\sigma^2}}^{-\frac{d(n-k)}{4}},
\]
where the equality follows as in \eqref{eq:geomMatch_overlapError_V_in_EigvalCirculant}, and the inequality by similar arguments as in \eqref{eq:geomMatch_overlapError_V_finalBound} (where $(B)$ is maximized when the derangement part of $\Pi$ contains a maximal number of two-cycles).

Therefore, combining the bounds for $(A)$ and $(B)$, we have from \eqref{eq:geomMatch_UB_EZsquare_casebLessThanEqual1} that
\begin{align}
    \EE[\cZ^2 \1_{\cE}] &\leq (n!) \sum_{k=0}^{n} \binom{n}{k} (n-k)! \max_{\Pi : \abs{S(\Pi)} = k} \EE \exp \inbraces{  \frac{1}{\sigma} \cdot \zb^\top \insquare{(I_n + \Pi) \otimes I_{d} } \xb  } \nonumber\\
    &\leq (n!)^2 \sum_{k=0}^{n} \frac{1}{k!} \inparen{1 - \frac{4}{\sigma^2}}^{-\frac{d(n+k)}{4}} \nonumber\\
    &\leq (n!)^2 \sum_{k=0}^{n} \exp\inbraces{ - \frac{d(n+k)}{4} \inparen{ -\frac{4b\log n}{d}  +  O\!\inparen{\frac{(\log n)^2}{d^2}} } + k - k \log k  } \nonumber\\
    &= (n!)^2  \exp\inparen{ b n \log n +  O\!\inparen{\frac{n (\log n)^2}{d}} } \sum_{k=0}^{n} \exp\inbraces{  b k \log n + k - k \log k  }.\label{eq:GM_secondMoment_untruncated_sum_bound}
\end{align} 
where in the third line we Taylor expanded $\log (1 - \tfrac{4}{\sigma^2})$ and used $\tfrac{1}{k!} \leq (\tfrac{e}{k})^{k}$. The function $k \mapsto b k \log n + k - k \log k$ is concave and maximized at $k \sim n^b$. The summation above is hence bounded, for $b \leq 1$, by 
\[
\sum_{k=0}^{n} \exp\inbraces{  b k \log n + k - k \log k  } \leq n \exp\inbraces{  b n^b \log n + n^b - n^b \log n^b  } = \exp(o (n \log n)).
\]
Since also $d \gg \log n$, we conclude that 
\[
\EE[\cZ^2 \1_{\cE}] \leq \exp\inparen{ (b + 2)n\log n + o(n \log n)   },
\]
which yields \eqref{eq:GM_truncatedSecondMoment} as desired.

\medskip
\noindent \underline{\textbf{Case:}} $1 < b < 2$.
For such a fixed $b$, choose $0 < \epsilon < 2$ sufficiently small so that
\begin{equation}\label{eq:SM_gamma_b_epsilon}
    \gamma_{b,\epsilon}
    \coloneqq
    -b + 2\sqrt{b}\sqrt{2+\epsilon}
    - \frac{2+\epsilon}{2}
    < 1.
\end{equation}
Such a choice is possible since
$
    \gamma_{b,0}
    =
    -b + 2\sqrt{2b} - 1
    =
    1 - (\sqrt{2}-\sqrt{b})^2
    < 1
$
for every fixed $b<2$. Remark that this is the same $\epsilon$ appearing in \eqref{eq:GM_highProbEvent_def}, and that this is the only point in proof that dictates the choice of this parameter.

Similarly to \eqref{eq:geomMatch_UB_EZsquare_casebLessThanEqual1}, except we do not drop the indicator function, we have
\begin{align}\label{eq:GM_secondMoment_blarger_split}
    \EE[\cZ^2 \1_{\cE}] &\leq (\text{I}) + (\text{II}),
\end{align}
where with the same notation as in \eqref{eq:geomMatch_UB_EZsquare_casebLessThanEqual1},
\begin{align*}
    (\text{I}) &= n! \sum_{k=0}^{k_0} \sum_{\Pi : \abs{S(\Pi)} = k} \EE \insquare{ \exp \inbraces{  \frac{1}{\sigma}  \zb^\top \insquare{(I_n + \Pi) \otimes I_{d} } \xb  } \1_{\cE_k} } \\
    (\text{II}) &= n! \sum_{k=k_0 + 1}^{n} \sum_{\Pi : \abs{S(\Pi)} = k} \EE \insquare{ \exp \inbraces{  \frac{1}{\sigma}  \zb^\top \insquare{(I_n + \Pi) \otimes I_{d} } \xb  } \1_{\cE_k} }.
\end{align*}
For term $(\text{I})$, recall from \eqref{eq:GM_highProbEvent_def} that $k_0=c_0n$. After dropping the factor $\1_{\cE_k}$, we proceed as in the derivation of \eqref{eq:GM_secondMoment_untruncated_sum_bound}, but restrict the outer sum to $0\leq k\leq k_0$. This gives
\begin{align}
    (\text{I})
    &\leq
    n!\sum_{k=0}^{k_0}\binom{n}{k}(n-k)!
    \inparen{1-\frac{4}{\sigma^2}}^{-\frac{d(n+k)}{4}}
    \nonumber\\
    &\leq
    (n!)^2 \sum_{k=0}^{k_0}
    \exp\inbraces{
        b n\log n
        +
        b k\log n
        +
        k-k\log k
        +
        o(n\log n)
    }
    \nonumber\\
    &\leq
    \exp\inbraces{\inparen{b+2+(b-1)c_0+o(1)}n\log n},
    \label{eq:GM_secondMoment_(I)}
\end{align}
where in the last line we used $k\leq k_0=c_0n$, so $bk\log n-k\log k\leq (b-1)c_0 n\log n+O(n)$.

For term $(\text{II})$, we have
\begin{align}\label{eq:GM_secondMoment_TermII}
    (\text{II}) \leq (n!)^2 \sum_{k=k_0 + 1}^{n} \frac{1}{k!} \max_{\Pi : \abs{S(\Pi)} = k} \EE\insquare{ \exp\inparen{\sqrt{b \log n} \tr(I + \Pi)\frac{\Xb \Zb^\top}{\sqrt{d}}} \1_{\cE_k} }
\end{align}
Let $I + \Pi = 2 (I \cap \Pi) + I_{\backslash} + \Pi_{\backslash}$, thus splitting into a common overlapping part where $I$ and $\Pi$ coincide, and two non-overlapping parts. More precisely, $I\cap\Pi=\{(i,i):\Pi_{ii}=1\}$, while $I_{\backslash}$ and $\Pi_{\backslash}$ denote the matched pairs of $I$ and $\Pi$, respectively, with the common part $I\cap\Pi$ removed. We write $(i,i') \in (I \cap \Pi)$ to mean that $(i,i')$ is a matched pair of indices in $I \cap \Pi$ (equivalently the corresponding matrix entry is one); similarly for $I_{\backslash}$ and $\Pi_{\backslash}$. Fix $k_0 + 1 \leq k \leq n$ and any $\Pi$ with $\abs{S(\Pi)} = k$. Define
\[
U_1 = \frac{2}{\sqrt{d}} \sum_{j=1}^{d} \sum_{(i,i')\in (I\cap \Pi)} X_{ij} Z_{i'j}, \qquad  U_2 = \frac{1}{\sqrt{d}} \sum_{j=1}^{d} \sum_{(i,i')\in I_{\backslash}} X_{ij} Z_{i'j}, \qquad U_3 = \frac{1}{\sqrt{d}} \sum_{j=1}^{d} \sum_{(i,i')\in \Pi_{\backslash}} X_{ij} Z_{i'j}.
\]

We note that $U_1$ is indeed independent of $(U_2,U_3)$. The variance of $U_1$ is $4k$. Let
\[
V = U_1/(2\sqrt{k}).
\]
Let $\lambda = 2\sqrt{bk \log n}$ and $t = \sqrt{(2+\epsilon) k \log n}$ and note that $\lambda \asymp t$. Since $\cE_k$ implies $V \leq t$, we have
\begin{equation}\label{eq:GM_secondMoment_TermIIB}
    \EE\insquare{ \exp\inparen{\sqrt{b \log n} \tr(I + \Pi)\frac{\Xb \Zb^\top}{\sqrt{d}}} \1_{\cE_k} } \leq \EE\insquare{\exp\inparen{\lambda V} \1\!\inbraces{V \leq t}} \cdot \EE\insquare{\exp\inparen{\sqrt{b \log n} (U_2 + U_3)}}.
\end{equation}

We consider the term involving $U_1$ first. In the present case, we always have $\lambda \geq t$. Indeed, since $1<b<2$ and $0<\epsilon<2$, we have
$
    \frac{\sqrt{2+\epsilon}}{2}
    <1<\sqrt b.
$

We use a ``peeling argument''. Define $L = \ceil{\log_{1+\iota} t}$, where $\iota > 0$ is a small constant. Thus $(1+\iota)^{L-1} \leq t \leq (1+\iota)^L$. We have
\begin{align}
    &\EE\insquare{\exp\inparen{\lambda V} \1\!\inbraces{V \leq t}} \nonumber \\
    &\quad \leq \EE\insquare{\exp\inparen{\lambda V} \1\!\inbraces{V \leq 1}} + \sum_{\ell = 0}^{L} \EE\insquare{\exp\inparen{\lambda V} \1\!\inbraces{(1+\iota)^{\ell} < V \leq (1+\iota)^{\ell+1}}} \nonumber \\
    &\quad\leq e^{\lambda} + \sum_{\ell=0}^{L-1} \exp\inparen{\lambda(1+\iota)^{\ell+1}} \PP\insquare{V > (1+\iota)^\ell} + \exp\inparen{ \lambda(1+\iota)^{L+1}  } \PP\insquare{ V > t  } \label{eq:geomMatch_U1_peelingArgument}
\end{align}
A sub-Gaussian bound holds for the tail probabilities. To see this, without loss of generality suppose $\inbraces{ (1,1),\dots,(k,k) } = (I \cap \Pi)$. For $s > 0$ and we have
\begin{equation}\label{eq:U1_MGF}
\EE\insquare{\exp sV} = \EE\insquare{\EE\insquare{ \exp \frac{s}{\sqrt{dk}} \sum_{j=1}^{d}\sum_{i=1}^{k} X_{ij} Z_{ij} \,\bigg|\, X }} =  \prod_{j=1}^{d}\prod_{i=1}^{k} \EE\insquare{ \exp \frac{s^2}{2dk} X_{ij}^2  } = \inparen{1 - \frac{s^2}{dk}}^{-\frac{dk}{2}},
\end{equation}
where the last equality is valid whenever $s^2/(2dk) < 1/2$. For $s \lesssim \sqrt{n \log n}$ we have $s^2 \ll dk$ so the equality is valid for large $n$. By an exponential Markov bound, for any $s \lesssim \sqrt{n \log n}$ and for any $a \leq t$,
\begin{align*}
    \PP[V > a] &\leq \exp\inbraces{ -sa - \frac{dk}{2} \log \inparen{ 1 - \frac{s^2}{dk} } } = \exp \inbraces{ -sa + \frac{s^2}{2} + O\!\inparen{  \frac{s^4}{dk}  }  } = \exp\inbraces{ -\frac{a^2}{2} + o(a^2)  },
\end{align*}
where the last line follows by setting $s = a$. Then from \eqref{eq:geomMatch_U1_peelingArgument} we have
\begin{align}
    &\EE\insquare{\exp(\lambda V)\1\{V \leq t\}}
    =
    \EE\insquare{\exp\inparen{\sqrt{b \log n} \, U_1} \1\{V \leq t\}} \nonumber\\
    &\quad \leq e^{\lambda} +\sum_{\ell=0}^{L-1}\exp\inbraces{\lambda(1+\iota)^{\ell+1}-\frac{(1+\iota)^{2\ell}}{2} + o\inparen{(1+\iota)^{2\ell}} } + \exp\inparen{\lambda(1+\iota)^{L+1} - \frac{t^2}{2} + o(t^2)} \nonumber \\
    &\quad \leq e^{\lambda} + (L+1) \exp\inbraces{ \lambda t(1+\iota)^2 - \frac{t^2}{2(1+\iota)^2} + o(t^2)  } \nonumber\\
    &\quad \leq \exp \inbraces{ \lambda t - \frac{t^2}{2} + C\iota n \log n  }  \nonumber\\
    &\quad = \exp\inbraces{ \inparen{2\sqrt{b}\sqrt{2+\epsilon} - \frac{2+\epsilon}{2}} k \log n + C\iota n \log n  }
    \label{eq:GM_U1_blarger}
\end{align}
where in the third line we used that the function $x \mapsto \lambda (1+\iota)^{x+1} - (1+\iota)^{2x}/2$ is increasing in $x$, for $0 \leq x \leq L$.

On the other hand, for the fixed $1<b<2$ under consideration, by the exact same computation as \eqref{eq:truncatedSecondMoment_(A)_and_(B)} term (B),
\begin{equation}\label{eq:GM_U2U3}
    \EE\insquare{\exp\inparen{\sqrt{b \log n} (U_2 + U_3)}} \leq \inparen{1 - \frac{4}{\sigma^2}}^{-\frac{d(n-k)}{4}} = \exp \inbraces{ b(n-k)\log n + o(n \log n)  }.
\end{equation}

Using \eqref{eq:GM_U1_blarger} and \eqref{eq:GM_U2U3}
in \eqref{eq:GM_secondMoment_TermII} and \eqref{eq:GM_secondMoment_TermIIB}, we have
\begin{align*}
    (\text{II}) &\leq (n!)^2\exp\inbraces{(b+C\iota) n \log n}
    \sum_{k=k_0+1}^{n}
    \exp\inbraces{
        \inparen{
            - b
            + 2\sqrt{b}\sqrt{2+\epsilon}
            - \frac{2+\epsilon}{2}
        }
        k \log n
        + k - k \log k
    } \\
    &\leq \exp\inbraces{(b+2+C\iota)n\log n+o(n\log n)}.
\end{align*}
Here the last inequality follows from the choice
$\gamma_{b,\epsilon}<1$ in \eqref{eq:SM_gamma_b_epsilon}, since
\[
    k \mapsto
    \gamma_{b,\epsilon} k \log n + k - k\log k
\]
is maximized at $k^*\sim n^{\gamma_{b,\epsilon}}\ll n$.

From \eqref{eq:GM_secondMoment_blarger_split}, we thus have
\[
\limsup_{n \to \infty} \frac{1}{n \log n} \log \EE[\cZ^2 \1_{\cE}] \leq b + 2 + C\iota + (b-1)c_0
\]
Choosing $c_0$ and $\iota$ sufficiently small gives \eqref{eq:GM_truncatedSecondMoment}.
\end{proof}

\subsubsection{Proof of Theorem \ref{thm:GM_noiseFE_asymptotics}}

The subcritical branch ($b < 2$) of Theorem \ref{thm:GM_noiseFE_asymptotics} will follow from Propositions \ref{prop:GM_ElogZ_UB} and \ref{prop:GM_ElogZ_LB} below, as well as the concentration of the free energy about its expectation in Lemma \ref{lemma:GM_noiseFE_concentration}. The critical and supercritical branch ($b \geq 2$) given in Proposition \ref{prop:GM_noiseFE_supercritical} below.

Our next result gives the upper bound. We remark here that this does not require the second moment machinery.

\begin{proposition}\label{prop:GM_ElogZ_UB}
    For any $b < 2$,
    \begin{equation}
        \limsup_{n \to \infty} \frac{1}{n \log n} \EE \log \cZ \leq \frac{b}{2} + 1.
    \end{equation}
\end{proposition}
\begin{proof}
    Fix any $\epsilon,c_0>0$ and let $\cE=\cE(\epsilon,c_0)$.
    Write
    \begin{align}
        \EE \log \cZ &= \EE[\log \cZ \mid \cE] \ \PP[\cE] + \EE\insquare{ \boldsymbol{1}_{\cE^c} \log \cZ  } \leq \PP[\cE] \ \log \EE[\cZ \mid \cE] + \EE[\boldsymbol{1}_{\cE^c} \log \cZ] \nonumber\\
        &= \PP[\cE] \log \EE[\cZ \boldsymbol{1}_{\cE}] - \PP[\cE] \log \PP[\cE] + \EE\insquare{ \boldsymbol{1}_{\cE^c} \log \cZ }. \label{eq:firstMoment_ElogZ_expansion}
    \end{align}
    The second term in the right-hand side of \eqref{eq:firstMoment_ElogZ_expansion} converges to zero by Lemma \ref{lemma:GM_highProbEvent}. We use Cauchy-Schwarz on the third term to have $\EE\insquare{ \boldsymbol{1}_{\cE^c} \log \cZ } \leq \sqrt{\PP[\cE^c]} \sqrt{\EE\log ^2\cZ}$. We give a bound on $\EE \log^2 \cZ$ now. 
    We first bound the variance of $\log \cZ$. Let $\lambda=\sqrt{b\log n}$ and set $p_{\Pi}=\exp\inbraces{\lambda \tr \Pi \Xb\Zb^\top/\sqrt d}/\cZ$. Then by the Gaussian Poincare inequality (see e.g.~\cite[Corollary~2.27]{vanhandel2016probability}), and direct differentiation, we obtain
    \begin{align*}
        \Var \log \cZ
        &\leq
        \EE\insquare{\norm{\nabla_{\Xb} \log\cZ}_F^2+\norm{\nabla_{\Zb} \log\cZ}_F^2} \\
        &\leq
        \frac{\lambda^2}{d}
        \EE\insquare{
        \sum_{\Pi \in \mathscr{P}_n}p_{\Pi}\norm{\Pi^\top\Zb}_F^2
        +
        \sum_{\Pi \in \mathscr{P}_n}p_{\Pi}\norm{\Pi\Xb}_F^2
        } \\
        &=
        \frac{\lambda^2}{d}\EE\insquare{\norm{\Zb}_F^2+\norm{\Xb}_F^2}
        \leq C n\log n ,
    \end{align*}
    where in the second line we used Jensen's inequality, and in the equality we used the fact that Frobenius norm is preserved under permutation, and the final inequality uses $\EE\norm{\Xb}_F^2=\EE\norm{\Zb}_F^2=nd$. We also have
    \begin{align*}
        \EE \log \cZ &\leq \EE \log \inbraces{n! \max_{\Pi \in \mathscr{P}_n} \exp \inparen{\sqrt{b n \log n} \tr \Pi \frac{\Xb \Zb^\top}{\sqrt{nd}} } } \\
        &= \log n! +  \sqrt{b n \log n} \ \EE \max_{\Pi \in \mathscr{P}_n} \tr \Pi \frac{\Xb \Zb^\top}{\sqrt{nd}}   \\
        &= (1+o(1)) C n \log n,
    \end{align*}
    where the last line is by \eqref{eq:UncondHighProb_Emax_bound}. Thus $\EE \log^2 \cZ = \Var \log \cZ + \inparen{\EE \log \cZ}^2 = O((n \log n)^2)$. Since $\PP[\cE^c] = o(1)$ by Lemma \ref{lemma:GM_highProbEvent}, it follows that $\frac{1}{n \log n}\EE\insquare{ \boldsymbol{1}_{\cE^c} \log \cZ } = o(1)$.

    The result follows from Lemmas \ref{lemma:GM_highProbEvent} and \ref{lemma:GM_logEZ1cE_UB} applied to the first term in \eqref{eq:firstMoment_ElogZ_expansion}.
\end{proof}

\begin{proposition}\label{prop:GM_ElogZ_LB}
    For any $b < 2$,
    \begin{equation}
        \liminf_{n \to \infty} \frac{1}{n \log n} \EE \log \cZ \geq \frac{b}{2} + 1.
    \end{equation}
\end{proposition}

\begin{proof}
    Fix any small constant $\eta>0$, and choose $\zeta>0$ sufficiently small. By Lemma \ref{lemma:GM_truncatedSecondMoment}, choose $\epsilon,c_0>0$ so that \eqref{eq:GM_truncatedSecondMoment} holds with $\zeta/3$ in place of $\zeta$, and let $\cE=\cE(\epsilon,c_0)$. Combining this with Lemma \ref{lemma:GM_logEZ1cE_LB}, for all large $n$,
    \[
        \frac{\EE[\cZ \boldsymbol{1}_{\cE}]^2}{\EE[\cZ^2\boldsymbol{1}_{\cE}]}
        \geq
        \exp\inparen{-\zeta n\log n}.
    \]
    Thus, for any $\epsilon_0>0$, for $n$ large enough, by the Paley-Zygmund inequality,
    \begin{align*}
        &\PP\insquare{ \frac{1}{n\log n} \log \cZ \geq \frac{1}{n\log n} \log \EE[\cZ\boldsymbol{1}_{\cE}] - \epsilon_0 \,\bigg|\, \cE  } = \PP\insquare{ \cZ \geq e^{-\epsilon_0 n \log n} \EE[\cZ\boldsymbol{1}_{\cE}] \,\Big|\, \cE } \\
        &\quad\geq  \PP\insquare{ \cZ \geq \frac{1}{2} \EE[\cZ\boldsymbol{1}_{\cE}] \,\Big|\, \cE } \geq \frac{1}{4} \frac{\EE[\cZ\boldsymbol{1}_{\cE}]^2}{\EE[\cZ^2\boldsymbol{1}_{\cE}]} \geq \exp\inbraces{-\zeta n \log n},
    \end{align*}
    From the concentration of $\log \cZ$ in Lemma \ref{lemma:GM_noiseFE_concentration}, we have for some constants $C,c> 0$,
    \begin{align*}
        &\PP\insquare{ \abs{ \frac{1}{n\log n} \log \cZ - \EE \frac{1}{n\log n} \log \cZ } \leq \eta  } + \PP\insquare{ \frac{1}{n\log n} \log \cZ \geq \frac{1}{n\log n} \log \EE[\cZ\boldsymbol{1}_{\cE}] - \epsilon_0   } \\
        &\quad \geq 1 - C\exp\inbraces{ -c\frac{\eta^2}{b} n \log n  } + \PP[\cE]\exp\inbraces{-\zeta n \log n} > 1,
    \end{align*}
    where the last strict inequality follows by choosing $\zeta < \frac{c\eta^2}{b}$. For any two events $A,B$ we have $\PP\insquare{A\cap B} \geq \PP A + \PP B - 1$. Therefore, the above display implies that the probability of the two events
    \[
    \inbraces{ \abs{ \frac{1}{n\log n} \log \cZ - \EE \frac{1}{n\log n} \log \cZ } \leq \eta  } \qquad\text{and}\qquad \inbraces{ \frac{1}{n\log n} \log \cZ \geq \frac{1}{n\log n} \log \EE[\cZ\boldsymbol{1}_{\cE}] - \epsilon_0  }
    \]
    happening simultaneously is strictly positive. It follows that we must have
    \[
    \frac{1}{n\log n} \log \EE[\cZ \boldsymbol{1}_{\cE}] - \epsilon_0 \leq \frac{1}{n\log n} \EE \log \cZ + \eta,
    \]
    otherwise the two events would be disjoint. The result follows from Lemma \ref{lemma:GM_logEZ1cE_LB} and sending $\eta$ and $\epsilon_0$ to $0$.
\end{proof}

\begin{proposition}\label{prop:GM_noiseFE_supercritical}
For any $b \geq 2$,
\[
    \lim_{n \to \infty} \frac{1}{n\log n} \log \cZ = \sqrt{2b}, \qquad \text{a.s.}.
\]
\end{proposition}

\begin{proof}
For $s > 0$, define
\[
    \cZ_n(s)
    \coloneqq
    \sum_{\Pi}
    \exp\inbraces{
    s\sqrt{\log n}\tr \Pi\frac{\Xb\Zb^\top}{\sqrt d}
    },
    \qquad
    f_n(s)
    \coloneqq
    \frac{1}{n\log n}\log \cZ_n(s).
\]
Thus $\cZ=\cZ_n(\sqrt b)$. For every $n$, the function $s\mapsto f_n(s)$ is convex. Fix $s \geq \sqrt 2$ and rational numbers $0<q<r<\sqrt 2$. Convexity gives
\[
    f_n(s)
    \geq
    f_n(r)
    +(s-r)\frac{f_n(r)-f_n(q)}{r-q}.
\]
For the countable set of positive rational $q,r<\sqrt 2$, the already established subcritical branch holds simultaneously almost surely. On this event, taking the limit inferior in the previous display gives
\[
    \liminf_{n\to\infty} f_n(s)
    \geq
    1+\frac{r^2}{2}
    +
    (s-r)\frac{r+q}{2}.
\]

Sending rational $q,r\uparrow\sqrt 2$, we obtain
\begin{equation}\label{eq:GM_noiseFE_supercritical_LB}
    \liminf_{n\to\infty}f_n(s)\geq \sqrt 2\,s,
    \qquad s\geq\sqrt 2,
    \quad\text{a.s.}
\end{equation}

For the matching upper bound, fix $s\geq\sqrt 2$ and $0<\theta\leq 1$. The subadditivity of $x\mapsto x^\theta$ on $\RR_+$ yields
\begin{align*}
    \EE \cZ_n(s)^\theta
    &\leq
    \sum_{\Pi}
    \EE\exp\inbraces{
    \theta s\sqrt{\log n}\tr \Pi\frac{\Xb\Zb^\top}{\sqrt d}
    } \\
    &=
    n!
    \inparen{1-\frac{\theta^2s^2\log n}{d}}^{-\frac{nd}{2}} \\
    &=
    \exp\inbraces{
    \inparen{1+\frac{\theta^2s^2}{2}+o(1)}n\log n
    },
\end{align*}
where the last line uses $d=\omega(\log n)$. Hence, for any fixed $\eta>0$, Markov's inequality gives
\[
    \PP\insquare{
    f_n(s)
    \geq
    \frac{1}{\theta}+\frac{\theta s^2}{2}+\eta
    }
    \leq
    \exp\inbraces{
    -\theta\eta n\log n+o(n\log n)
    }.
\]
The right-hand side is summable in $n$. By Borel-Cantelli, sending $\eta\downarrow 0$, and then taking $\theta=\sqrt 2/s$,
\begin{equation}\label{eq:GM_noiseFE_supercritical_UB}
    \limsup_{n\to\infty}f_n(s)
    \leq
    \sqrt 2\,s,
    \qquad s\geq\sqrt 2,
    \quad\text{a.s.}
\end{equation}
Combining \eqref{eq:GM_noiseFE_supercritical_LB} and \eqref{eq:GM_noiseFE_supercritical_UB}, and substituting $s=\sqrt b$, proves the supercritical branch of \eqref{eq:GM_noiseFE_asymptotics}.
\end{proof}

\subsection{Proof of Corollary \ref{cor:GM_infoTheoretic_nothing}}
\label{sec:GM_fullFE_proofs}
\label{sec:GM_infoTheoretic_nothing_proofs}

\begin{proof}[Proof of Corollary \ref{cor:GM_infoTheoretic_nothing}(i)]
    We first show that the MMSE is equal to half of the expected error in using a posterior sample as an estimator: for $\Pi \sim \inangle{\cdot}$,
    \begin{equation}
        \EE\inangle{\norm{\Pi - \Pi_*}^2_F} = 2 \EE\norm{  \inangle{\Pi} - \Pi_*  } ^2_F.
    \end{equation}
    To see this, expand the left-hand side into $\EE \inangle{ \norm{\Pi}^2_F  } + \EE\norm{\Pi_*}^2_F - 2 \EE{\tr \inangle{\Pi} \Pi_*^\top}$. Note that $\EE \inangle{ \norm{\Pi}^2_F  } = \EE\norm{\Pi_*}^2_F = n$ because the marginal distribution of $\Pi \sim \inangle{\cdot}$ is equal to that of $\Pi_*$. Furthermore, $\EE{\tr \inangle{\Pi} \Pi_*^\top} = \EE\norm{\inangle{\Pi}}^2_F$ since $\Pi$ and $\Pi_*$ are conditionally independent given $(\Xb,\Yb)$ (this is also the Nishimori identity from statistical physics \cite[Equation 15]{zdeborova2016statistical}.)

    Thus the left-hand side is $2(n - \EE\norm{\inangle{\Pi}}^2_F)$.

    On the other hand, $\EE\norm{  \inangle{\Pi} - \Pi_*  } ^2_F = \EE\norm{\inangle{\Pi}}^2_F + \EE\norm{\Pi_*}^2_F - 2 \EE\tr \inangle{\Pi}\Pi_*^\top$. The cross term is equal to $-2\EE\norm{\inangle{\Pi}}^2_F$ as before. Thus $\EE\norm{  \inangle{\Pi} - \Pi_*  } ^2_F = n - \EE\norm{\inangle{\Pi}}^2_F$.

    It remains to show that the expected overlap of a posterior sample with the truth is small. Assume without loss that $\delta \in (0,1)$. Let $B = \tr \Pi_* \Pi^\top$ and fix $\eta = \delta/2$. Since $0 \leq B \leq n$,
    \[
        \frac{B}{n} \leq \eta + \boldsymbol{1}\!\inbraces{B \geq \eta n}.
    \]
    By Theorem \ref{thm:GM_AoN}, for all sufficiently large $n$,
    $
        \EE\inangle{\boldsymbol{1}\!\inbraces{B \geq \eta n}} \leq \frac{\delta}{2}.
    $
    Taking expectation with respect to the posterior and then the observation, we obtain
    \[
        \frac{1}{n}\EE\inangle{\tr \Pi_* \Pi^\top} \leq \eta + \frac{\delta}{2} = \delta.
    \]
    Therefore
    $
        \EE\inangle{\norm{\Pi - \Pi_*}^2_F}
        = 2n - 2\EE\inangle{\tr \Pi_* \Pi^\top}
        \geq 2(1-\delta)n.
    $
    By the identity above, for all sufficiently large $n$,
    $
        \MMSE(b) \geq (1-\delta)n.
    $
    Since $\delta>0$ is arbitrary, $\liminf_n \MMSE(b)/n\geq 1$. The dummy estimator $\EE\Pi_*=\frac1n\boldsymbol{1}\boldsymbol{1}^{\top}$ gives $\MMSE(b)\leq n-1$, and hence $\MMSE(b)/n\to 1$.
\end{proof}

\begin{proof}[Proof of Corollary \ref{cor:GM_infoTheoretic_nothing}(ii)]
    Put $\widetilde{\Yb}=\Yb/\sigma=\sqrt{b\log n/d}\,\Pi_*\Xb+\Zb$. Define
    \[
        i_n(b)
        =
        \frac{1}{n\log n}I(\Pi_*;\widetilde{\Yb}\mid \Xb).
    \]
    Expanding the Gaussian likelihood and using $\EE\norm{\Xb}_F^2=nd$ gives
    \[
        i_n(b)
        =
        \frac{\log n!}{n\log n}
        +
        b
        -
        \EE\widetilde{\Psi}_{\sigma}.
    \]
    By Proposition \ref{prop:GM_full_free_energy_subcritical}, $i_n(b)\to b/2$ for every fixed $0<b<2$.

    For each $n$, $i_n$ is concave and differentiable in the SNR parameter $b$, and the Gaussian I-MMSE identity \cite{guo2005mutual} gives
    \[
        i_n'(b)
        =
        \frac{1}{2nd}\YMMSE(b).
    \]
    Since the pointwise limit $b/2$ is differentiable on $(0,2)$, the standard derivative-convergence property for concave functions yields $i_n'(b)\to1/2$ at every fixed $0<b<2$. The claim follows.
\end{proof}

\section{Proofs for the multi-view positive result}
\label{sec:Kway_positive_AER_proofs}

The starting point of our analysis is a reduction: we may assume without loss that
\begin{equation}\label{eq:Kway_identity_assumption}
    \Pi_*^{(a)}=I,
    \qquad a=1,\dots,K-1.
\end{equation}
Indeed, conditionally on $\bar{\Pi}_*^{(0)},\dots,\bar{\Pi}_*^{(K-1)}$, left multiplying the $a$-th view by $(\bar{\Pi}_*^{(a)})^\top$ is, by Gaussian orthogonal invariance, only relabeling the observed rows of $\Yb^{(a)}$.
The algorithms below are equivariant under such row relabelings, so their overlap with the relative matchings has the same distribution as their overlap with the identity in the relabeled model. In what follows, we will always assume that \eqref{eq:Kway_identity_assumption} is in force.

\begin{remark}\label{remark:Kview_TwoView_Mapping}
The following is a justification that under the mapping \eqref{eq:multi_twoView_Mapping}, i.e.~$\sigma^2 \sim \tau^4$,
the model \eqref{eq:KwayMatching_def} is asymptotically equivalent to a $K$-view extension of the setting \eqref{eq:GM_def}, which would be a $2$-view. Indeed, in setting \eqref{eq:GM_def}, after standardizing, and with the analogous reduction to $\Pi_*=I$, we observe
\begin{align*}
    \Xb \qquad\text{and}\qquad \frac{\Yb}{\sqrt{1+\sigma^2}} = \frac{\Xb}{\sqrt{1 + \sigma^2}} + \frac{\sigma}{\sqrt{1 + \sigma^2}} \Zb.
\end{align*}
On the other hand, under \eqref{eq:Kway_identity_assumption}, in \eqref{eq:KwayMatching_def} we equivalently observe the standardized
\begin{align*}
    \frac{\Yb^{(0)}}{\sqrt{1+\tau^2}} &= \frac{\Xb}{\sqrt{1+\tau^2}} + \frac{\tau}{\sqrt{1+\tau^2}} \Wb^{(0)} \qquad\text{and}\qquad \frac{\Yb^{(1)}}{\sqrt{1+\tau^2}} &= \frac{\Xb}{\sqrt{1+\tau^2}} + \frac{\tau}{\sqrt{1+\tau^2}} \Wb^{(1)}.
\end{align*}
Let $\widetilde{\Yb} = \Yb/\sqrt{1 + \sigma^2}$ and $\widetilde{\Yb}^{(a)} = \Yb^{(a)}/\sqrt{1+\tau^2}$. The pairs $\inparen{\Xb,\widetilde{\Yb}}$ and $\inparen{\widetilde{\Yb}^{(0)},
\widetilde{\Yb}^{(1)}}$ each have standard multivariate Gaussian marginals. The unmatched rows among each pair have zero covariance. The matched rows among the first pair have covariance $\Cov\big[ X_i, \widetilde{Y}_i  \big] = \frac{1}{\sqrt{1+\sigma^2}}I_d$, and among the second pair this is $\Cov\big[\widetilde{Y}^{(0)}_i, \widetilde{Y}^{(1)}_i\big] = \frac{1}{1+\tau^2}I_d$. Thus their laws agree when $\sigma^2=2\tau^2+\tau^4 \sim \tau^4$.
\end{remark}

Notice that the score $T_{\ib}$ corresponding to the true tuple $(i,\dots,i)$ has mean $C_K d$, where we denote
\[
    C_K=\binom{K}{2}.
\]

Evidently, Algorithm \ref{alg:Kway_general_rowwise_thresholding} succeeds when the score $T_{i,\dots,i}$ defined in \eqref{eq:KwayMatching_Ti_def} exceeds all other scores. We now record a moderate-deviation estimate for a general score $T_{\ib}$ that will be used in the proof. Consider a candidate tuple $\mathbf{i}=(i_0,\dots,i_{K-1})\in[n]^K$. Let $u_0(\mathbf{i})=i_0$, and write
\[
    \inbraces{u_1(\mathbf{i}),\dots,u_{r(\mathbf{i})}(\mathbf{i})}
    =
    \inbraces{i_0,\dots,i_{K-1}}\setminus\inbraces{i_0}
\]
for the set of distinct selected labels different from $i_0$. For $s=0,\dots,r(\mathbf{i})$, define
\begin{equation}\label{eq:Kway_multiplicity_notation}
    S(\mathbf{i})=\sum_{s=0}^{r(\mathbf{i})}\binom{\ell_s(\mathbf{i})}{2},
    \qquad\text{where}\qquad
    \ell_s(\mathbf{i})=\abs{\inbraces{a\in\inbraces{0,\dots,K-1}:i_a=u_s(\mathbf{i})}}.
\end{equation}
Thus $S(\mathbf{i})$ counts the pairs of views in the tuple $\mathbf{i}$ that select the same row label. Under \eqref{eq:Kway_identity_assumption}, it is exactly these pairs which contribute mean $d$ to $T_{\mathbf{i}}$, while pairs with different labels have mean zero; hence $\EE T_{\mathbf{i}}=dS(\mathbf{i})$.

\begin{lemma}\label{lemma:T_i_concentrates_on_mean}
For every fixed $\eta>0$, uniformly over $\mathbf{i}\in[n]^K$,
\begin{equation}\label{eq:Kway_score_moderate_deviation}
    \PP\insquare{T_{\mathbf{i}}\geq dS(\mathbf{i})+\eta d}
    \leq
    \exp\inbraces{
    -\inparen{\frac{\eta^2}{2C_K}+o(1)}
    \frac{d}{\tau^4}
    }.
\end{equation}
\end{lemma}

\begin{proof}
Fix a tuple $\mathbf{i}\in[n]^K$. 
For each $h=1,\dots,d$, set
\[
    G_h=\inparen{Y_{i_0,h}^{(0)},\dots,Y_{i_{K-1},h}^{(K-1)}}\in\RR^K.
\]
Then $G_1,\dots,G_d$ are i.i.d.~centered Gaussian vectors. Let $A \in \RR^{K \times K}$ be the symmetric matrix which has zeros on the diagonal and $1/2$ in every other entry. Then we may express $T_{\mathbf{i}}$ as
\[
    T_{\mathbf{i}}=\sum_{h=1}^{d}G_h^\top A G_h.
\]
We analyze one summand. Write $G$ for a generic copy of $G_h$. Its covariance matrix has the form
\[
    \Sigma=\tau^2 I_K+R,
\]
where $R_{ab}=1$ if $i_a=i_b$, and $R_{ab}=0$ otherwise. In particular, $\norm{R}_{\mathrm{op}}=O_K(1)$ and $\norm{\Sigma}_{\mathrm{op}}=O_K(\tau^2)$, uniformly over $\mathbf{i}$. Thus one coordinate of the score is
\[
    U=G^\top A G,
    \qquad
    \EE U=\tr(A\Sigma)=S(\mathbf{i}).
\]
For a centered Gaussian quadratic form,
$
    \Var(U)=2\tr\inparen{(A\Sigma)^2}.
$
Using $\Sigma=\tau^2 I_K+R$ and $2\tr(A^2)=C_K$, we obtain that uniformly over $\mathbf{i}$,
\[
    \Var(U)
    =
    2\tau^4\tr(A^2)+O_K(\tau^2)
    =
    (C_K+o(1))\tau^4.
\]

We now bound the centered moment generating function. Put $B=\Sigma^{1/2}A\Sigma^{1/2}$, and let $\theta_1,\dots,\theta_K$ be its eigenvalues. Since $\norm{B}_{\mathrm{op}}=O_K(\tau^2)$, for $\lambda=O(\tau^{-4})$ and all sufficiently large $\tau$ we have $\abs{2\lambda\theta_m}\leq 1/2$. Writing $G=\Sigma^{1/2}g$ with $g\sim\cN(0,I_K)$, diagonalizing $B$, and using the same chi-squared moment generating function as in \eqref{eq:geomMatch_overlapError_V_in_EigvalCirculant} gives 
\[
    \log \EE e^{\lambda(U-\EE U)}
    =
    -\lambda\tr B-\frac{1}{2}\log\det(I_K-2\lambda B)
    =
    \sum_{m=1}^{K}
    \inparen{
    -\lambda\theta_m-\frac{1}{2}\log(1-2\lambda\theta_m)
    }.
\]
Expanding the logarithm uniformly in $m$,
\[
    \log \EE e^{\lambda(U-\EE U)}
    =
    \lambda^2\sum_{m=1}^{K}\theta_m^2
    +O_K\inparen{\lambda^3\norm{B}_{\mathrm{op}}^3}
    =
    \frac{\Var(U)}{2}\lambda^2+O_K(\tau^6\lambda^3).
\]
Applying exponential Markov's inequality to $T_{\mathbf{i}}-dS(\mathbf{i})=\sum_{h=1}^{d}(U_h-\EE U_h)$, where the $U_h$'s are independent copies of $U$,  with $\lambda=\eta/(C_K\tau^4)$, we obtain
\[
    \log \PP\insquare{T_{\mathbf{i}}\geq dS(\mathbf{i})+\eta d}
    \leq
    -\lambda \eta d
    +d\inbraces{\frac{\Var(U)}{2}\lambda^2+O_K(\tau^6\lambda^3)}
    =
    -\inparen{\frac{\eta^2}{2C_K}+o(1)}\frac{d}{\tau^4}.
\]
This proves \eqref{eq:Kway_score_moderate_deviation} and finishes the proof.
\end{proof}

\begin{proof}[Proof of Theorem \ref{thm:Kway_positive_AER}]
Recall that it suffices to prove the result under \eqref{eq:Kway_identity_assumption}. Define
\[
    \cB
    =
    \inbraces{
    i\in[n]:
    \inparen{j_1^\star(i),\dots,j_{K-1}^\star(i)}
    \neq
    (i,\dots,i)
    }.
\]
Thus $\cB$ is the set of rows $i$ in view $0$ for which the row-wise maximization step in Algorithm \ref{alg:Kway_general_rowwise_thresholding} made at least one mistake among the $K-1$ other views. For each $i$, define the event
\[
    \cE_i
    =
    \inbraces{
    T_{i,\dots,i}
    \leq
    \max_{\substack{(j_1,\dots,j_{K-1})\in[n]^{K-1}\\(j_1,\dots,j_{K-1})\neq(i,\dots,i)}}
    T_{i,j_1,\dots,j_{K-1}}
    }.
\]
If $i\in\cB$, then a non-identity tuple was selected by the row-wise maximization, and hence $\cE_i$ occurs. Thus, by linearity of expectation and symmetry,
$
    \EE\abs{\cB}\leq n\PP(\cE_1),
$
and we claim that it suffices to show
\begin{equation}\label{eq:KwayMatching_PE1_goal}
\PP\cE_1=o(1).
\end{equation}

Indeed, we would have $\EE\abs{\cB}=o(n)$. Markov's inequality then gives $\abs{\cB}=o_{\PP}(n)$. Fix a view $a$ and define $G_a=\{i:j_a^\star(i)=i\}$. Thus $G_a$ is the set of rows of view $0$ whose row-wise choice in view $a$ is correct. Since $[n]\setminus\cB\subseteq G_a$, we have $\abs{G_a}\geq n-\abs{\cB}$. The rows in $G_a$ propose the correct diagonal matches $(i,i)$ for $\widehat{\Pi}^{(a)}$. In the second stage of Algorithm \ref{alg:Kway_general_rowwise_thresholding}, such a proposed diagonal match can fail to be kept only if some row $i'\in[n]\setminus G_a$ also proposes row $i$ in view $a$, i.e.~$j_a^\star(i')=i$. Hence at most $n-\abs{G_a}$ of these proposed diagonal matches can fail to be kept, and therefore
\[
    \tr\widehat{\Pi}^{(a)}
    \geq
    \abs{G_a}-(n-\abs{G_a})
    =
    2\abs{G_a}-n
    \geq
    n-2\abs{\cB}.
\]
This yields $\tr \widehat{\Pi}^{(a)} / n \to 1$ in probability, and hence almost exact recovery \eqref{eq:KwayMatching_AER_statement} under \eqref{eq:Kway_identity_assumption}.

We now prove \eqref{eq:KwayMatching_PE1_goal}. Recall that $C_K=\binom{K}{2}$, and fix a small $\epsilon>0$, to be chosen below. If $\cE_1$ occurs and $T_{1,\dots,1}>(C_K-\epsilon)d$, then some non-identity tuple has score at least $(C_K-\epsilon)d$. Hence
\begin{equation}\label{eq:Kway_E1_decomposition}
    \PP(\cE_1)
    \leq
    \PP\insquare{T_{1,\dots,1}\leq (C_K-\epsilon)d}
    +
    \PP\insquare{
    \max_{\substack{\mathbf{i}=(1,i_1,\dots,i_{K-1})\\\mathbf{i}\neq(1,\dots,1)}}
    T_{\mathbf{i}}\geq (C_K-\epsilon)d
    }.
\end{equation}

We first control the first term in \eqref{eq:Kway_E1_decomposition}. The true score satisfies
\[
    \EE T_{1,\dots,1}=C_K d,
    \qquad
    \Var(T_{1,\dots,1})=O_K(d\tau^4).
\]
Since $d/\tau^4\gtrsim \log n$, Chebyshev's inequality gives
\begin{equation}\label{eq:Kway_true_score_LB}
    \PP\insquare{T_{1,\dots,1}> (C_K-\epsilon)d}=1 - o(1).
\end{equation}

We next control the second term in \eqref{eq:Kway_E1_decomposition}. For any non-identity $\mathbf{i}=(1,i_1,\dots,i_{K-1})$, we have $\EE T_{\mathbf{i}}=dS(\mathbf{i})$, while the true mean is $C_Kd$. Define the gap to be
\[
D(\mathbf{i})=C_K-S(\mathbf{i}).
\]
For this $\mathbf{i}$, the event
$
    \inbraces{T_{\mathbf{i}}\geq (C_K-\epsilon)d}
$
requires a deviation of size $(D(\mathbf{i})-\epsilon)d$ above its mean. Lemma \ref{lemma:T_i_concentrates_on_mean} controls this as
\[
    \PP\insquare{T_{\mathbf{i}}\geq (C_K-\epsilon)d}
    \leq
    \exp\inbraces{
    -\inparen{\frac{(D(\mathbf{i})-\epsilon)^2}{2C_K}+o(1)}
    \frac{d}{\tau^4}
    }.
\]
We now sum this bound over all non-identity tuples anchored at $1$. Suppose a tuple uses $r$ distinct row labels other than the anchor label $1$, and fix the multiplicities $\boldsymbol{\ell}=(\ell_0,\dots,\ell_r)$ with which the $K$ views select these $r+1$ labels, as in \eqref{eq:Kway_multiplicity_notation}. Write
\[
    D_{\boldsymbol{\ell}}
    =
    C_K-\sum_{s=0}^{r}\binom{\ell_s}{2}.
\]
For every tuple $\mathbf{i}$ with these multiplicities, $D(\mathbf{i})=D_{\boldsymbol{\ell}}$, and there are at most $C'_K n^r$ such tuples.
Taking a union bound over $r$ and over the finitely many possible multiplicities $\boldsymbol{\ell}$ gives
\begin{equation}\label{eq:Kway_wrong_tuple_union_bound}
    \PP\insquare{
    \max_{\substack{\mathbf{i}=(1,i_1,\dots,i_{K-1})\\\mathbf{i}\neq(1,\dots,1)}}
    T_{\mathbf{i}}\geq (C_K-\epsilon)d
    }
    \leq
    \sum_{r=1}^{K-1}
    C'_K \cdot n^r
    \max_{\substack{\ell_0+\cdots+\ell_r=K\\ \ell_s\geq 1\ \forall\, 0\leq s\leq r}}
    \exp\inbraces{
    -\inparen{\frac{(D_{\boldsymbol{\ell}}-\epsilon)^2}{2C_K}+o(1)}
    \frac{d}{\tau^4}
    }.
\end{equation}

It remains to verify that the right-hand side of \eqref{eq:Kway_wrong_tuple_union_bound} is $o(1)$. The most dangerous terms in the union bound arise from multiplicities making $D_{\boldsymbol{\ell}}$ small relative to the number $r$ of free labels. Indeed, the factor $n^r$ counts the choices of these labels, while the deviation bound improves as the mean gap $D_{\boldsymbol{\ell}}d$ grows. To minimize $D_{\boldsymbol{\ell}}$, we equivalently maximize $\sum_s\binom{\ell_s}{2}$ over the constraints. If two multiplicities satisfy $1<\ell_s\leq \ell_t$, then replacing $(\ell_s,\ell_t)$ by $(\ell_s-1,\ell_t+1)$ changes this sum by
\[
    \binom{\ell_s-1}{2}+\binom{\ell_t+1}{2}
    -\binom{\ell_s}{2}-\binom{\ell_t}{2}
    =
    \ell_t-\ell_s+1
    \geq 1.
\]
Thus the sum is maximized by concentrating any excess mass in one multiplicity. Hence, for fixed $r$, $D_{\boldsymbol{\ell}}$ is minimized by a permutation of $(K-r,1,\dots,1)$, and
\[
    D_{\boldsymbol{\ell}}\geq
    C_K-\binom{K-r}{2}
    =
    \frac{r(2K-r-1)}{2}.
\]
Consequently, 
\[
    \frac{D_{\boldsymbol{\ell}}^2}{2C_Kr}
    \geq
    \frac{r(2K-r-1)^2}{4K(K-1)}.
\]
The latter is at least $(K-1)/K$ exactly when
\[
    r(2K-r-1)^2\geq 4(K-1)^2,
\]
which we now verify. Viewing the left-hand side as a function of real $r$, its derivative is $(2K-r-1)(2K-3r-1)$, so its minimum over $1\leq r\leq K-1$ is attained at an endpoint. The endpoint values are $4(K-1)^2$ and $K^2(K-1)$, and $K^2(K-1)\geq 4(K-1)^2$. This implies that
\[
    \frac{D_{\boldsymbol{\ell}}^2}{2C_Kr}
    \geq
    \frac{K-1}{K},
\]
with equality only at the one-outlier tuples with multiplicities $(K-1,1)$.

We now apply this deterministic lower bound to the exponent in \eqref{eq:Kway_wrong_tuple_union_bound}. Set
\[
    \alpha_n=\frac{\tau^4\log n}{d}.
\]
The assumption $\tau^4=d/(b\log n)$ gives
$
    \alpha_n
    =
    \frac{1}{b}.
$
Thus the condition $b>K/(K-1)$ is exactly the condition that $\alpha_n<(K-1)/K$. Since $d/\tau^4=\alpha_n^{-1}\log n$, a summand in \eqref{eq:Kway_wrong_tuple_union_bound} with $r$ free labels is, up to the constant $C'_K$,
\[
    \exp\inbraces{
    \inparen{
    r
    -
    \frac{(D_{\boldsymbol{\ell}}-\epsilon)^2}{2C_K\alpha_n}
    +o(1)
    }
    \log n
    }.
\]
The lower bound above gives $D_{\boldsymbol{\ell}}^2/(2C_K)\geq r(K-1)/K$. Combining this with $\alpha_n<(K-1)/K$, and using that there are only finitely many multiplicity patterns, we can choose $\epsilon>0$ sufficiently small so that, uniformly over all such choices,
\[
    \liminf_{n\to\infty}
    \frac{(D_{\boldsymbol{\ell}}-\epsilon)^2}{2C_K\alpha_n}
    >
    r.
\]
Using $d/\tau^4=\alpha_n^{-1}\log n$, each summand in \eqref{eq:Kway_wrong_tuple_union_bound} is then $o(1)$, and the whole right-hand side is $o(1)$.

Combining this with \eqref{eq:Kway_true_score_LB} and \eqref{eq:Kway_E1_decomposition}, we get $\PP\cE_1=o(1)$. This finishes the proof.
\end{proof}

\subsection*{AI use declaration}

ChatGPT 5.5 assisted with consistency checks. The authors assume responsibility for all content.

\subsection*{Acknowledgments}

ZF was supported in part by NSF DMS-2142476 and a Sloan Research Fellowship. CM was supported in part by NSF grant DMS-2338062. Part of this research was carried out while TW and CM were visiting the
Institute for Mathematical and Statistical Innovation (IMSI), which is supported by the NSF under Grant No.~DMS-2425650.

\bibliography{testing}

@misc{vanhandel2016probability,
  title={Probability in High Dimension},
  author={van Handel, Ramon},
  note={APC 550 lecture notes},
  year={2016},
  url={https://web.math.princeton.edu/~rvan/APC550.pdf}
}

@article{moharrami2021planted,
  title={The planted matching problem: Phase transitions and exact results},
  author={Moharrami, Mehrdad and Moore, Cristopher and Xu, Jiaming},
  journal={The Annals of Applied Probability},
  volume={31},
  number={6},
  pages={2663--2720},
  year={2021},
  publisher={Institute of Mathematical Statistics}
}

@article{ding2023planted,
  title={The planted matching problem: Sharp threshold and infinite-order phase transition},
  author={Ding, Jian and Wu, Yihong and Xu, Jiaming and Yang, Dana},
  journal={Probability Theory and Related Fields},
  volume={187},
  number={1-2},
  pages={1--71},
  year={2023},
  publisher={Springer}
}

@article{wee2025cluster,
  title={Cluster expansion of the log-likelihood ratio: Optimal detection of planted matchings},
  author={Wee, Timothy L. H. and Mao, Cheng},
  journal={arXiv preprint arXiv:2512.14567},
  year={2025}
}

@article{addarioberry2026statistical,
  title={The statistical threshold for planted matchings and spanning trees},
  author={Addario-Berry, Louigi and Angel, Omer and Lugosi, G{\'a}bor and R{\'a}cz, Mikl{\'o}s Z. and Schramm, Tselil},
  journal={arXiv preprint arXiv:2602.07669},
  year={2026}
}

@article{even2025statistical,
  title={Statistical-computational gap in multiple Gaussian graph alignment},
  author={Even, Bertrand and Ganassali, Luca},
  journal={arXiv preprint arXiv:2512.00610},
  year={2025}
}

@article{dai2023gaussian,
  title={Gaussian database alignment and gaussian planted matching},
  author={Dai, Osman Emre and Cullina, Daniel and Kiyavash, Negar},
  journal={arXiv preprint arXiv:2307.02459},
  year={2023}
}

@article{chertkov2010inference,
  title={Inference in particle tracking experiments by passing messages between images},
  author={Chertkov, Michael and Kroc, Lukas and Krzakala, F and Vergassola, M and Zdeborov{\'a}, L},
  journal={Proceedings of the National Academy of Sciences},
  volume={107},
  number={17},
  pages={7663--7668},
  year={2010},
  publisher={National Academy of Sciences}
}

@article{schwengber2024geometric,
  title={Geometric planted matchings beyond the Gaussian model},
  author={da Rocha Schwengber, Lucas and Oliveira, Roberto Imbuzeiro},
  journal={arXiv preprint arXiv:2403.17469},
  year={2024}
}

@article{fan2026bayesian,
  title={Bayesian inference of planted matchings: Local posterior approximation and infinite-volume limit},
  author={Fan, Zhou and Wee, Timothy L. H. and Yang, Kaylee Y.},
  journal={arXiv preprint arXiv:2603.08542},
  year={2026}
}

@article{haghverdi2018batch,
  title={Batch effects in single-cell RNA-sequencing data are corrected by matching mutual nearest neighbors},
  author={Haghverdi, Laleh and Lun, Aaron T. L. and Morgan, Michael D. and Marioni, John C.},
  journal={Nature Biotechnology},
  volume={36},
  number={5},
  pages={421--427},
  year={2018},
  publisher={Nature Publishing Group}
}

@article{demetci2022scot,
  title={SCOT: single-cell multi-omics alignment with optimal transport},
  author={Demetci, Pinar and Santorella, Rebecca and Sandstede, Bj{\"o}rn and Noble, William Stafford and Singh, Ritambhara},
  journal={Journal of Computational Biology},
  volume={29},
  number={1},
  pages={3--18},
  year={2022},
  publisher={Mary Ann Liebert, Inc.}
}

@article{ma2021image,
  title={Image matching from handcrafted to deep features: A survey},
  author={Ma, Jiayi and Jiang, Xingyu and Fan, Aoxiang and Jiang, Junjun and Yan, Junchi},
  journal={International Journal of Computer Vision},
  volume={129},
  number={1},
  pages={23--79},
  year={2021},
  publisher={Springer}
}

@article{sayers2016probabilistic,
  title={Probabilistic record linkage},
  author={Sayers, Adrian and Ben-Shlomo, Yoav and Blom, Ashley W. and Steele, Fiona},
  journal={International Journal of Epidemiology},
  volume={45},
  number={3},
  pages={954--964},
  year={2016},
  publisher={Oxford University Press}
}

@book{christen2012data,
  title={Data Matching: Concepts and Techniques for Record Linkage, Entity Resolution, and Duplicate Detection},
  author={Christen, Peter},
  series={Data-Centric Systems and Applications},
  year={2012},
  publisher={Springer},
  doi={10.1007/978-3-642-31164-2}
}

@article{pomerleau2015review,
  title={A review of point cloud registration algorithms for mobile robotics},
  author={Pomerleau, Fran{\c{c}}ois and Colas, Francis and Siegwart, Roland},
  journal={Foundations and Trends in Robotics},
  volume={4},
  number={1},
  pages={1--104},
  year={2015},
  publisher={Now Publishers},
  doi={10.1561/2300000035}
}

@article{bendory2020single,
  title={Single-particle cryo-electron microscopy: Mathematical theory, computational challenges, and opportunities},
  author={Bendory, Tamir and Bartesaghi, Alberto and Singer, Amit},
  journal={IEEE Signal Processing Magazine},
  volume={37},
  number={2},
  pages={58--76},
  year={2020},
  publisher={IEEE},
  doi={10.1109/MSP.2019.2957822}
}

@inproceedings{bayati2009algorithms,
  title={Algorithms for large, sparse network alignment problems},
  author={Bayati, Mohsen and Gerritsen, Margot and Gleich, David F. and Saberi, Amin and Wang, Ying},
  booktitle={2009 Ninth IEEE International Conference on Data Mining},
  pages={705--710},
  year={2009},
  organization={IEEE}
}

@article{semerjian2020recovery,
  title={Recovery thresholds in the sparse planted matching problem},
  author={Semerjian, Guilhem and Sicuro, Gabriele and Zdeborov{\'a}, Lenka},
  journal={Physical Review E},
  volume={102},
  number={2},
  pages={022304},
  year={2020},
  publisher={APS}
}

@inproceedings{kunisky2022strong,
  title={Strong recovery of geometric planted matchings},
  author={Kunisky, Dmitriy and Niles-Weed, Jonathan},
  booktitle={Proceedings of the 2022 Annual ACM-SIAM Symposium on Discrete Algorithms (SODA)},
  pages={834--876},
  year={2022},
  organization={SIAM}
}

@book{talagrand2010mean,
  title={Mean field models for spin glasses: Volume I: Basic examples},
  author={Talagrand, Michel},
  volume={54},
  year={2010},
  publisher={Springer Science \& Business Media}
}

@book{vershynin2018,
  title={High-dimensional probability: An introduction with applications in data science},
  author={Vershynin, Roman},
  volume={47},
  year={2018},
  publisher={Cambridge university press}
}

@article{laurent2000adaptive,
  title={Adaptive estimation of a quadratic functional by model selection},
  author={Laurent, Beatrice and Massart, Pascal},
  journal={The Annals of Statistics},
  volume={28},
  number={5},
  pages={1302--1338},
  year={2000},
  doi={10.1214/aos/1015957395}
}

@article{chen2022onewayLowRankMatching,
  title={One-way matching of datasets with low rank signals},
  author={Chen, Shuxiao and Jiang, Sizun and Ma, Zongming and Nolan, Garry P. and Zhu, Bokai},
  journal={arXiv preprint arXiv:2204.13858},
  year={2022}
}

@article{zdeborova2016statistical,
  title={Statistical physics of inference: Thresholds and algorithms},
  author={Zdeborov{\'a}, Lenka and Krzakala, Florent},
  journal={Advances in Physics},
  volume={65},
  number={5},
  pages={453--552},
  year={2016},
  publisher={Taylor \& Francis},
  doi={10.1080/00018732.2016.1211393}
}

@article{guo2005mutual,
  title={Mutual information and minimum mean-square error in Gaussian channels},
  author={Guo, Dongning and Shamai, Shlomo and Verd{\'u}, Sergio},
  journal={IEEE Transactions on Information Theory},
  volume={51},
  number={4},
  pages={1261--1282},
  year={2005},
  publisher={IEEE},
  doi={10.1109/TIT.2005.844072}
}

@inproceedings{wang2022random,
  title={Random graph matching in geometric models: the case of complete graphs},
  author={Wang, Haoyu and Wu, Yihong and Xu, Jiaming and Yolou, Israel},
  booktitle={Proceedings of Thirty Fifth Conference on Learning Theory},
  series={Proceedings of Machine Learning Research},
  volume={178},
  pages={3441--3488},
  year={2022},
  publisher={PMLR}
}

@article{hou2026recovery,
  title={Recovery thresholds for hidden weighted sparse graphs},
  author={Hou, Zhe and Liu, Jingcheng},
  journal={arXiv preprint arXiv:2606.14335},
  year={2026}
}
\bibliographystyle{alpha}

\end{document}